\documentclass[12pt]{book}

\usepackage{mathrsfs,mathtools,latexsym}
\usepackage{amscd,amsfonts,amssymb,amsmath}
\usepackage{graphicx,graphics,color}
\usepackage{geometry,setspace,parskip}
\usepackage[bookmarksnumbered, colorlinks, plainpages]{hyperref}
\hypersetup{colorlinks=true,linkcolor=red,anchorcolor=green,citecolor=cyan,urlcolor=red,filecolor=magenta,pdftoolbar=true}
\usepackage{makeidx}
\usepackage{multirow}
\usepackage{eurosym}
\usepackage{fancyhdr}
\usepackage{tocbibind}
\usepackage[decmulti]{inputenc}
\usepackage[active]{srcltx}
\usepackage[english]{babel}
\usepackage{enumerate} 
\usepackage{url}
\usepackage{pifont}
\usepackage{epstopdf}
\usepackage[T1]{fontenc}
\usepackage[dvips]{psfrag}
\usepackage{float}
\usepackage{xcolor}
\usepackage{paxe}

\begin{document}

\title{Vector-Valued Banach Limits and Vector-Valued Almost Convergence}
\author{F. J. Garc\'ia-Pacheco and F. J. P\'erez-Fern\'andez}
\date{}
\maketitle

\tableofcontents

\frontmatter

\chapter{Preface}

This book was intended to be Rafael Armario's Ph.D. thesis dissertation for obtaining his doctoral degree in Pure Mathematics at the University of C\'adiz (Spain, EU). Unfortunately he passed away on January $28^{th}$, 2013, right after he started writing this manuscript. His two Ph.D. advisors at the time of his death (the authors of this book) decided then to finish his work in book format and this is how this manuscript was given birth.

\section{Rafael Armario's life}

Rafael Armario, better known as ``Rafa'' by his friends, was a joyful and warm person, a beloved son and friend, a very hard-working graduate student, an excellent high-school teacher, and an extremely brilliant mathematician.

\subsection{Childhood and early life}

Rafa was born and raised in Cadiz, the oldest European city, within a humble middle-class family. He quickly developed a strong interest in the Cadiz culture and folklore, in particular, in the worldwide famous Cadiz carnival. Other local feasts such as the El Puerto fair also grew on him in a very profound way.

\subsection{University of Cadiz undergraduate period}

Rafa decided to register as a math major in the university of Cadiz after graduating from his neighborhood high school. He met a lot of friends there among students and professors, including the one who left the strongest imprint on him, Prof. Antonio Aizpuru, who eventually became his Ph.D. advisor. He also met a bunch of friends a buddies like Fernandito, Mari, Sole, Marina, Marimar, el Gafas, etc.

\section{Rafael Armario's work}

During his working period as a top-level mathematician, Rafa accomplished many objectives and obtained many different results. Indeed, all the theorems, definitions, examples, etc., which are anonymous in this book (including those of the framework) are due to his joint work with Prof. Aizpuru and the authors of this manuscript.

\subsection{University of Cadiz graduate period}

Rafa's graduate period was undoubtedly conditioned by the death of his Ph.D. advisor at the time, Prof. Antonio Aizpuru, who passed away in March 2008 after Rafa and him published their first paper (see \cite{AizArPer}), which was also co-authored by his Ph.D. co-advisor at the time, Prof. Francisco Javier P\'erez-Fern\'andez.

In September 2010, Prof. Garc\'ia-Pacheco, better known as Paquito, was hired at the level of Assistant Professor by the Mathematics Department of the University of Cadiz after spending nine semesters (three American academic years) at Kent State University (Ohio, USA) and ten semesters at Texas A\&M University (Texas, USA). One of the first assignments of Prof. Garc\'ia-Pacheco was to co-advice Rafael Armario together with Prof. P\'erez-Fern\'andez, continuing this way the work once started by Prof. Antonio Aizpuru. The fruit of this co-advising was the co-authorship of the papers \cite{AAGPPF, AAGPPF2, AGPPF, AGPPF2, GPPF}. The present book consists of those papers.

Many anecdotes occurred during his graduate school period. For instance, in one occasion, Rafa and Paquito shared a hotel room, due to a lack of research funding, while attending a conference held in the university of Almeria (Spain, EU). The hotel workers will never forget that week.

\mainmatter
\pagestyle{fancy}
\chapter{Framework}

All the vector spaces treated in this book will be considered to be REAL and NON-ZERO (although most of the results also work for complex spaces and the zero vector space). The general notation employed for the vector spaces will be $X,Y,Z,\dots$ (including the topological and the normed vector spaces) and $K$ for the Hausdorff compact topological spaces.

Throughout the mainmatter chapters (the first five after the framework) we will be working exclusively with topological vector spaces and normed spaces, which will not be assumed complete unless explicitly stated. In the backmatter chapter, as the title says, $G$-spaces will be the objects to deal with.

Particularly throughout the whole of this second frontmatter chapter (the framework), $X$ will stand for a topological vector space or a normed space, unless otherwise explicitly stated, except for the previous section to the very last one of this framework in which we will deal with pre-ordered spaces and nets and filters in sets and in topological spaces.

The reader should not consider this frontmatter framework as the regular framework of any book. The reason for this is that we have accomplished plenty of original results, some of them we have gathered in this framework for the sake of mathematical chronology.

$\Po(X)$ {\dotfill the power set of $X$}\\
$\Po^{\times}(X)$ {\dotfill the set of non-empty subsets of $X$}\\
$\Po_{\mathtt{f}}(X)${\dotfill the set of finite subsets of $X$}\\
$\Po_{\mathtt{f}}^{\times}(X)${\dotfill the set of non-empty finite subsets of $X$}\\
$\Po_k(X)${\dotfill the set of subsets of $X$ with cardinality less than or equal to $k$}\\
$\Po_k^{\times}(X)${\dotfill the set of non-empty subsets of $X$ with cardinality less than or equal to $k$}\\
$\card\left(A\right)$ {\dotfill the cardinality of $A$}\\
$\den\left(A\right)$ {\dotfill the density character of $A$}\\
$\Ne_x$ {\dotfill the neighborhood filter of $x$}\\
$\inte\left(M\right)$ {\dotfill the topological interior of $M$}\\
$\inte_A\left(M\right)$ {\dotfill the topological interior of $M$ relative to $A$}\\
$\bd\left(M\right)$ {\dotfill the topological boundary of $M$}\\
$\bd_A\left(M\right)$ {\dotfill the topological boundary of $M$ relative to $A$}\\
$\cl\left(M\right)$ {\dotfill the topological closure of $M$}\\
$\cl_A\left(M\right)$ {\dotfill the topological closure of $M$ relative to $A$}\\
$\co\left(M\right)$ {\dotfill the convex hull of $M$}\\
$\cco\left(M\right)$ {\dotfill the closed convex hull of $M$}\\
$\ba\left(M\right)$ {\dotfill the balanced hull of $M$}\\
$\cba\left(M\right)$ {\dotfill the closed balanced hull of $M$}\\
$\aco\left(M\right)$ {\dotfill the absolutely convex hull of $M$}\\
$\caco\left(M\right)$ {\dotfill the closed absolutely convex hull of $M$}\\
$\spa\left(M\right)$ {\dotfill the linear span of $M$}\\
$\cspan\left(M\right)$ {\dotfill the closed linear span of $M$}\\
$\mathrm{drop}\left(M,n\right)$ {\dotfill the drop of $M$ and $n$}\\
$\suppv_f(M)${\dotfill the exposed face of the vectors of $M$ supported by $f$}\\
$\ext\left(M\right)$ {\dotfill the set of extreme points of $M$}\\
$\exp\left(M\right)$ {\dotfill the set of exposed points of $M$}\\
$\exp^{w^*}\left(M\right)$ {\dotfill the set of $w^*$-exposed points of $M$}\\
$\smo\left(M\right)$ {\dotfill the set of smooth points of $M$}\\
$\B_X\left(x,r\right)$ {\dotfill the closed ball of center $x$ and radius $r$ in $X$}\\
$\U_X\left(x,r\right)$ {\dotfill the open ball of center $x$ and radius $r$ in $X$}\\
$\E_X\left(x,r\right)$ {\dotfill the sphere of center $x$ and radius $r$ in $X$}\\
$\B_X$ {\dotfill the closed unit ball in $X$}\\
$\U_X$ {\dotfill the open unit ball in $X$}\\
$\E_X$ {\dotfill the unit sphere in $X$}\\
$\J_X$ {\dotfill the dual map of $X$}\\
$\I_X$ {\dotfill the identity map of $X$}\\
$\mathsf{NA}\left(X\right)$ {\dotfill the set of norm-attaining functionals on $X$}\\
$X^*$ {\dotfill the topological dual of $X$}\\
$X^{**}$ {\dotfill the topological bidual of $X$}\\
$c_{00}(X)$ {\dotfill the space of eventually null sequences of $X$}\\
$\ell_{1}(X)$ {\dotfill the space of absolutely summable sequences of $X$}\\
$bps(X)$ {\dotfill the space of sequences with bounded partial sums of $X$}\\
$c_{0}(X)$ {\dotfill the space of null sequences of $X$}\\
$c(X)$ {\dotfill the space of convergent sequences of $X$}\\
$ac_0(X)$ {\dotfill the space of null almost convergent sequences of $X$}\\
$ac(X)$ {\dotfill the space of almost convergent sequences of $X$}\\
$wac(X)$ {\dotfill the space of weakly almost convergent sequences of $X$}\\
$w^*ac(X^*)$ {\dotfill the space of weakly-star almost convergent sequences of $X^*$}\\
$sac(X)$ {\dotfill the space of almost summable sequences of $X$}\\
$wsac(X)$ {\dotfill the space of weakly almost summable sequences of $X$}\\
$w^*sac(X^*)$ {\dotfill the space of weakly-star almost summable sequences of $X^*$}\\
$\ell_\infty(X)$ {\dotfill the space of bounded sequences of $X$}\\
$\mathcal{S}(\sum x_n)$ {\dotfill the summing multiplier space of $\sum x_n$}\\
$\mathcal{S}_w(\sum x_n)$ {\dotfill the weakly summing multiplier space of $\sum x_n$}\\
$\mathcal{S}_{w^*}(\sum x^*_n)$ {\dotfill the weakly-star summing multiplier space of $\sum x^*_n$}\\
$\mathcal{S}_{\AC}(\sum x_n)$ {\dotfill the almost summing multiplier space of $\sum x_n$}\\
$\mathcal{S}_{\wAC}(\sum x_n)$ {\dotfill the weakly almost summing multiplier space of $\sum x_n$}\\
$\mathcal{S}_{\wsAC}(\sum x^*_n)$ {\dotfill the weakly-star almost summing multiplier space of $\sum x^*_n$}\\
$X\left(\mathcal{S}\right)$ {\dotfill the ${\cal S}$-multiplier space of summable sequences}\\
$X_w\left(\mathcal{S}\right)$ {\dotfill the ${\cal S}$-multiplier space of weakly summable sequences}\\
$X^*_{w^*}\left(\mathcal{S}\right)$ {\dotfill the ${\cal S}$-multiplier space of weakly-star summable sequences}\\
$X_{\AC}\left(\mathcal{S}\right)$ {\dotfill the ${\cal S}$-multiplier space of almost summable sequences}\\
$X_{\wAC}\left(\mathcal{S}\right)$ {\dotfill the ${\cal S}$-multiplier space of weakly almost summable sequences}\\
$X^*_{\wsAC}\left(\mathcal{S}\right)$ {\dotfill the ${\cal S}$-multiplier space of weakly-star almost summable sequences}\\
$\lim$ {\dotfill the limit function on $c(X)$}\\
$\AClim$ {\dotfill the almost convergent limit function on $ac(X)$}\\
$\wAClim$ {\dotfill the weakly almost convergent limit function on $wac(X)$}\\
$\wsAClim$ {\dotfill the weakly-star almost convergent limit function on $w^*ac(X^*)$}\\
$\ACS x_n$ {\dotfill the almost sum of $(x_n)_{n\in\N}$}\\
$\wACS x_n$ {\dotfill the weakly almost sum of $(x_n)_{n\in\N}$}\\
$\wsACS x_n$ {\dotfill the weakly-star almost sum of $(x_n)_{n\in\N}$}\\
$\mathcal{L}\left(X,Y\right)$ {\dotfill the space of continuos and linear operators from $X$ to $Y$}\\
$\mathcal{N}_X$ {\dotfill the space of invariant operators under the forward shift}\\
$\mathcal{L}_X$ {\dotfill the set of extensions of the limit function}\\
$\BL(X)$ {\dotfill the set of vector-valued Banach limits}\\

\section{Analytical geometry}

There are many different versions of geometry, each of them with its corresponding category. For example, in differential geometry the relevant category is the one of differential manifolds, in algebraic geometry we find the category of algebraic manifolds, in metric geometric the category of metric spaces comes into play, etc. The category of normed spaces is also a typical and characteristic category of a particular version of geometry, the analytical geometry.

Analytical geometry is then the field of geometry taking care of the geometrical aspects of topological vector spaces. Sometimes, the geometrical aspects of topological modules, often similar to those of topological vector spaces, are also included in analytical geometry. One of the most important and impacting branches of analytical geometry is the extremal theory.

\subsection{Convexity}

In a (topological) vector space, a subset is said to be convex provided that the segment joining any two points of the subset lies entirely in the subset. The non-empty intersection of any family of convex subsets is also convex, so the convex hull of any subset is defined as the intersection of all the convex subsets containing it. The convex hull is usually denoted as $\co(A)$. It is very easy to show that $$\co(A)=\left\{\sum_{i=1}^nt_ia_i:\sum_{i=1}^nt_i=1,t_i\geq 0,a_i\in A,1\leq i \leq n\right\}.$$

The closed convex hull, $\cco(A)$, is the intersection of all the closed convex subsets containing a given set $A$. Since the closure of any convex set is convex, we have that $\cco(A)=\cl\left(\co(A)\right)$. However, the convex hull of a closed set does not necessarily be convex (for instance, the canonical basis of $c_0$ is closed but not so is its convex hull).

A subset $A$ is said to be balanced provided that $[-1,1]A=A$. The intersection of any family of balanced subsets is also balanced, so the balanced hull, $\ba(A)$, is defined as the intersection of all the balanced subsets containing a given set $A$. It can be easily shown that $$\ba(A)=[-1,1]A=\{ta:t\in[-1,1],a\in A\}\subseteq \co\left(A \cup \left\{0\right\}\right) \cup \co\left(-A \cup \left\{0\right\}\right).$$

The closed balanced hull, $\cba(A)$, is the intersection of all the closed balanced subsets containing a given set $A$. Since the closure of any balanced set is balanced and the balanced hull of any closed set is closed, we have that $\cba(A)=\cl\left(\ba(A)\right)=\ba\left(\cl(A)\right)$.

By absolutely convex we mean a subset which is balanced and convex. The absolutely convex hull, $\aco(A)$, is, as expected, the intersection of all the absolutely convex subsets containing a given set $A$. Since the convex hull of a balanced set is balanced, we have that $$\aco(A)=\co\left(\ba(A)\right)=\left\{\sum_{i=1}^nt_ia_i:\sum_{i=1}^n|t_i|\leq 1,a_i\in A,1\leq i \leq n\right\}=\co(A\cup -A).$$

It can be shown that $$\ba(\co(A))= \co\left(A \cup \left\{0\right\}\right) \cup \co\left(-A \cup \left\{0\right\}\right),$$ formula that can be used to show that the balanced hull of a convex set need not necessarily be convex (it suffices to take $A$ to be any segment not aligned with $0$ in $\mathbb{R}^2$). The sets of the form $\co(A\cup\{b\})$ with $b\notin A$ are usually called drops and denoted by $\mathrm{drop}(A,b)$.

The closed absolutely convex hull, $\caco(A)$, is the intersection of all the closed absolutely convex subsets containing a given set $A$. Since the closure of any absolutely convex set is absolutely convex, we have that $\caco(A)=\cl\left(\aco(A)\right)$. However, the absolutely convex hull of a closed set does not necessarily be convex (for instance, the canonical basis of $c_0$ is closed but not so is its absolutely convex hull).

The reader may have noticed that all the hulls previously defined verify the conditions of a hull operator, that is, a function $T:\Po(Z)\to\Po(Z)$ verifying the following three conditions:
\begin{itemize}
\item Extensive: $A\subseteq T(A)$.
\item Increasing: $A\subseteq B\Rightarrow T(A)\subseteq T(B)$.
\item Idempotent: $T(T(A))=T(A)$.
\end{itemize}

On the other hand, a subset $A$ is called absorbing provided that for all $x\in X$ there exists $\lambda_x>0$ with $\left[-\lambda_x x,\lambda_x x\right]\subseteq A$. Keep in mind that there is no sense in defining the absorbing hull since every set containing an absorbing set is itself absorbing.

Any topological vector space has a basis of balanced and absorbing neighborhoods of $0$. A topological vector space with a basis of absolutely convex and absorbing neighborhoods of $0$ is called locally convex. A notable convex subset of topological vector spaces is shown in the next result.

\begin{theorem}\label{n}
If $X$ is a topological vector space, then the set $$\bigcap \mathcal{N}_0=\left\{x\in X: x\text{ belongs to every neighborhood of }0\right\}$$ is a closed bounded vector subspace of $X$ whose induced topology is the trivial topology. Furthermore, it is always topologically complemented with every subspace with which it is algebraically complemented.
\end{theorem}

\begin{proof}
Let us denote the above set by $N$. Then:
\begin{enumerate}
\item $N$ is a vector subspace of $X$: Indeed, let $U$ be any neighborhood of $0$ and let $n,m\in N$ and $\alpha,\beta \in \mathbb{R}$. There exists a neighborhood $V$ of $0$ such that $V+V\subseteq U$. Now, there are $W_1$ and $W_2$ neighborhoods of $0$ such that $\alpha W_1 , \beta W_2 \subseteq V$. Observe that $n \in W_1$ and $m\in W_2$. Therefore $$\alpha n + \beta m \in \alpha W_1 + \beta W_2 \subseteq V + V\subseteq U.$$
\item $N$ is closed: Indeed, let $x\in X\setminus N$. There exists a neighborhood $U$ of $0$ such that $x\notin U$. There exists another neighborhood $V$ of $0$ such that $V+V\subseteq U$. Finally, $x+V$ is a neighborhood of $x$ such that $\left(x+V\right) \cap N = \varnothing$.
\item $N$ is bounded: Obvious since $$N=\bigcap\left\{U\subseteq X: U\text{ is a neighborhood of }0\right\}.$$
\item The relative topology of $N$ is the trivial topology: Obvious from the above equality.
\item $N$ is complemented in $X$: Indeed, let $M$ be another vector subspace of $X$ such that $N \cap M=\left\{0\right\}$ and $X=M+N$. Observe that the linear projection
\begin{equation*}
\begin{array}{rcl}
X & \to & N\\
m + n & \mapsto & n
\end{array}
\end{equation*}
is continuous since the induced topology on $N$ is the trivial topology. Another way to show that the topology on $X$ coincides with the product topology on $M\oplus N$ is by means of nets. If $\left(m_i + n_i\right)_{i \in I}$ is a net of $M\oplus N$ converging to $m+n\in M\oplus N$, then $\left(n_i\right)_{i\in I}$ converges to $n$ (again, because the topology on $N$ is trivial). Therefore, $\left(m_i\right)_{i\in I}$ converges to $m$.
On the other hand, observe that $M$ is not closed (unless $N=\left\{0\right\}$). Indeed, $0\in M\subseteq \mathrm{cl}\left(M\right)$, therefore $N \subseteq \mathrm{cl}\left(M\right)$ and hence $M$ is dense in $X$.
\end{enumerate}
\end{proof}

Observe that in semi-normed spaces $\cap \;\mathcal{N}_0=\left\{x\in X: \|x\|=0\right\}$.

\subsection{Extremal theory}\label{geometry}

A point $x$ in the unit sphere $\mathsf{S}_X$ of a normed space $X$ is said to be a smooth point of $\mathsf{B}_X$ provided that there is only one functional in $\mathsf{S}_{X^*}$ attaining its norm at $x$. This unique functional is usually denoted by $\mathsf{J}_X\left(x\right)$. The set of smooth points of the (closed) unit ball $\mathsf{B}_X$ of $X$ is usually denoted as $\mathrm{smo}\left(\mathsf{B}_X\right)$. This way $X$ is said to be smooth provided that $\mathsf{S}_X=\mathrm{smo}\left(\mathsf{B}_X\right)$.

In case $x\in\mathsf{S}_X$ is not a smooth point then $\mathsf{J}_X\left(x\right)$ is defined as $x^{-1}\left(1\right)\cap\mathsf{B}_{X^*}$, that is, the set $\left\{x^*\in\mathsf{B}_{X^*}:x^*\left(x\right)=1\right\}$. 

If $X$ is a smooth normed space, then the dual map of $X$ is defined as the map $\mathsf{J}_X: X \to X^*$ such that $\left\|\mathsf{J}_X\left(x\right)\right\|=\left\|x\right\|$ and $\mathsf{J}_X\left(x\right)\left(x\right)=\left\|x\right\|^2$ for all $x\in X$. It is well known that the dual map is $\left\|\cdot\right\|$-$w^*$ continuous and verifies that $\mathsf{J}_X\left(\lambda x\right)= \lambda\mathsf{J}_X\left(x\right)$ for all $\lambda\in \mathbb{R}$ and all $x\in X$.

We refer the reader to \cite{DGZ, D1, Me} for a better perspective on these concepts. The following definition is very well known amid the analytical geometers.

\begin{definition}\label{face}
Consider a convex subset $M$ of a topological vector space $X$. A convex subset $C$ of $M$ is said to be
\begin{itemize}
\item a face of $M$ provided that $C$ verifies the extremal condition with respect to $M$, that is, if $x,y\in M$ and $t\in\left(0,1\right)$ are so that $tx+\left(1-t\right)y\in C$, then $x,y\in C$;
\item an exposed face of $M$ provided that there exists a supporting functional on $C$, that is, $f\in X^*\setminus \{0\}$ so that $C=\suppv_f(M):=\{m\in M:f(m)=\sup f(M)\}$.
\end{itemize}
\end{definition}

It is immediate that every exposed face is a face, and every proper face must be contained in the boundary of $M$. A face or an exposed face which is a singleton is called an extreme point or an exposed point, respectively. The set of extreme points of $M$ is usually denoted by $\ext(M)$, and the set of exposed points, $\exp(M)$.

In dual spaces, $w^*$-exposed faces refer to exposed faces relative to the $w^*$-topology, in other words, the supporting functional must be $w^*$-continuous. For instance, notice that $\mathsf{J}_X\left(x\right)=\suppv_x\left(\mathsf{B}_{X^*}\right)$ for every $x\in\mathsf{S}_X$, that is, $\mathsf{J}_X\left(x\right)$ is a $w^*$-exposed face of $\mathsf{B}_{X^*}$. The set of $w^*$-exposed points of $M$ is usually denoted by $\wsexp(M)$.

On the other hand, it is well known folklore that
\begin{itemize}
\item $\mathrm{ext}\left(\mathsf{B}_{\ell_{\infty}}\right)=\left\{\left(\varepsilon_n\right)_{n\in \mathbb{N}}\in \mathbb{R}^{\mathbb{N}}:\left|\varepsilon_n\right|=1\text{ for all }n\in \mathbb{N}\right\}$,
\item $\mathrm{ext}\left(\mathsf{B}_{c}\right)=\mathrm{ext}\left(\mathsf{B}_{\ell_{\infty}}\right)\cap c$,
\item $\mathrm{ext}\left(\mathsf{B}_{c_0}\right)=\mathrm{ext}\left(\mathsf{B}_{c_{00}}\right)=\varnothing$, and
\item $\ext\left(\B_{\ell_1}\right)=\{e_n:n\in\mathbb{N}\}$.
\end{itemize}

Notable extremal properties involving the concepts of extreme point, exposed point, $w^*$-exposed point, and smooth point  follow (see \cite{GPbuch}):

\begin{itemize}
\item If $x$ is a smooth point, then $\mathsf{J}_X\left(x\right)$ is a $w^*$-exposed point.
\item $\mathrm{int}_{\mathsf{S}_X}\left(C\right)\subseteq \mathrm{smo}\left(\mathsf{B}_X\right)$ for every proper face $C$ of $\mathsf{B}_X$.
\item The faces of a face of a convex set are also faces of the convex set.
\item The non-empty intersection of faces is a face.
\end{itemize}

\begin{lemma}\label{polla1}
Let $Y$ be a closed subspace of $X$ and consider $y^*\in \mathsf{S}_{Y^*}$ and $y \in \mathsf{S}_Y$ with $y^*\left(y\right)=1$. The following conditions are equivalent:
\begin{enumerate}
\item $y$ is a smooth point of $\mathsf{B}_Y$.
\item The set of all norm-$1$ Hahn-Banach extensions of $y^*$ to $X$ is the $w^*$-exposed face of $\mathsf{B}_{X^*}$ given by $y^{-1}\left(1\right) \cap \mathsf{B}_{X^*}.$
\end{enumerate}
\end{lemma}

\begin{proof}
In the first place, assume that $y$ is a smooth point of $\mathsf{B}_Y$. On the one hand, $y^{-1}\left(1\right) \cap \mathsf{B}_{X^*}$ clearly contains all norm-$1$ Hahn-Banach extensions of $y^*$ to $X$. On the other hand, if $x^*\in y^{-1}\left(1\right) \cap \mathsf{B}_{X^*}$, then the smoothness of $y$ in $\mathsf{B}_Y$ assures that $x^*|_Y=y^*$. Conversely, assume that the set of all norm-$1$ Hahn-Banach extensions of $y^*$ to $X$ is the $w^*$-closed exposed face of $\mathsf{B}_{X^*}$ given by $y^{-1}\left(1\right) \cap \mathsf{B}_{X^*}.$ If $y$ is not a smooth point of $\mathsf{B}_Y$, then we can find $y_0^* \in \mathsf{S}_{Y^*}\setminus \left\{y^*\right\}$ such that $y_0^*\left(y\right) = 1$. By the Hahn-Banach Extension Theorem, $y_0^*$ may be extended to the whole of $X$ preserving its norm. Denote this extension by $x_0^*$. Observe that $x_0^* \in y^{-1}\left(1\right) \cap \mathsf{B}_{X^*}$ but $x_0^*|_Y = y_0^* \neq y^*$.
\end{proof}

The following well-known result (see \cite[p. 168]{Banach}) will be called on in the mainmatter chapters.

\begin{theorem}[Banach, 1932; \cite{Banach}]\label{bsmo}
The smooth points of $\mathsf{B}_{\mathcal{C}\left(K\right)}$ are exactly the elements of $\mathsf{S}_{\mathcal{C}\left(K\right)}$ that attain their absolute maximum value on $K$ at only one point of $K$. 
\end{theorem}

We mean to finish this subsection with the concept of convex component, which is nothing but a maximal convex subset. Convex components allow us to extend the concept of extreme point to non-convex sets. We refer the reader to \cite{GPafa} for a meticulous study on convex components and multi-slices.

\begin{definition}[Garc\'ia-Pacheco, 2015; \cite{GPafa}]
An element $m$ of a non-empty subset $M$ of a vector space $X$ is said to be an extreme point of $M$ provided that the following conditions holds: if $C$ is a segment of $M$ containing $m$, then $m$ is an extreme of $C$. The set of extreme points of $M$ is usually denoted by $\mathrm{ext}\left(M\right)$.
\end{definition}

The previous definition matches the usual concept of extreme point in convex sets.

\begin{theorem}[Garc\'ia-Pacheco, 2015; \cite{GPafa}]\label{ei}
Let $M$ be a non-empty subset of a vector space $X$ and let $\left\{C_i\right\}_{i\in I}$ be the set of convex components of $X$.
\begin{enumerate}
\item $\mathrm{ext}\left(M\right)\subseteq \bigcup_{i\in I}\mathrm{ext}\left(C_i\right)$.
\item If $C_i\cap C_j=\varnothing$ for all $i\neq j\in I$, then $\mathrm{ext}\left(M\right)= \bigcup_{i\in I}\mathrm{ext}\left(C_i\right)$.
\end{enumerate}
\end{theorem}

\begin{proof}
\mbox{}
\begin{enumerate}
\item Obvious since $C_i\subseteq M$ for every $i\in I$.
\item Let $x\in \mathrm{ext}\left(C_i\right)$ for some $i\in I$. Let $C$ be a segment of $M$ containing $x$. Since $C$ is convex, there exists $j\in I$ such that $C\subseteq C_j$. Therefore $x\in C_i\cap C_j$ which means by hypothesis that $C_i=C_j$. Since $x\in \mathrm{ext}\left(C_i\right)$, we deduce that $x$ is an extreme of $C$.
\end{enumerate}
\end{proof}

We would like the reader to be aware of the following observations about {\em 2} of Theorem \ref{ei}:
\begin{itemize}

\item The hypothesis that the convex components be disjoint cannot be removed from {\em 2} of Theorem \ref{ei}. Indeed, if $N$ is a non-empty convex subset of any vector space such that $0\notin N$, then $M:=\co(N\cup\{0\})\cup\co(-N\cup\{0\})$ verifies that $0$ is an extreme point of a convex component of $M$ but it is not an extreme point of $M$.

\item The converse to {\em 2} of Theorem \ref{ei} does not hold true as shown in the next example. Indeed, consider the $2$-dimensional Euclidean space, that is, $\ell_2^2$. Take $M:=2\mathsf{B}_{\ell_2^2}\setminus \mathsf{U}_{\ell_2^2}$. It is not difficult to observe that the convex components of $M$ are the sets $$C_f:=M\cap f^{-1}\left(\left[1,+\infty\right)\right)$$ where $f\in \mathsf{S}_{\left(\ell_2^2\right)^*}$. Finally, it can be checked that $$\mathrm{ext}\left(M\right)=\bigcup_{f\in  \mathsf{S}_{\left(\ell_2^2\right)^*}}\mathrm{ext}\left(C_f\right).$$ However, the convex components of $M$ are not pairwise disjoint.
\end{itemize}

We spare to the reader the details of the proof of the following technical lemma.

\begin{lemma}[Garc\'ia-Pacheco, 2015; \cite{GPafa}]\label{superhelp}
Let $M$ be a convex subset of a vector space $X$.
\begin{enumerate}
\item If $x\in X\setminus M$, then $\mathrm{ext}\left(\mathrm{co}\left(M \cup \left\{x\right\}\right)\right)\setminus \left\{x\right\} \subseteq \mathrm{ext}\left(M\right)$.
\item If $0\notin M$, then $$\ext\left(\ba\left(M\right)\right) =  \mathrm{ext}\left(\mathrm{co}\left(M \cup \left\{0\right\}\right)\right)\setminus \left\{0\right\} \cup \mathrm{ext}\left(\mathrm{co}\left(-M \cup \left\{0\right\}\right)\right)\setminus \left\{0\right\}.$$
\item $\mathrm{ext}\left(\ba\left(M\right)\right) \subseteq \mathrm{ext}\left(M\right)\cup \mathrm{ext}\left(-M\right)$.
\end{enumerate}
\end{lemma}

\section{Sequences and series}

The most general setting in which summability takes place is the category of topological additive groupoids, that is, additive groupoids endowed with a topology that makes the addition continuous.

\subsection{Spaces of sequences}

The reader should be familiar with the usual sequence spaces such as $$c_{00}\left(X\right)\subset \ell_1(X)\subset c_0\left(X\right)\subset c\left(X\right)\subset \ell_\infty\left(X\right).$$ If $X$ is complete, then so is each of them when $\ell_\infty\left(X\right)$ is endowed with the sup norm, except for the very first one, $c_{00}\left(X\right)$, which is dense in $c_0\left(X\right)$.

We recall the reader that $\overline{X}$ stands for the completion of $X$. It is fairly trivial that $c_{00}(X)$, $\ell_1(X)$, $c_0(X)$, and $\ell_\infty(X)$ are dense in $c_{00}\left(\overline{X}\right)$, $\ell_1\left(\overline{X}\right)$, $c_0\left(\overline{X}\right)$, and $\ell_\infty\left(\overline{X}\right)$, respectively. The density of $c(X)$ in $c\left(\overline{X}\right)$ is not that trivial.

\begin{lemma}\label{bixic}
$c\left(X\right)$ is dense in $c\left(\overline{X}\right)$.
\end{lemma}

\begin{proof}
Let $\left(y_n\right)_{n\in\mathbb{N}}\in c\left(\overline{X}\right)$ and fix an arbitrary $\varepsilon >0$. For every $n\in\mathbb{N}$ let $x_n\in X$ such that $\left\|y_n-x_n\right\|\leq \varepsilon/2$. Take $n_0\in\mathbb{N}$ such that $1/n_0 < \varepsilon/2$ and $\left\|y_p-y_q\right\|<\varepsilon/2$ for all $p,q\geq n_0$. Then $\left(x_1,x_2,\dots,x_{n_0-1},x_{n_0},x_{n_0},x_{n_0},\dots\right)\in c\left(X\right)$ and $$\left\|\left(x_1,x_2,\dots,x_{n_0-1},x_{n_0},x_{n_0},x_{n_0},\dots\right)-\left(y_n\right)_{n\in\mathbb{N}}\right\|_\infty <\varepsilon$$ since $$\left\|x_{n_0}-y_n\right\|\leq \left\|x_{n_0}-y_{n_0}\right\| + \left\|y_{n_0}-y_n\right\| < \frac{\varepsilon}{2}+  \frac{\varepsilon}{2} = \varepsilon$$ for all $n\geq n_0$.
\end{proof}

It is easy to check that $c(X)$ is never dense in $\ell_\infty(X)$. When $X=\mathbb{R}$ a separability argument can be applied. In general, one can see that a sequence like $\left((-1)^nx\right)_{n\in\mathbb{N}}$ can never be approximated by convergent sequences in the sup norm if $x\neq 0$.

\begin{theorem}\label{jopac}
The following conditions are equivalent:
\begin{enumerate}
\item $c\left(\overline{X}\right)\cap \ell_\infty\left(X\right)=c\left(X\right)$.
\item $c\left(X\right)$ is a closed subspace of $\ell_\infty\left(X\right)$.
\item $X$ is complete.
\end{enumerate}
\end{theorem}

\begin{proof}
\mbox{}
\begin{enumerate}
\item[{\em 1} $\Rightarrow$ {\em 2}] Immediate if taken into account that $c\left(\overline{X}\right)$ is closed in $\ell_\infty\left(\overline{X}\right)$ since $\overline{X}$ is complete.
\item[{\em 2} $\Rightarrow$ {\em 3}] Consider a Cauchy sequence $\left(x_n\right)_{n\in\mathbb{N}}\subset X$. It is obvious that $\left(x_n\right)_{n\in\mathbb{N}}\in  c\left(\overline{X}\right)\cap \ell_\infty\left(X\right)= c\left(X\right)$. So $\left(x_n\right)_{n\in\mathbb{N}}$ is convergent in $X$.
\item[{\em 3} $\Rightarrow$ {\em 1}] Obvious.
\end{enumerate}
\end{proof}

The reader should realize that Theorem \ref{jopac} does not hold if we substitute $c\left(X\right)$ with $c_{00}(X)$, $\ell_1(X)$, or $c_0(X)$, as $c_{00}(X)=c_{00}\left(\overline{X}\right)\cap \ell_\infty(X)$, $\ell_1(X)=\ell_1\left(\overline{X}\right)\cap \ell_\infty(X)$, and $c_0(X)=c_0\left(\overline{X}\right)\cap \ell_\infty(X)$. As a consequence, $\ell_1(X)$ and $c_0(X)$ are always closed in $\ell_\infty(X)$ even if $X$ is not complete. On the other hand, $c_0(X)=\ker\left(\lim\right)$, where $\lim$ clearly denotes the limit function on $c(X)$.

A not so usual sequence space is the one of all sequences in $X$ with bounded partial sums, usually denoted by $bps\left(X\right)$. It is clear that $$c_{00}\left(X\right)\subset \ell_1\left(X\right)\subset bps\left(X\right)\subset \ell_\infty\left(X\right).$$ As expected, $bps(X)$ is dense in $bps\left(\overline{X}\right)$ and $bps(X)=bps\left(\overline{X}\right)\cap \ell_\infty(X)$. However, $bps\left(X\right)$ is not closed in $\ell_\infty\left(X\right)$ even if $X$ is complete, since $c_{00}\left(X\right)$ is dense in $c_0\left(X\right)$ and $c_0\left(X\right)$ contains sequences with unbounded partial sums. On the other hand, it can also be seen without employing an enormous effort that $bps(X)$ is never dense in $\ell_\infty(X)$. The following lemma clarifies the nature of $bps\left(X\right)$ (see \cite[Lemma 2.2]{AAGPPF}).

\begin{lemma}\label{bps}
$bps\left(X\right)=\left\{\left(z_{n+1}-z_n\right)_{n\in\mathbb{N}}:\left(z_n\right)_{n\in\mathbb{N}}\in \ell_\infty\left(X\right)\right\}.$
\end{lemma}

\begin{proof}
If $\left(a_n\right)_{n\in\N}\in bps\left(X\right)$, then choose any $z_1 \in X$ and (necessarily) define $$z_{n+1}:=z_1 + \sum_{i=1}^na_i$$ for all $n\in \mathbb{N}$. It is immediate that $\left(z_n\right)_{n\in\mathbb{N}}\in \ell_\infty\left(X\right)$ and $z_{n+1}-z_n=a_n$ for all $n\in\N$. Conversely, let $\left(z_n\right)_{n\in\mathbb{N}}\in \ell_\infty\left(X\right)$. Simply notice that for all $p\in \mathbb{N}$ we have that $$\left\|\sum_{k=1}^p\left(z_{k+1}-z_{k}\right)\right\|=\left\|z_{p+1}-z_{1}\right\|\leq 2\left\|\left(z_n\right)_{n\in \mathbb{N}}\right\|_{\infty}.$$
\end{proof}

On the other hand, the following linear operator $$\begin{array}{rcl} bps\left(X\right)&\to&\ell_\infty\left(X\right)\\
 \left(x_i\right)_{i\in\mathbb{N}}&\mapsto&\left(x_1,x_1+x_2,x_1+x_2+x_3,\dots\right),
\end{array}$$ whose inverse is precisely (see Lemma \ref{bps} above) $$\begin{array}{rcl} \ell_\infty\left(X\right)&\to&bps\left(X\right)\\
 \left(z_n\right)_{n\in\mathbb{N}}&\mapsto&\left(z_1,z_2-z_1,z_3-z_2,\dots\right),
\end{array}$$ induces the following norm on $bps\left(X\right)$:
\begin{equation}\label{99}
\left\| \left(x_i\right)_{i\in\mathbb{N}} \right\|=\sup_{n\in\mathbb{N}} \left\| \sum_{i=1}^n x_i \right\|=\left\|\left(\sum_{i=1}^{n}x_i\right)_{n\in\mathbb{N}}\right\|_{\infty},
\end{equation}
which will make it complete if $X$ is so. Notice that when endowed with the norm given in \eqref{99}, $bps\left(X\right)$ is linearly isometric to $\ell_\infty\left(X\right)$.

If $x\in X$, then ${\bf x}$ denotes the constant sequence of term equal to $x$, and ${\bf X}$ denotes the set of all constant sequences. It is clear that ${\bf X}$ is a closed subspace of $\ell_\infty\left(X\right)$ and contained in $c(X)$.

\begin{lemma}\label{pollazo}
Let $x\in \mathsf{S}_X$ and $x^*\in \mathsf{S}_{X^*}$ with $x^*\left(x\right)=1$ and fix any arbitrary $0<\varepsilon < 1$.
\begin{enumerate}
\item $x^*\left(\mathsf{U}_X\left(x,\varepsilon\right)\right)\subset \left(1-\varepsilon,+\infty\right)$.
\item If $\left(x_n\right)_{n\in\mathbb{N}}\in bps\left(X\right)$, then $\left(x_n\right)_{n\in\mathbb{N}}\nsubseteq \mathsf{U}_X\left(x,\varepsilon\right)$.
\item $d\left({\bf x},bps\left(X\right)\right)= 1$.
\end{enumerate}
\end{lemma}

\begin{proof}
\mbox{}
\begin{enumerate}
\item Because of the convexity of $x^*\left(\mathsf{U}_X\left(x,\varepsilon\right)\right)$ it suffices to show that no element $y\in \mathsf{U}_X\left(x,\varepsilon\right)$ verifies that $x^*\left(y\right)=1-\varepsilon$. Indeed, assume to the contrary that such element $y$ exists. Then $$\varepsilon=\left|x^*\left(x\right)-x^*\left(y\right)\right|=\left|x^*\left(x-y\right)\right|\leq \left\|x-y\right\|<\varepsilon,$$ which is a contradiction. 

\item Suppose that $\left(x_n\right)_{n\in\mathbb{N}}\subset \mathsf{U}_X\left(x,\varepsilon\right)$. Then $x^*\left(x_n\right)>1-\varepsilon$ for all $n\in\mathbb{N}$. Let $M>0$ be such that $$\left\|\sum_{n=1}^m x_n\right\|\leq M$$ for all $m\in\mathbb{N}$. Finally notice that $$m\left(1-\varepsilon\right)<x^*\left(\sum_{n=1}^m x_n\right)\leq \left\|\sum_{n=1}^m x_n\right\|\leq M$$ for all $m\in \mathbb{N}$ which is impossible. 

\item By {\em 2} we know that $\left\|\left(x_n\right)_{n\in\mathbb{N}}-\mathbf{x}\right\|_\infty \geq \varepsilon$ for all $\left(x_n\right)_{n\in\mathbb{N}}\in bps\left(X\right)$. Thus, $d\left({\bf x},bps\left(X\right)\right)\geq 1$. In order to see that $d\left({\bf x},bps\left(X\right)\right)= 1$ it only suffices to realize that $$\left(x,0,0,\dots,0,\dots\right)\in bps\left(X\right)$$ and $$\left\|\left(x,0,0,\dots,0,\dots\right)-\mathbf{x}\right\|_\infty =1.$$

\end{enumerate}
\end{proof}

A class of spaces with special convergence properties is the class of all Grothendieck spaces. Recall (see \cite{AizGuPer}) that a vector subspace $M$ of the dual $X^{\ast\ast}$ of a normed space $X$ is called a $M$-Grothendieck space if every sequence in $X^\ast$ which is $\sigma(X^\ast,X)$ convergent is also $\sigma (X^\ast,M)$ convergent. As expected, $X$ is said to be Grothendieck if it is $X^{\ast\ast}$-Grothendieck. In other words, a normed space $X$ has the Grothendieck property if every weakly-star convergent sequence in $X^*$ is weakly convergent.

If $\mathcal{S}$ is a vector subspace of $\ell_\infty$ containing $c_0$, then $\ell_\infty$ can be identified with a vector subspace of $\mathcal{S}^{\ast\ast}$ via the bidual map corresponding to the natural inclusion of $c_0 $ into $\mathcal{S}$. This bidual map is an isometry from $c_0^{\ast\ast} \equiv \ell_\infty$ into $\mathcal{S}^{\ast\ast}$. Therefore, we can identify a bounded sequence $\left(a_j\right)_{j\in \mathbb{N}} \in \ell_\infty$ with the map
\begin{equation*}
\begin{array}{rcl}
\mathcal{S}^\ast &\to& \mathbb{R}\\
g&\mapsto & \sum_{j=1}^\infty a_jg\left(e^j\right),
\end{array}
\end{equation*}
where $\left(e^j\right)$ stands for the canonical $c_0$-basis. As a consequence, it makes sense to ask whether $\mathcal{S}$ is $\ell_\infty$-Grothendieck. This observation is relevant towards the statement of Theorem \ref{teor}.

The end of this subsection is due to introduce a new and different concept of closure defined in $\ell_\infty(X)$. It is trivial that if $A$ is a non-empty subset of $\ell_\infty(X)$, then $$\cl(A)=\bigcap_{\varepsilon >0} \left(A+\varepsilon\B_{\ell_\infty(X)}\right).$$ By taking advantage of this obvious fact, we define the $\ell_p$-closure as: $$\cl_p(A)=\bigcap_{\varepsilon >0} \left(A+\varepsilon\B_{\ell_p(X)}\right).$$ Since $\|\cdot\|_\infty\leq \|\cdot\|_p$, we have that $$A\subseteq \cl_p(A)\subseteq \cl_\infty(A)=\cl(A).$$ Finally, a set $A$ such that $A=\cl_p(A)$ is called $p$-closed.

\subsection{Classifications of series}\label{vectorseries}

The reader of this book oughts to be necessarily familiar with the usual convergence concepts related to series, such as unconditional, conditional, absolute, and subseries convergence (uc, cc, ac, and sc, respectively). The concepts of unconditionally Cauchy and weakly unconditionally Cauchy series (uC and wuC, respectively) should also sound familiar to the reader as well as their usual characterizations:
\begin{itemize}
\item A normed space is finite-dimensional if and only if every uc series is ac (this is the famous Dvoretzky-Roger Theorem).
\item A normed space is complete if and only if every ac series is uc (see Theorem \ref{incomplete} for a new different proof without having involving the Baire Category Theorem).
\item A series $\sum x_i$ in a normed space $X$ is wuC if and only if $\sum |x^*(x_i)|< \infty$ for each $x^* \in X^*$.
\item A series $\sum x_i^*$ in $X^*$ is wuC if and only if $\sum |x^*_i(x)|< \infty$ for each $x \in X$.
\end{itemize}

We refer the reader to \cite{D2} for a magnificent perspective on these concepts. Precisely in \cite{D2} the following characterization of wuC series can be found, on which we will rely several times in this book.

\begin{theorem}[Diestel, 1984; \cite{D2}]\label{DiestelwuC}
A series $\sum x_i$ in a normed space $X$ is wuC if and only if there exists $H>0$ such that
\begin{eqnarray*}
  H &=& \sup \left\{\left \|\sum_{i=1}^n a_i x_i\right\|: n \in \mathbb{N}, |a_i| \leqslant 1, i \in \{1, \dots, n\} \right\}  \\
   &=& \sup \left\{ \left\|\sum_{i=1}^n \varepsilon_i x_i\right\|: n \in \mathbb{N}, \varepsilon_i \in \{-1,1\}, i \in \{1, \dots, n \}\right \}  \\
   &=&  \sup \left\{ \sum_{i=1}^n\left |f(x_i)\right| : f \in \B_{X^*} \right\}.
\end{eqnarray*}
\end{theorem}

Now, we remind the reader about the famous and classic Orlicz-Pettis Theorem, which will be versioned in the last subsection of the third chapter for the almost summability.

\begin{theorem}[Orlicz, 1929; \cite{O}]\label{OP}
If $\sum x_i$ is a series in a Banach space $X$ such that $w\sum_{i \in M} x_i$ exists for each $M \subset \mathbb{N}$, then $\sum x_i$ is $uc$.
\end{theorem}

In more simple words, the Orlicz-Pettis Theorem establishes that w-sc implies uc in the category of Banach spaces.

\begin{theorem}[Schur, 1929; \cite{O}]\label{HS}
For each $i\in \N$, let $\sum_j t_{ij}$ be an ac series of scalars. Assume that $\lim_i\sum_js_jt_{ij}$ exists for every $(s_j)_{j\in\N}\in\ell_\infty$ and let $t_j:=\lim_it_{ij}$ for every $j\in\N$. Then:
\begin{itemize}
\item $\sum_j t_j$ is ac and $\lim_i\sum_js_jt_{ij}=\sum_js_jt_j$.
\item $\lim_i\sum_j|t_{ij}-t_j|=0$.
\item $\sum_j|t_{ij}|$ converges uniformly on $i\in\N$.
\end{itemize}
\end{theorem}

\subsection{Biorthogonal systems}

In a vector space a biorthogonal system is a pair $\left(e_i,e_i^*\right)_{i\in I}\subseteq X\times X^*$ such that $e_i^*(e_j)=\delta_{ij}$ where $\delta_{ij}$ is the Kronecker $\delta$. A biorthogonal system is said to be:
\begin{itemize}
\item expanding if $X=\spa\{e_i:i\in I\}$;
\item fundamental if $X=\cspan\{e_i:i\in I\}$ (provided that $X$ is endowed with a vector topology);
\item total if $X^*=\cspan^{w^*}\{e_i^*:i\in I\}$.
\end{itemize}

We refer the reader to \cite{Bior} for a complete perspective on biorthogonal systems in normed spaces.

A very well known fact in the theory of Banach spaces states that the infinite dimensional Banach spaces must have uncountable dimension. This fact involves the Baire Category Theorem. We will provide a much more simple proof of this fact without involving strong machinery.

\begin{theorem}\label{ex2}
If $X$ is a Hausdorff locally convex topological vector space of infinite countable dimension, then there exists a expanding and total biorthogonal system $\left(e_n,e^*_n\right)_{n\in \mathbb{N}}\subseteq X \times X^*$.
\end{theorem}

\begin{proof}
Let $\left(u_n\right)_{n\in \mathbb{N}}\subset X$ be a Hamel basis for $X$. We will construct the biorthogonal system inductively:
\begin{enumerate}
\item[{\bf Step 1}] Choice of $e_1$:
\begin{enumerate}
\item[{\bf Step 1.1}] Take $e_1:= u_1$. Obviously, $\mathrm{span}\left\{e_1\right\}=\mathrm{span}\left\{u_1\right\}$.
\item[{\bf Step 1.2}] The Hahn-Banach Theorem allows us to find $e_1^*\in X^*$ such that $e^*_1\left(e_1\right)=1$.
\end{enumerate}
\item[{\bf Step 2}] Choice of $e_2$:
\begin{enumerate}
\item[{\bf Step 2.1}] Take $e_2:= u_2-e^*_1\left(u_2\right)u_1$. Note that $\mathrm{span}\left\{e_1,e_2\right\}=\mathrm{span}\left\{u_1,u_2\right\}$.
\item[{\bf Step 2.2}] The Hahn-Banach Theorem allows us to find $e_2^*\in X^*$ such that $1=e_2^*\left(e_2\right) > \sup e_2^*\left(\mathbb{R}e_1\right).$ Therefore $e_2^*\left(e_1\right)=0$.
\end{enumerate}
\item[{\bf Step 3}] Choice of $e_3$:
\begin{enumerate}
\item[{\bf Step 3.1}] Take $e_3:= u_3-e_1^*\left(u_3\right)e_1-e_2^*\left(u_3\right)e_2$. Observe that $\mathrm{span}\left\{e_1,e_2,e_3\right\}=\mathrm{span}\left\{u_1,u_2,u_3\right\}$.
\item[{\bf Step 3.2}] The Hahn-Banach Theorem allows us to find $e_3^*\in X^*$ such that $1=e_3^*\left(e_3\right) > \sup e_3^*\left(\mathbb{R}e_1\oplus \mathbb{R}e_2\right).$ Therefore $e_3^*\left(e_1\right)=e_3^*\left(e_2\right)=0$.
\end{enumerate}
\end{enumerate}
We omit the rest of the steps. To see that $X^*=\overline{\mathrm{span}}^{\hspace{.1cm}w^*}\left\{e^*_n:n\in\mathbb{N}\right\}$ it suffices to realize that $\mathrm{span}\left\{e^*_n:n\in\mathbb{N}\right\}$ separates points of $X$.
\end{proof}

\begin{corollary}\label{incomplete}
If $X$ is a normed space of infinite countable dimension, then there exists an absolutely convergent series in $X$ which is non-convergent, in other words, $X$ is not complete.
\end{corollary}

\begin{proof}
By Theorem \ref{ex2}, there exists a biorthogonal system $\left(e_n,e^*_n\right)_{n\in \mathbb{N}}\subseteq X \times X^*$ such that $X= \mathrm{span}\left\{e_n:n\in \mathbb{N}\right\}$ and $X^*=\overline{\mathrm{span}}^{\hspace{.1cm}w^*}\left\{e^*_n:n\in\mathbb{N}\right\}$. We may assume that $\left(e_n\right)_{n\in \mathbb{N}}\subset \mathsf{S}_X$. Note that the series $\sum_{n=1}^\infty \frac{1}{2^n}e_n$ is absolutely convergent. Assume that $\sum_{n=1}^\infty \frac{1}{2^n}e_n$ is convergent in $X$. There are $\lambda_1,\dots,\lambda_p\in \mathbb{R}$ such that $\sum_{n=1}^\infty \frac{1}{2^n}e_n = \sum_{n=1}^p \lambda_n e_n$. Finally, $$\frac{1}{2^{p+1}}=e^*_{p+1}\left(\sum_{n=1}^\infty \frac{1}{2^n}e_n\right)=e^*_{p+1}\left(\sum_{n=1}^p \lambda_n e_n\right)=0,$$ which is impossible.
\end{proof}

\subsection{Schauder bases}

A topological vector space is said to have a Schauder basis $(e_n)_{n\in\mathbb{N}}\subset X$ provided that for every $x\in X$ there exists a unique sequence $(\lambda_n)_{n\in\mathbb{N}}\subset \mathbb{K}$ in such a way that:
\begin{itemize}
\item $\sum_{n=1}^\infty\lambda_ne_n$ converges to $x$. 
\item The coefficient functionals $$\begin{array}{rrcl} e_n^* :&X&\to &\mathbb{K}\\ & x &\mapsto &\lambda_n\end{array}$$ are continuous for every $n\in\N$.
\end{itemize}

Schauder bases were introduced for the first time in \cite{S}. Observe that $(e_n,e^*_n)_{n\in\N}$ is a fundamental and total biorthogonal system (the totalness follows from the fact that $\{e_n^*:n\in\mathbb{N}\}$ is separating). We would like to recall the reader that in case $X$ is a Banach space, it is not necessary to include in the definition of Schauder basis the continuity of the coefficient functionals $e_n^*$'s as it can be derived by using the completeness.

The $n$-th canonical projections $$\begin{array}{rrcl} p_n:&X&\to & \mathrm{span}\{e_1,\dots,e_n\}\\ & x&\mapsto & \sum_{i=1}^n\lambda_ie_i\end{array}$$ are continuous linear projections. Notice that the set $\{p_n:n\in\mathbb{N}\}$ is pointwise bounded since $(p_n(x))_{n\in\mathbb{N}}$ is a convergent sequence to $x$. Therefore if $X$ is a barrelled normed space, then the set $\{p_n:n\in\mathbb{N}\}$ is bounded. In barrelled normed spaces, $\mathrm{bc}\left((e_n)_{n\in\mathbb{N}}\right):=\sup\{\|p_n\|:n\in\mathbb{N}\}$ is called the basis constant.

A normalized Schauder basis in a barrelled normed space is said to be monotone provided that the canonical projections have norm $1$. It is very simple to verify that under barrelledness \begin{equation*}\label{m}|||x|||_m:=\sup\{\|p_n(x)\|:n\in\mathbb{N}\}\end{equation*} is an equivalent norm on $X$ that makes a normalized Schauder basis monotone. We refer the reader to \cite{GH} where it is shown that a uniformly Frechet smooth Banach space with a Schauder basis can be equivalently renormed to remain uniformly Frechet smooth and to make the basis monotone.

 For simplicity purposes, we convey that $p_0$ is the null projection. Notice that a Schauder basis always verify that $$p_n(x)-p_{n-1}(x)=e_n^*(x)e_n $$ for all $n\in\mathbb{N}$ and all $x\in X$. Therefore $$ \|p_n-p_{n-1}\|=\|e_n^*\|^*$$ for all $n\in\mathbb{N}$.
 
 We also recall the reader that another norm $|||\cdot|||$ on a normed space is said to be left-comparable to the original norm provided that there exists a positive constant $c>0$ in such a way that $$|||\cdot|||\leq c\|\cdot\|.$$ 
 
 In a similar way, the concept of right-comparability for norms can be defined.

\begin{lemma}\label{du}
If $X$ is an infinite dimensional barrelled normed space with a normalized Schauder basis $(e_n)_{n\in\mathbb{N}}$, then \begin{equation}\label{d}|||x|||_d:= \sup\{|e_n^*(x)|:n\in\mathbb{N}\}\end{equation} is a norm on $X$ left-comparable to its original norm verifying that $|||e_n^*|||_d^*=1$ for all $n\in\mathbb{N}$.
\end{lemma}

\begin{proof}
By the observation prior to this lemma, we have that $$|||x|||_d\leq \sup\{\|p_n-p_{n-1}\|:n\in\mathbb{N}\}\|x\|\leq 2\mathrm{bc}\left((e_n)_{n\in\mathbb{N}}\right)\|x\|.$$
\end{proof}

In the settings of the previous lemma, since $(e_n)_{n\in\mathbb{N}}$ is normalized we deduce that $(e_n^*(x))_{n\in\mathbb{N}}$ converges to $0$ for all $x\in X$, therefore $X$ endowed with the norm $|||\cdot|||_d$ given in Equation \eqref{d} is isometric to a subspace of $c_0$ according to the embedding $$\begin{array}{rcl} X & \to & c_0 \\ x&\mapsto&(e_n^*(x))_{n\in\mathbb{N}} \end{array}$$

 This fact leads us to the following corollary.

\begin{corollary}
If $X$ is an infinite dimensional barrelled normed space with a normalized Schauder basis $(e_n)_{n\in\mathbb{N}}$ such that $|||\cdot|||_d$ is equivalent to its original norm, then $X$ is isomorphic to a subspace of $c_0$.
\end{corollary}

We are now at the right point to find an equivalent norm to make a normalized Schauder basis monotone and to make its dual basis normalized.

\begin{theorem}\label{myd}
Let $X$ be a barrelled normed space with a normalized Schauder basis $(e_n)_{n\in\mathbb{N}}$. Then \begin{equation}\label{md}|||x|||_{md}:=\max\left\{ |||x|||_m,|||x|||_d\right\}\end{equation} verifies the following properties:
\begin{enumerate}
\item It is an equivalent norm on $X$.
\item $|||e_n|||_{md}=|||e_n^*|||_{md}^*=1$ for all $n\in\mathbb{N}$.
\item For every $n\in\mathbb{N}$, $p_n$ has $|||\cdot|||_{md}$-norm $1$.
\end{enumerate}
\end{theorem}

\begin{proof}
\mbox{}
\begin{enumerate}
\item It is trivial that $|||\cdot|||_{md}$ defines a norm on $X$. Notice in first place that $|||x|||_m\leq \mathrm{bc}\left((e_n)_{n\in\mathbb{N}}\right)\|x\|$ and, in virtue of Lemma \ref{du}, $|||x|||_d\leq 2\mathrm{bc}\left((e_n)_{n\in\mathbb{N}}\right)\|x\|$, thus $|||x|||_{md}\leq 2\mathrm{bc}\left((e_n)_{n\in\mathbb{N}}\right)\|x\|$. Now let $x\in X$. Observe that $(p_n(x))_{n\in\mathbb{N}}$ converges to $x$, therefore $\|x\|\leq |||x|||_m\leq |||x|||_{md}$.
\item Obviously, $|||e_n|||_{md}=\|e_n\|=1$ for all $n\in\mathbb{N}$. Thus, it suffices to show that $|||e_n^*|||_{md}^*\leq 1$ for all $n\in\mathbb{N}$. Indeed, if $x\in X$, then $|e_n^*(x)| \leq |||x|||_d\leq|||x|||_{md}$ for all $n\in\mathbb{N}$.
\item If $x\in X$, then $\|p_n(x)\|\leq |||x|||_m\leq |||x|||_{md}$ for all $n\in\mathbb{N}$.
\end{enumerate}
\end{proof}

In the settings of the previous theorem, we have the following chain of inequalities $$\|\cdot\|\leq|||\cdot|||_m\leq|||\cdot|||_{md}\leq 2\mathrm{bc}\left((e_n)_{n\in\mathbb{N}}\right)\|\cdot\|.$$ Therefore in terms of unit balls $$ \mathsf{B}_{|||\cdot|||_{md}}\subseteq \mathsf{B}_{|||\cdot|||_m}\subseteq\mathsf{B}_X\subseteq 2\mathrm{bc}\left((e_n)_{n\in\mathbb{N}}\right)\mathsf{B}_{|||\cdot|||_{md}}.$$

The norm $|||\cdot|||_{d}$ can actually be generalized in order to achieve a right-comparable norm. 
\begin{theorem}\label{norma2}
Let $X$ be a barrelled normed space with a normalized Schauder basis $(e_n)_{n\in\mathbb{N}}$. Fix an arbitrary $k\in\mathbb{N}$ and consider \begin{equation}\label{k} |||x|||_{k}:=\sup\left\{\left\|\sum_{i\in A}\lambda_ie_i\right\|:x=\sum_{i=1}^\infty\lambda_ie_i,\;A\in\mathcal{P}_k^{\times}(\mathbb{N})\right\}.\end{equation} Then:
\begin{enumerate}
\item $|||\cdot|||_1=|||\cdot|||_d$.
\item $\|p_k(x)\|\leq |||x|||_k$ for all $k\in\mathbb{N}$ and all $x\in X$.
\item $|||\cdot|||_k\leq 2k\mathrm{bc}\left((e_n)_{n\in\mathbb{N}}\right)\|\cdot\|$ for all $k\in\mathbb{N}$.
\item $|||\cdot|||_{k_1}\leq|||\cdot|||_{k_2}$ if $k_1\leq k_2$ and thus $ \mathsf{B}_{|||\cdot|||_{k_2}}\subseteq\mathsf{B}_{|||\cdot|||_{k_1}}$.
\item $|||\cdot|||_k$ is a left-comparable norm on $X$ such that for every $A\in\mathcal{P}_k^{\times}(\mathbb{N})$, $\sum_{i\in A}p_i-p_{i-1}$ has $|||\cdot|||_k$-norm $1$. 
\end{enumerate}
In particular, when $X$ is endowed with the norm $|||\cdot|||_k$ both the Schauder basis $(e_n)_{n\in\mathbb{N}}$ and its dual basis are normalized.
\end{theorem}

\begin{proof}
First off, notice that $$\sum_{i\in A}\lambda_ie_i=\sum_{i\in A}p_i(x)-p_{i-1}(x) $$ for all $A\in\mathcal{P}_k^{\times}(\mathbb{N})$. Then $$|||x|||_k=\sup\left\{\left\|\sum_{i\in A}p_i(x)-p_{i-1}(x)\right\|:x=\sum_{i=1}^\infty\lambda_ie_i,\;A\in\mathcal{P}_k^{\times}(\mathbb{N})\right\}$$ and from here it is easy to understand that $\sum_{i\in A}p_i-p_{i-1}$ has $|||\cdot|||_k$-norm $1$ for every $A\in\mathcal{P}_k^{\times}(\mathbb{N})$. We will only show (2), we spare the rest of the details of the proof to the reader. Observe that $A=\{1,\dots,k\}\in \mathcal{P}_k^{\times}(\mathbb{N})$ and $$|||x|||_k\geq \left\|\sum_{i\in A} p_i(x)-p_{i-1}(x)\right\|=\left\|p_k(x)\right\|.$$
\end{proof}

We refer the reader to \cite{PBA} for several interesting characterizations of completeness of normed spaces through weakly unconditionally Cauchy series.

We recall the reader that a weakly unconditionally Cauchy Schauder basis is a Schauder basis $(e_n)_{n\in\mathbb{N}}$ in which the series $x=\sum_{n=1}^\infty\lambda_n e_n$ is weakly unconditionally Cauchy. Unconditional Schauder bases are examples of weakly unconditionally Cauchy Schauder bases.

\begin{corollary}
Let $X$ be a barrelled normed space with a normalized weakly unconditionally Cauchy Schauder basis $(e_n)_{n\in\mathbb{N}}$. Then $$|||\cdot |||_\infty=\sup_{k\to\infty}|||\cdot|||_k$$ is a right-comparable norm on $X$ such that for every $A\in\mathcal{P}_{\mathtt{f}}^{\times}(\mathbb{N})$, $\sum_{i\in A}p_i-p_{i-1}$ has $|||\cdot|||_\infty$-norm $1$. In particular, when $X$ is endowed with the norm $|||\cdot|||_\infty$ the Schauder basis $(e_n)_{n\in\mathbb{N}}$ is monotone and its dual basis is normalized.
\end{corollary}

\begin{proof}
According to (2) of Theorem \ref{norma2} we deduce that $\|\cdot\|\leq |||\cdot|||_\infty$. It only remains to show that $|||\cdot|||_\infty$ is well defined. Fix an arbitrary $x=\sum_{i=1}^\infty\lambda_ie_i\in X$. Since the previous series is weakly unconditionally Cauchy, Theorem \ref{DiestelwuC} allows us to deduce that there exists a constant $H_x>0$ such that
\begin{eqnarray*}
|||x|||_\infty&=&\sup\left\{\left\|\sum_{i\in A}\lambda_ie_i\right\|:x=\sum_{i=1}^\infty\lambda_ie_i,\;A\in\mathcal{P}_{\mathtt{f}}^{\times}(\mathbb{N})\right\}\\
&=& \sup\left\{\left\|\sum_{i\in A}p_i(x)-p_{i-1}(x)\right\|:x=\sum_{i=1}^\infty\lambda_ie_i,\;A\in\mathcal{P}_{\mathtt{f}}^{\times}(\mathbb{N})\right\}\\
&\leq & H_x
\end{eqnarray*}
\end{proof}

\section{Operator theory}

Operator theory takes care of studying the morphisms of a given algebraic and topological category. In our case, we will consider operators between normed spaces, or more generally, between topological vector spaces.

As usual, $\mathcal{L}\left(X,Y\right)$ is meant to be the space of all continuous and linear operators from $X$ to $Y$. When $X$ and $Y$ are simply topological vector spaces, we can just consider the pointwise convergence topology on $\mathcal{L}\left(X,Y\right)$. If they are normed spaces, then $\mathcal{L}\left(X,Y\right)$ can be endowed with the strong convergence topology, which turns $\mathcal{L}\left(X,Y\right)$ into a normed space. Note that if $Y$ is complete, then so is $\mathcal{L}\left(X,Y\right)$.

When $X=Y$, then we simply write $\mathcal{L}(X)$. It is clear that $\mathcal{L}\left(X\right)$ is an associative algebra with unity. If $X$ is a normed space, then $\mathcal{L}\left(X\right)$ becomes a unital normed algebra, and if $X$ is complete, then $\mathcal{L}\left(X\right)$ is a unital Banach algebra.

\subsection{Spaces of continuous and linear operators}\label{contlin}

Notable subsets of $\mathcal{L}\left(\ell_\infty\left(X\right),X\right)$, which we will be working with, are:

\begin{itemize}

\item $\mathcal{N}_X$, that is, the set of all continuous and linear operators from $\ell_{\infty}\left(X\right)$ to $X$ which are invariant under the forward shift operator on $\ell_{\infty}\left(X\right)$. In other words,
\begin{eqnarray*}
\mathcal{N}_{X}&:=&\left\{T\in \mathcal{L}\left(\ell_\infty\left(X\right),X\right): T\left(\left(x_n\right)_{n \in\mathbb{N}}\right)=T\left(\left(x_{n+1}\right)_{n \in\mathbb{N}}\right)\right.\\
&&\left.\text{for all }\left(x_n\right)_{n \in\mathbb{N}}\in \ell_\infty\left(X\right)\right\}.
\end{eqnarray*} 

\item $\mathcal{L}_X$, that is, the set of all continuous and linear operators from $\ell_{\infty}\left(X\right)$ to $X$ which are extensions of the limit function on $c\left(X\right)$. In other words, $$\mathcal{L}_{X}:=\left\{T\in \mathcal{L}\left(\ell_\infty\left(X\right),X\right): T|_{c\left(X\right)}=\lim \right\}.$$

\end{itemize}

Easy verifiable properties of the above sets are described in the following proposition, the details of its proof we spare to the reader.

\begin{proposition}\label{prim}
\mbox{}
\begin{enumerate}
\item $\mathcal{N}_X$ is a closed vector subspace of $\mathcal{L}\left(\ell_\infty\left(X\right),X\right)$ when $\mathcal{L}\left(\ell_\infty\left(X\right),X\right)$ is endowed with the pointwise convergence topology.
\item $\mathcal{N}_{X}=\left\{T\in \mathcal{L}\left(\ell_\infty\left(X\right),X\right): bps\left(X\right)\subseteq \ker\left(T\right) \right\}$.
\item If $T \in \mathcal{N}_{X}$, then $c_0\left(X\right)\subseteq \ker\left(T\right)$.
\item $\mathcal{L}_X$ is an affine subset of $\mathcal{L}\left(\ell_\infty\left(X\right),X\right)$, that is, if $T,S\in\mathcal{L}_X$ and $t\in \mathbb{R}$, then $tT+\left(1-t\right)S\in\mathcal{L}_X$.
\item $\mathcal{L}_X$ is closed in $\mathcal{L}\left(\ell_\infty\left(X\right),X\right)$
when $\mathcal{L}\left(\ell_\infty\left(X\right),X\right)$ is endowed with the pointwise convergence topology.
\item If $T\in\mathcal{L}_X$, then $\left\|T\right\|\geq \left\|T|_{c\left(X\right)}\right\|=\left\|\lim\right\|=1$.
\item $ \mathcal{L}_X =\left\{T\in \mathcal{L}\left(\ell_\infty\left(X\right),X\right): c_0\left(X\right)\subseteq \ker\left(T\right) \text{ and }T|_{\bf X}=\lim \right\}.$
\item $\mathcal{N}_X \cap \mathcal{L}_X =\left\{T\in \mathcal{L}\left(\ell_\infty\left(X\right),X\right): bps\left(X\right)\subseteq \ker\left(T\right) \text{ and }T|_{\bf X}=\lim\right\}.$
\item $\mathcal{HB}\left(\lim\right)=\mathcal{L}_X\cap \B_{\mathcal{L}\left(\ell_\infty\left(X\right),X\right)}=\mathcal{L}_X\cap \E_{\mathcal{L}\left(\ell_\infty\left(X\right),X\right)}$.
\end{enumerate}
\end{proposition}

The reader may have noticed that in the last item of the previous proposition we make use of $\mathcal{HB}\left(\lim\right)$, which stands for all the norm-$1$ Hahn-Banach extensions of the limit function on $c(X)$ to the whole of $\ell_\infty(X)$. In general, $\mathcal{HB}\left(T\right)$ consists of all norm-$1$ Hahn-Banach extensions of $T$ to a (given) superspace.

\begin{corollary}\label{BLcomp}
If $\mathcal{N}_X \cap \mathcal{L}_X \neq \varnothing$, then $X$ is complete.
\end{corollary}

\begin{proof}
Let $(x_n)_{n\in\N}$ be a Cauchy sequence in $X$. Consider any $T\in \mathcal{N}_X \cap \mathcal{L}_X$. We will show that $(x_n)_{n\in\N}$ converges to $T\left((x_n)_{n\in\N}\right)\in X$. Fix an arbitrary $\varepsilon >0$. There exists $n_{\varepsilon}\in\N$ such that if $p,q\geq n_\varepsilon$, then $\|x_p-x_q\|<\frac{\varepsilon}{\|T\|}$. Now, for all $p\geq n_\varepsilon$ we have that 
\begin{eqnarray*}
\left\|T\left((x_n)_{n\in\N}\right)-x_p\right\|&=& \left\|T\left((x_{n+p})_{n\in\N}\right)-x_p\right\|\\
&=& \left\|T\left((x_{n+p})_{n\in\N}\right)-T\left({\bf x_p}\right)\right\|\\
&=& \left\|T\left((x_{n+p}-x_p)_{n\in\N}\right)\right\|\\
&\leq & \|T\|\left\| \left(x_{n+p}-x_p\right)_{n\in\N}\right\|_\infty\\
&\leq&\varepsilon.
\end{eqnarray*}
\end{proof}

Later on, we will show that $\mathcal{N}_X\neq \{0\}$. Sufficient conditions for the non-emptiness of $\mathcal{L}_X$ will be given when assuring the existence of Banach limits, since every Banach limit is an extension of the limit function.

\begin{corollary}\label{prim3}
If $T\in\mathsf{S}_{\mathcal{N}_X}$, then $$\left\|T|_{c_0\left(X\right)\oplus \mathbb{R}\mathbf{x}}\right\|=\frac{\left\|T\left(\mathbf{x}\right)\right\|}{\left\|x\right\|}$$ for all $x\in X\setminus \left\{0\right\}$.
\end{corollary}

\begin{proof}
In the first place, recall that $c_0\left(X\right)\subseteq \ker\left(T\right)$ as pointed out in Proposition \ref{prim}. Fix an arbitrary $x\in X\setminus \left\{0\right\}$. Let $\left(y_n\right)_{n\in\mathbb{N}}\in c_0\left(X\right)$ and $\lambda \in \mathbb{R}$. Since $\left(y_n\right)_{n\in\mathbb{N}}$ converges to $0$, we trivially have that
\begin{equation}\label{ozua}
\left\|\left(y_n+\lambda x\right)_{n\in\mathbb{N}}\right\|_\infty=\sup_{n\in\mathbb{N}}\left\|y_n+\lambda x\right\|\geq \left\|\lambda x\right\|.
\end{equation}
Therefore, if $\left(y_n+\lambda x\right)_{n\in\mathbb{N}}\in\mathsf{B}_{c_0\left(X\right)\oplus\mathbb{R}\mathbf{x}}$, in virtue of Equation \eqref{ozua} we have that
\begin{equation*}
\left\|T\left(\left(y_n+\lambda x\right)_{n\in\mathbb{N}}\right)\right\| = \left|\lambda\right|\left\|T\left(\mathbf{x}\right)\right\|\\
\leq  \frac{\left\|T\left(\mathbf{x}\right)\right\|}{\left\|x\right\|}.
\end{equation*} In order to finish the proof it suffices to note that $\frac{\mathbf{x}}{\left\|x\right\|}\in \mathsf{S}_{c_0\left(X\right)\oplus\mathbb{R}\mathbf{x}}$.
\end{proof}

It is not hard to observe that $bps(X)\oplus \mathbf{X}$ is never dense in $\ell_\infty(X)$.

\begin{corollary}\label{prim2}
There exists $T\in \mathsf{S}_{\mathcal{N}_X}$ such that $\left\|T\left(\mathbf{x}\right)\right\|=0$ for all $x\in X$. If, in addition, $\dim\left(X\right)>1$, then $T$ can be constructed not to be onto.
\end{corollary}

\begin{proof}
Fix an arbitrary $\left(z_n\right)_{n\in\mathbb{N}}\in\ell_\infty(X)$ such that $$d\left(\left(z_n\right)_{n\in\mathbb{N}},bps(X)\oplus \mathbf{X}\right)>0.$$ By a corollary of the Hahn-Banach Extension Theorem we can find $f\in\mathsf{S}_{\ell_\infty(X)^*}$ such that $$f\left(bps(X)\oplus \mathbf{X}\right)=\{0\}$$ and $$f\left(\left(z_n\right)_{n\in\mathbb{N}}\right)=d\left(\left(z_n\right)_{n\in\mathbb{N}},bps(X)\oplus \mathbf{X}\right).$$ Now fix an arbitrary $y\in\mathsf{S}_X$ and define $$\begin{array}{rrcl} T:&\ell_\infty(X) & \to & X \\ & \left(y_n\right)_{n\in\mathbb{N}}&\mapsto &T\left(\left(y_n\right)_{n\in\mathbb{N}}\right)=f\left(\left(y_n\right)_{n\in\mathbb{N}}\right)y. \end{array}$$ Notice that $T$ is a norm-$1$ continuous linear operator such that $T|_{bps\left(X\right)}=0$ (so $T\in\mathsf{S}_{\mathcal{N}_X}$ in virtue of Proposition \ref{prim}) and $T|_{\mathbf{X}}=0$.
\end{proof}

The group of surjective linear isometries on $X$ will be denoted by $\mathcal{G}_X$. A natural left action can be defined of $\mathcal{G}_X$ on $\E_X$ which turns $\E_X$ into a $\mathcal{G}_X$-space (we refer the reader to the backmatter chapter for an extended survey on $G$-spaces):

\begin{equation}\label{transitive}
\begin{array}{rcl}
\mathcal{G}_X\times \mathsf{S}_{X}& \to & \mathsf{S}_{X}\\
\left(T,x\right)&\mapsto & T\left(x\right).
\end{array}
\end{equation}

It is well known that for all Hilbert spaces the previous action is transitive. On the other hand, notice that the action \eqref{transitive} is not free if and only if there exists a surjective linear isometry other than the identity with a non-zero fixed point (this is the case, for instance, of the Hilbert spaces with dimension $3$ or more).

A normed space is said to enjoy
\begin{itemize}
\item the approximation property provided that the identity operator $I$ is in the closure of the finite-rank operators on $X$, $\F(X)$, when $\Li(X)$ is endowed with the topology of uniform convergence on compact subsets; and
\item the bounded approximation property provided that there exists $\lambda\geq 1$ such that the identity operator $I$ is in the closure of the finite-rank operators on $X$ with norm less than or equal to $\lambda$, that is, $\F(X)\cap \lambda\B_{\Li(X)}$, when $\Li(X)$ is endowed with the topology of uniform convergence on compact subsets.
\end{itemize}

\subsection{Complementation, M-ideals and injectivity}

A subspace of a topological vector space is said to be (topologically) complemented if it is the range of a continuous linear projection. We remind the reader that a continuous linear projection on $X$ is an idempotent of $\mathcal{L}(X)$. Every complemented subspace $M$ is closed and verifies that there exists another closed subspace $N$ in such a way that $M\cap N=\{0\}$ and $X=M+N$. This condition is not sufficient to assure that $M$ is complemented, unless we are working in the class of Banach spaces.

A subspace $M$ of a normed space $X$ is said to be an $\mathsf{L}_p$-summand subspace of $X$ provided that there exists another subspace $N$ of $X$ with $X=M\oplus_p N$, for $1\leq p \leq \infty$. $\mathsf{L}_\infty$-summands and $\mathsf{L}_1$-summands are sometimes called $\mathsf{M}$-summands and $\mathsf{L}$-summands, respectively. It is not hard to check that all $\mathsf{L}_p$-summand subspaces must be closed, and thus they will be complemented in Banach spaces.

A vector $x\in\mathsf{S}_X$ is called an $\mathsf{L}_p$-summand vector of $X$ when $\mathbb{R}x$ is an $\mathsf{L}_p$-summand subspace of $X$.

\begin{lemma}\label{ayuda1}
Assume that $M$ and $N$ are topologically complemented in $X$. Then:
\begin{enumerate}
\item $bps\left(M\right)$ and $bps\left(N\right)$ are topologically complemented in $bps\left(X\right)$.
\item If, in addition, $X=M\oplus_\infty N$, then $bps\left(X\right)=bps\left(N\right)\oplus_\infty bps\left(M\right)$ when $bps\left(X\right)$ is endowed with the norm \eqref{99}.
\end{enumerate}
\end{lemma}

\begin{proof}
\mbox{}
\begin{enumerate}
\item Let $P:X\to M$ be a continuous linear projection with $\ker(P)=N$. It is not hard to check that
$$
\begin{array}{rcl}
bps\left(X\right) & \to & bps\left(M\right)\\
\left(x_n\right)_{n\in\mathbb{N}}&\mapsto &\left(P\left(x_n\right)\right)_{n\in\mathbb{N}}
\end{array}
$$
is well defined and a continuous linear projection of norm $\left\|P\right\|$ whose kernel is $bps\left(N\right)$.

\item Simply notice that if $\left(x_i\right)_{i\in\mathbb{N}}\in bps\left(X\right)$ and $x_i=m_i + n_i$ with $m_i\in M$ and $n_i\in N$ for all $i\in\mathbb{N}$, then
\begin{eqnarray*}
\left\|\left(x_i\right)_{i\in\mathbb{N}}\right\|&=&\sup_{k\in\mathbb{N}}\left\|\sum_{i=1}^k x_i\right\|\\
&=&\sup_{k\in\mathbb{N}} \left\|\sum_{i=1}^k m_i +\sum_{i=1}^k n_i\right\|\\
&=& \sup_{k\in\mathbb{N}}\left(\max\left\{\left\|\sum_{i=1}^k m_i\right\|, \left\|\sum_{i=1}^k n_i\right\|\right\}\right)\\
&=&\max \left\{ \sup_{k\in\mathbb{N}}\left\|\sum_{i=1}^k m_i\right\|, \sup_{k\in\mathbb{N}} \left\|\sum_{i=1}^k n_i\right\|\right\}\\
&=& \max \left\{ \left\|\left(m_i\right)_{i\in\mathbb{N}}\right\|, \left\|\left(n_i\right)_{i\in\mathbb{N}}\right\|\right\}.
\end{eqnarray*}
\end{enumerate}
\end{proof}

\begin{lemma}\label{asinosva}
If $k_0$ is an isolated point of a compact Hausdorff topological space $K$, then $\chi_{\left\{k_0\right\}}$ is an $\mathsf{L}_\infty$-summand vector of $\mathcal{C}\left(K\right)$.
\end{lemma}

\begin{proof}
Let $k_0\in K$ be an isolated point. Notice that $\chi_{\left\{k_0\right\}}\in\mathsf{S}_{\mathcal{C}\left(K\right)}$, $\delta_{k_0}\in\mathsf{S}_{\mathcal{C}\left(K\right)^*}$, $\delta_{k_0}\left(\chi_{\left\{k_0\right\}}\right)=1$, and $$\mathcal{C}\left(K\right)=\mathbb{R}\chi_{\left\{k_0\right\}}\oplus \ker\left(\delta_{k_0}\right),$$ where $\chi_{\left\{k_0\right\}}$ denotes the characteristic function of $\left\{k_0\right\}$ and $\delta_{k_0}$ is the evaluation functional. Let $f\in \mathcal{C}\left(K\right)$ and write $$f=f\left(k_0\right)\chi_{\left\{k_0\right\}} + \left(f-f\left(k_0\right)\chi_{\left\{k_0\right\}} \right).$$ Then we have that
\begin{eqnarray*}
\left\|f\right\|_\infty &=& \max \left|f\right|\left(K\right)\\
&=&\max\left\{\left|f\left(k_0\right)\right|, \max \left|f\right|\left(K\setminus\left\{k_0\right\}\right)\right\}\\
&=&\max\left\{\left\|f\left(k_0\right)\chi_{\left\{k_0\right\}} \right\|_\infty, \left\|f-f\left(k_0\right)\chi_{\left\{k_0\right\}} \right\|_\infty\right\}.
\end{eqnarray*}
This shows that $\chi_{\left\{k_0\right\}}$ is an $\mathsf{L}_\infty$-summand vector of $\mathcal{C}\left(K\right)$.
\end{proof}

We remind the reader that an object $A$ of a category $\mathcal{C}$ is said to be $\mathcal{H}$-injective, where $\mathcal{H}$ is a class of morphisms of $\mathcal{C}$, provided that the following conditions holds: If $B,C\in\ob(\mathcal{C})$, $h\in\mor(B,C)\cap\mathcal{H}$ and $g\in\mor(B,A)$, then there is $f\in\mor(C,A)$ such that $f\circ h =g$. When $\mathcal{H}=\mon(\mathcal{C})$, the class of monomorphisms of $\mathcal{C}$, we call $A$ an injective object.

As expected, injective Banach spaces are injective objects of the category of Banach spaces. A Banach space $X$ is said to be $1$-injective if it is an injective Banach space and the operator $f$ preserves the norm of the operator $g$ in the definition right above.

We refer the reader to \cite{R,W} for a wide perspective on injective Banach spaces. A Banach space $X$ is $1$-injective if and only if $X$ is isometrically isomorphic to a $\mathcal{C}\left(K\right)$ for some extremally disconnected compact Hausdorff topological space $K$ (see \cite[Page 123]{D}).

\section{Measure theory}

\subsection{Spaces of bounded measurable functions}\label{vhs}

Recall that a Boolean algebra is a subset $\mathcal{A}$ of the power set $\Po(\Omega)$ of a non-empty subset $\Omega$ verifying that
\begin{itemize}
\item $\varnothing\in\mathcal{A}$,
\item if $A\in \mathcal{A}$, then $\Omega\setminus A\in \mathcal{A}$, and
\item if $\mathcal{C}\subseteq\mathcal{A}$ is finite, then $\cap\;\mathcal{C}\in\mathcal{A}$.
\end{itemize}
If the last condition holds for countable subsets of $\mathcal{A}$, then we are talking about $\sigma$-Boolean algebras.

If $\mathcal{F}$ is a $\sigma$-Boolean algebra, then $\mathcal{B}\left(\mathcal{F}\right)$ stands for the Banach space of bounded, $\mathcal{F}$-measurable, scalar-valued functions equipped with the sup norm; $\mathcal{B}_{\mathtt{s}}\left(\mathcal{F}\right)$ denote the dense subspace of  $\mathcal{B}\left(\mathcal{F}\right)$ consisting of all simple functions. Observe (see \cite{Sch}) that $\mathcal{B}\left(\mathcal{F}\right)^*$ may be represented as the space of all scalar-valued finitely additive measures on $\mathcal{F}$ with bounded variation norm (note that throughout this book we only work with real spaces and a real-valued measure has bounded variation norm if and only if it is bounded).

According to \cite{Dvm,Sch}, if $\mathcal{F}$ is a $\sigma$-Boolean algebra, then the following theorems hold for $\mathcal{F}$:

\begin{itemize}

\item {\bf Vitali-Hahn-Saks (VHS).} A sequence $\left(\mu_n\right)_{n\in \mathbb{N}}$ in $\mathcal{B}\left(\mathcal{F}\right)^*$ so that $\left(\mu_n\left(A\right)\right)_{n\in \mathbb{N}}$ converges for every $A\in \mathcal{F}$ is uniformly strongly additive.

\item {\bf Grothendieck (G).} A sequence $\left(\mu_n\right)_{n\in \mathbb{N}}$ in $\mathcal{B}\left(\mathcal{F}\right)^*$, that converges weakly-star, converges weakly. In other words, $\mathcal{B}\left(\mathcal{F}\right)$ is a Grothendieck space.

\item {\bf Nikodym (N).} A family $\mathcal{M} \subseteq \mathcal{B}\left(\mathcal{F}\right)^*$, such that $\left\{\mu\left(A\right):\mu \in \mathcal{M}\right\}$ is bounded for every $A\in \mathcal{F}$, is uniformly bounded. In other words, $\mathcal{B}_{\mathtt{s}}\left(\mathcal{F}\right)$ is barrelled.

\end{itemize}

Following Schachermayer (see \cite{Sch}), we will say that a Boolean algebra $\mathcal{F}$ verifies the properties VHS, N, or G if the corresponding theorem holds on $\mathcal{F}$. According to Schachermayer (see \cite{Sch}), Diestel, Faires, and Huff prove in 1976 that a Boolean algebra has VHS if and only if it has N and G (see \cite{DFH}).

On the other hand, note that $\mathcal{B}\left(\mathcal{F}\right)$ may be naturally identified with $\mathcal{C}\left(T\right)$, where $T$ is the Stone space of $\mathcal{F}$. This identification makes $\mathcal{B}_{\mathtt{s}}\left(\mathcal{F}\right)$ correspond to $\mathcal{C}_0\left(T\right)$. The properties N and G can be reformulated in terms of $T$ as follows:
\begin{enumerate}
\item $\mathcal{F}$ enjoys N if and only if $\mathcal{C}_0\left(T\right)$, the subspace of $\mathcal{C}\left(T\right)$ of all finite-valued functions, is barrelled.
\item $\mathcal{F}$ enjoys G if and only if $\mathcal{C}\left(T\right)$ is a Grothendieck space.
\end{enumerate}

By $\phi\left(\mathbb{N}\right)$ we intend to denote the Boolean sub-algebra of $\mathcal{P}\left(\mathbb{N}\right)$ consisting of the finite/cofinite subsets of $\mathbb{N}$. As in \cite{AizGuPer} we will say that a Boolean sub-algebra $\mathcal{F}$ of $\mathcal{P}\left(\mathbb{N}\right)$ is a natural Boolean algebra if $\phi\left(\mathbb{N}\right)\subseteq \mathcal{F}$. If we let $T$ stand for the Stone space of a natural Boolean algebra $\mathcal{F}$, then $\mathcal{C}\left(T\right)$ may be linearly and isometrically identified with a closed subspace of $\ell_{\infty}$ containing $c_0$ in virtue of \cite{Sch}.

 There exists a non-reflexive Grothendieck closed subspace of $\ell_\infty$ containing $c_0$ but failing to have a copy of $\ell_\infty$. To show this, we remind the reader that Haydon (see \cite{Haydon}) constructs, by using transfinite induction, a natural Boolean algebra $\mathcal{F}$ containing $\left\{\left\{i\right\}:i\in \mathbb{N}\right\}$ and such that the corresponding $\mathcal{C}\left(T\right)$, when identified with a closed subspace of $\ell_{\infty}$ containing $c_0$, is Grothendieck and fails to have a copy of $\ell_{\infty}$.

\section{Convergence theory}

The most general concept of convergence takes place in the category of pre-ordered spaces through two mainstream objects: nets and filters. We remind the reader that a pre-ordered space is a pair $(A,\leq)$ such that $A$ is a non-empty set and $\leq$ is a pre-order on $A$, that is, a binary relation on $A$ which is reflexive and transitive. Notable subsets of pre-ordered sets are $$\uparrow a:=\{x\in A:a\leq x\}\text{ and }\downarrow a:=\{x\in A:x\leq a\}.$$ Other notable subsets are the so called coinitial and cofinal subsets, that is, $B$ is coinital (cofinal) in $A$ provided that for all $a\in A$ we have that $B\;\cap \downarrow a\neq \varnothing$ ($B\;\cap \uparrow a\neq \varnothing$).

If $(A,\leq)$ is a pre-ordered space, then the binary relation given by $$\mathcal{R}:=\{(a,b)\in A\times A: a\leq b\text{ and }b\leq a\}$$ is an equivalence relation on $A$. It is immediate to observe that $\leq$ induces a partial order on $A/\mathcal{R}$.

\subsection{Nets and filters}

A net is a map from a directed space into a non-empty set. A directed space is a pair $(D,\leq)$ where $D$ is a non-empty set and $\leq$ is an upward directed pre-order. We remind the reader that a pre-order is upward directed provided that for all $x,y$ there exists $z$ with $z\geq x$ and $z\geq y$. Trivial examples of directed spaces are the totally ordered sets.

Given a net $\gamma: D\to X$, a subnet of $\gamma$ is a net $\lambda: E\to Y$ in such a way that for all $d\in D$ there exists $e\in E$ with $\lambda\left(\uparrow e\right)\subseteq \gamma\left(\uparrow d\right)$. Equivalently, $E_d:=\left\{e\in E:\lambda\left(\uparrow e\right)\subseteq \gamma \left(\uparrow d\right)\right\}\neq \varnothing$ for every $d\in D$.

\begin{proposition}
If $\gamma: D\to X$ is a net, $\alpha: E\to D$ is increasing, and $\alpha(E)$ is cofinal in $D$, then $\lambda:=\gamma \circ \alpha:E\to X$ is a subnet of $\gamma$.
\end{proposition}

\begin{proof}
Fix $d\in D$. By hypothesis, there exists $e\in E$ such that $\alpha(e)\geq d$. As a consequence, $\alpha\left(\uparrow e\right)\subseteq\;\uparrow \alpha(e)\subseteq \;\uparrow d$. Finally, $\gamma\left(\alpha\left(\uparrow e\right)\right)\subseteq\gamma\left( \uparrow d\right)$.
\end{proof}

We will prove now an approach to the converse of the previous proposition. We will need to define the following sets, $E'_d:=\{e\in E_d:\lambda(e)=\gamma(d)\}$ and $D'=\{d\in D: E'_d\neq \varnothing\}$.

\begin{proposition}
Assume that $\gamma: D\to X$ is a net and $\lambda: E\to Y$ is a subnet.
\begin{enumerate}
\item If $D$ is well ordered, then $D'$ is cofinal in $D$.
\item If both $D$ and $E$ are well ordered, then there exists a subset $E'$ of $E$ and a map $\beta:D'\to E'$ such that $\lambda\circ\beta=\gamma$.
\end{enumerate}
\end{proposition}

\begin{proof}
\mbox{}
\begin{enumerate}
\item Fix $d_0\in D$. By hypothesis, we can find $d_1:=\min\{d\in \;\uparrow d_0: \text{ there is }e\in E_{d_0}\text{ with }\lambda(e)=\gamma(d)\}$. It is obvious that $d_1\geq d_0$. We will show that $d_1\in D'$. Indeed, let $e\in E_{d_0}$ with $\lambda(e)=\gamma(d_1)$. Note that $\lambda\left(\uparrow e\right)\subseteq \gamma\left(\uparrow d_0\right)$. If there exists $e_1\in\;\uparrow e$ and $d_2\in\; \uparrow d_0\setminus \uparrow d_1$ with $\lambda(e_1)=\gamma(d_2)$, then, by minimality, $d_1\leq d_2$ which contradicts the fact that $d_2\notin \;\uparrow d_1$. As a consequence, $\lambda\left(\uparrow e\right)\subseteq \gamma\left(\uparrow d_1\right)$. This implies that $e\in E'_{d_1}\neq\varnothing$ and thus $d_1\in D'$.

\item For every $d\in D'$, by hypothesis, there exists $e_d:=\min(E'_d)$. By the Axiom of Choice there exists a set $E'$ whose elements are exactly all the $e_d$'s. Now it is possible to define the map $$\begin{array}{rrcl} \beta:& D'&\to & E'\\&d&\mapsto&\beta(d):=e_d.\end{array}$$ Finally, $\lambda\circ\beta=\gamma$. Indeed, $\lambda\left(\beta(d)\right)=\lambda\left(e_d\right)
=\gamma\left(d\right)$ because $e_d\in E'_d$.
\end{enumerate}
\end{proof}

We remind the reader that a subset $F$ of a partially ordered set $P$ is said to be a filter provided that
\begin{itemize}
\item $F$ is proper and non-empty;
\item $F$ is downward directed, that is, for all $x,y\in F$ there exists $z\in F$ with $z\leq x$ and $z\leq y$;
\item $F$ is upward closed, that is, if $x\in F$, $y\in P$ and $x\leq y$, then $y\in F$.
\end{itemize}

Observe that if $P$ has a minimum element, then no filter can have it. Notable examples of filters are the so called $p$-principal filter for a fixed $p\in P$, $\uparrow p$. Every filter trivially verifies that $$F=\bigcup_{p\in F}\uparrow p.$$

It is also immediate to note that if a filter is finite, then it is contained in a $p$-principal filter.

Maximal elements of the set of filters are called ultrafilters. The existence of ultrafilters is guaranteed by the Zorn's Lemma. 

\begin{proposition}
$p$ is a minimal element if and only if $\uparrow p$ is an ultrafilter.
\end{proposition}

\begin{proof}
Let $F$ be a filter containing $\uparrow p$ and suppose to the contrary that we can find $y\in F\setminus \uparrow p$. Notice that $p,y\in F$ therefore there exists $z\in F$ with $z\leq p$ and $z\leq y$. By hypothesis, $z=p$ and thus $p\leq y$, which implies the contradiction that $y\in\;\uparrow p$. Conversely, if $\uparrow p$ is an ultrafilter and $y\leq p$, then $\uparrow p\subseteq \;\uparrow y$, so they must be equal by hypothesis and thus $y\in\;\uparrow p$, which means that $p\leq y$.
\end{proof}

As a direct consequence of the previous proposition we have that if $P$ is finite, then every ultrafilter is $p$-principal for $p$ minimal.

In case $B$ is a subset of $P$ verifying the first two conditions of the definition of filter, then we will say that $B$ is a filter base. Every filter base not containing the minimum of $P$ generates a filter: $$\mathcal{J}(B):=\{x\in P: \text{ exists }b\in B \text{ with }b\leq x\}.$$

If $X$ is a non-empty set, then the power set of $X$, $\mathcal{P}(X)$, can be partially ordered by the inclusion, so we can consider filters of $\mathcal{P}(X)$, or more simply, filters in $X$. Notice that a filter in $X$ cannot contain the empty set since otherwise it would not be a proper subset of $\mathcal{P}(X)$.

If $F$ is a filter in $X$, then $F$ is said to be free provided that $\cap F =\varnothing$. Filters which are not free are called principal. It is well known that:
\begin{itemize}
\item A filter $F$ is principal if and only if there is $A\in \Po(X)\setminus\{\varnothing\}$ so that $F\subseteq \;\uparrow A$.
\item An ultrafilter $F$ is principal if and only if there is $x\in X$ so that $F=\;\uparrow\{x\}$.
\item If $X$ is finite, then all ultrafilters are principal.
\item If $X$ is infinite, then a filter is free if and only if it contains the Frechet filter, that is, the filter given by $$\mathcal{F}_X:=\left\{A\subseteq X:X\setminus A\text{ is finite}\right\}.$$
\end{itemize}

\subsection{Convergence in topological spaces}\label{filters}

Recall that given a topological space $X$ and an element $x\in X$, the set of all neighborhoods of $x$, $\Ne_x$, is a filter in $X$ called the neighborhood filter or neighborhood system of $x$.

The most general setting convergence in topological spaces is defined in follows:
\begin{itemize}
\item A filter base $B$ in $X$ is said to be convergent to $x$ if $\Ne_x\subseteq \mathcal{J}(B)$.
\item If $f:M\to X$ is just a function and $B$ is a filter base in $M$, then $B\lim f = x$ means that the filter base $f(B)$ converges to $x$ in $X$.
\item If $\gamma : D\to X$ is a net, then $\lim \gamma$ refers to $\mathcal{B}\lim \gamma$ where $\mathcal{B}$ is the filter base given by $\mathcal{B}:=\{\uparrow d: d\in D\}$.
\end{itemize}

It is not hard to show that if $f:M\to X$ is just a function and $\mathcal{F}$ is an ultrafilter in $M$, then $\mathcal{J}\left(f\left(\mathcal{F}\right)\right)$ is an ultrafilter in $X$. On the other hand, it is almost immediate to check that a sequence in a topological space converges to a given element if and only if it converges to that element through the Frechet filter.

\begin{lemma}\label{engaya}
Let $\left(x_n\right)_{n\in\mathbb{N}}$ be a sequence in a topological space $X$ and consider an element $x\in X$.
\begin{enumerate}
\item If $\left(x_n\right)_{n\in\mathbb{N}}$ converges to $x$, then $\mathcal{U}\lim x_n = x$ for every free filter $\mathcal{U}$ of $\mathbb{N}$.
\item If $\mathcal{U}\lim x_n = x$ for some free filter $\mathcal{U}$ of $\mathbb{N}$, then $x\in \mathrm{cl}\left(\left\{x_n:n\in\mathbb{N}\right\}\right)$.
\end{enumerate}
\end{lemma}

\begin{proof}
\mbox{}
\begin{enumerate}
\item Obvious since every free filter of $\N$ contains the Frechet filter.
\item If $W$ is a neighborhood of $x$, then $\left\{n\in\mathbb{N}:x_n\in W\right\}\in \mathcal{U}$, which means that $\left\{n\in\mathbb{N}:x_n\in W\right\}$ is infinite.
\end{enumerate}
\end{proof}

\section{The motivation for this work}

This thesis dissertation is simply an ongoing search for solutions to classical problems in operator theory and in the theory of sequences and series.

\subsection{Summing multiplier spaces with the Bade property}

Our main interests in this topic are focused on proving or disproving the following conjecture:

\begin{conjecture}\label{badeq}
A closed subspace of $\ell_\infty$ has the Bade property if and only if it is the summing multiplier space associated to some series in a Banach space.
\end{conjecture}

Observe that the Corollary \ref{cnotB} already disproves the implication $\Rightarrow$ of Conjecture \ref{badeq}. And Theorem \ref{sepc2} disproves the converse implication $\Leftarrow$. Then a new question comes into play.

\begin{question}
For which series in Banach spaces does the corresponding associated summing multiplier space have the Bade property?
\end{question}

In Corollary \ref{scSS} we show that every sc series provides a positive answer to the previous question (even if the ambient space is simply a topological vector space).

\subsection{Generalizations of Banach limits}

There are mainly two questions on this topic. The first question reads as follows:

\begin{question}
Is there a conception of Banach limit for vector-valued bounded sequences that coincides with the original concept of Banach limit when restricted to scalar-valued bounded sequences?
\end{question}

We prove that the answer to the previous question is in fact positive (see Subsection \ref{motBL}). A second question then arises:

\begin{question}
Is there a vector-valued version of the Lorentz's intrinsic characterization of almost convergence?
\end{question}

With respect to this second question, we prove that one of the implications of the Lorentz's intrinsic characterization of almost convergence can be fully versioned for vector-valued sequences (see {\em 2} of Corollary \ref{aquitepillo3}).

We also show that the converse implication cannot be fully extended to the vector-valued case because there are Banach spaces free of vector-valued Banach limits (see Example \ref{kad}). This fact already implies a negative answer to this second question. However, we have accomplished partial extensions of the converse implication (see Section \ref{secLIC}).

\subsection{Almost convergence extension of the Orlicz-Pettis Theorem}

In summability theory it is of a great interest to obtain extensions/generalizations of the Orlicz-Pettis Theorem via different summability methods (see \cite{AAGPPF, AizArPer,AizPer,BuWu}). We list now several of these extensions:
\begin{itemize}
\item By means of subalgebras of $\mathcal{P} \left(\mathbb{N}\right)$ in \cite{AizGuPer}.
\item By means of a regular matrix summability method in \cite{AizGPPE,AizPESS}.
\end{itemize}

We refer the reader to \cite{Boos} for a wide perspective on general matrix summability methods. An unexpected characteristic of the almost convergence is that the almost summability is not representable by any matrix summability method (see \cite{Boos}). Thus, it is legitimate to raise the following question:

\begin{question}
Is there a version of the Orlicz-Pettis Theorem for the almost convergence?
\end{question}

We provide a positive solution to this question in Theorem \ref{ACOPT}.

\subsection{Almost convergence extension of the Hahn-Schur Theorem}

It is also of a great interest in summability theory to obtain extensions/generalizations of the Hahn-Schur Theorem via different summability methods (see \cite{Swartz1,Swartz2}). A nice generalization via uniform convergence is due to Swartz (see Subsection \ref{SwSchur}). Therefore, it is legitimate to raise the following question:

\begin{question}
Is there a version of the Hahn-Schur Theorem for the almost convergence?
\end{question}

We provide a positive solution to this question in Theorem \ref{teor} and in Theorem \ref{teor2}.

\chapter{The Krein-Milman property and the Bade property}

This chapter is a step forward on the ongoing search for the solutions to several open problems on the Krein-Milman property and the Bade property.

The most interesting problem related to the Krein-Milman property is on determining whether a Banach space enjoying the Krein-Milman property also has the Radon-Nikodym property. 

The main problem related to the Bade property is to determine when the summing multiplier space associated to a given series satisfies the Bade property.

Other problems related to these topics and similar others have been recently considered in \cite{GPbanach, GPafa,GPZ} and mostly deal with the structure of convex sets in topological vector spaces.

\section{The Krein-Milman property}

The origin of the well-known and long-standing open problem of the possible equivalence of the Krein-Milman property and the Radon-Nikodym property dates back to the Summer of 1973 when Lindenstrauss (see, for instance, \cite{FLP}) showed that every Banach space having the Radon-Nikodym property also enjoys the Krein-Milman property. In \cite{HM} the authors approach the above problem in the positive by proving that if a dual Banach space has the Krein-Milman property, then it also verifies the Radon-Nikodym property. We refer the reader to \cite[Chapter 5, Section 5.4]{Pi} for a proper discussion of the problem about the Radon-Nikodym property being equivalent to the Krein-Milman property.

Everything on this topic starts with the famous Krein-Milman Theorem (see \cite{KM}).

\begin{theorem}[Krein and Milman, 1940; \cite{KM}]
Each compact and convex subset of a Hausdorff locally convex topological vector space coincides with the closed convex hull of its extreme points.
\end{theorem}

This result motivated the definition of the Krein Milman property:
\begin{itemize}
\item A closed bounded convex set is said to have the weak Krein-Milman property if it has extreme points.
\item A closed bounded convex set is said to have the Krein-Milman property if it is the closed convex hull of its extreme points.
\item A topological vector space is said to have the (weak) Krein-Milmam property when every closed bounded convex subset enjoys the (weak) Krein-Milman property.
\end{itemize}

We remind the reader that a Hausdorff locally convex topological vector space is called subreflexive provided that for every bounded closed convex subset $C$ of $X$ the supporting functionals on $C$ are dense in $X^*$ endowed with the $C$-sup pseudo norm. Examples of subreflexive Hausdorff locally convex topological vector spaces are all the Banach spaces in virtue of the famous Bishop-Phelps Theorem.

\begin{theorem}
In the class of subreflexive Hausdorff locally convex topological vector spaces, the Krein-Milman property and the weak Krein-Milman property are equivalent.
\end{theorem}

\begin{proof}
Let $X$ be a subreflexive Hausdorff locally convex topological vector space with the weak Krein-Milman property. Let $C$ be a closed convex subset of $X$. Suppose to the contrary that there exists $c\in C\setminus \cco(\ext(C))$. By the Hahn-Banach Separation Theorem, there exists $f\in X^*$ such that $f(c)>\sup f\left(\cco(\ext(C))\right)=\sup f(\ext(C))$. By the subreflexivity of $X$, there exists $g\in X^*$ such that $\exp_g(C):=\{d\in C:g(d)=\sup g(C)\}\neq \varnothing$ and $$\left\|f|_C-g|_C\right\|_\infty<\varepsilon:=\frac{f(c)-\sup f\left(\ext(C)\right)}{2}.$$ By hypothesis, there exists $e\in\ext\left(\exp_g(C)\right)$. Notice that $\exp_g(C)$ is a face of $C$ and so $\ext\left(\exp_g(C)\right)\subseteq \ext(C)$. As a consequence, we reach the following contradiction:
\begin{eqnarray*}
g(e)&=&\sup g(C)\\
& \geq &g(c) \\
&>& f(c) - \varepsilon\\
 &=& \sup f(\ext(C))+\varepsilon\\
 & \geq & f(e)+\varepsilon\\
 & > &g(e)
\end{eqnarray*}
\end{proof}

Examples of topological vector spaces verifying the Krein-Milman property include all Banach spaces with the Radon-Nikodym property (in particular, all reflexive Banach spaces).

\subsection{Topological impact}

An immediate consequence of Theorem \ref{n} is that the Krein-Milman property forces the spaces that have it to be Hausdorff.

\begin{corollary}
If a topological vector space enjoys the weak Krein-Milman property, then it is Hausdorff.
\end{corollary}

\subsection{Algebraic impact}

In \cite[Lemma 5.5]{GPZ} it is shown that if $X$ is a topological vector space and $\left(e_n\right)_{n\in \mathbb{N}}\subset X$ is a linearly independent sequence, then the set $$M:= \left\{\sum_{n=1}^{\infty}\lambda_ne_n: \left(\lambda_n\right)_{n\in \mathbb{N}}\in c_{00}\text{ and }\left|\lambda_n\right|\leq \frac{1}{2^n}\text{ for all }n\in \mathbb{N}\right\}$$ satisfies the following:
\begin{itemize}
\item $M$ is absolutely convex but free of extreme points.
\item If $\left(e_n\right)_{n\in \mathbb{N}}$ is bounded, then $M$ is bounded.
\item If there exists a biorthogonal system $\left(e_i,e_i^*\right)_{i\in I} \subset X\times X^*$ such that $\left(e_n\right)_{n\in \mathbb{N}}\subset \left(e_i\right)_{i\in I}$, then $M$ is closed in $\mathrm{span}\left\{e_i:i\in I\right\}$.
\end{itemize}

\begin{theorem}\label{kmuc}
If $X$ is a first-countable Hausdorff locally convex topological vector space with the weak Krein-Milman property, then the cardinality of a Hamel basis of $X$ is uncountable.
\end{theorem}

\begin{proof}
Suppose to the contrary that the dimension of $X$ is countable. In accordance with Theorem \ref{ex2}, there exists a biorthogonal system $\left(e_n,e^*_n\right)_{n\in \mathbb{N}}\subseteq X \times X^*$ such that $X= \mathrm{span}\left\{e_n:n\in \mathbb{N}\right\}$. Since $X$ is first countable, there exists a countable nested basis of balanced and absorbing neighborhoods of $0$, so we may assume that $\left(e_n\right)_{n\in \mathbb{N}}$ is convergent to $0$ and hence bounded. Finally, by applying \cite[Lemma 5.5]{GPZ}, we deduce the contradiction that the set $M$ defined above does not have the Krein-Milman property.
\end{proof}

\subsection{A simplified reformulation}

The idea for an equivalent reformulation of the Krein-Milman property is not having to check all the bounded closed convex subsets (see \cite[Section 3]{GPafa}). For this we need to extend the concept of extreme point to non-convex sets.

\begin{theorem}
The following conditions are equivalent for a subreflexive Hausdorff locally convex topological vector space $X$:
\begin{enumerate}
\item $X$ has the Krein-Milman property.
\item If $M\subseteq X$ is closed, bounded, and convex and $x\in X\setminus M$, then $$\mathrm{ext}\left(\mathrm{co}\left(M \cup \left\{x\right\}\right)\right)\setminus \left\{x\right\} \neq \varnothing.$$
\item If $M\subseteq X$ is closed, bounded, and convex and $0\notin M$, then $\mathrm{ext}\left(\ba\left(M\right)\right) \neq \varnothing$.
\end{enumerate}
\end{theorem}

\begin{proof}
\mbox{}
\begin{enumerate}
\item[{\em 1} $\Rightarrow$ {\em 2}] Let $M\subseteq X$ be closed, bounded, and convex and consider $x\in X\setminus M$. Since $\mathrm{co}\left(M \cup \left\{x\right\}\right)$ is closed, bounded, and convex, by hypothesis we have that $$\mathrm{co}\left(M \cup \left\{x\right\}\right)=\overline{\mathrm{co}}\left( \mathrm{ext}\left(\mathrm{co}\left(M \cup \left\{x\right\}\right)\right) \right).$$ If $\mathrm{ext}\left(\mathrm{co}\left(M \cup \left\{x\right\}\right)\right)= \left\{x\right\}$, then $$\mathrm{co}\left(M \cup \left\{x\right\}\right)=\overline{\mathrm{co}}\left( \mathrm{ext}\left(\mathrm{co}\left(M \cup \left\{x\right\}\right)\right) \right)=\left\{x\right\},$$ which is impossible since $x\notin M$. As a consequence, $$\mathrm{ext}\left(\mathrm{co}\left(M \cup \left\{x\right\}\right)\right)\setminus \left\{x\right\} \neq \varnothing.$$
\item[{\em 2} $\Rightarrow$ {\em 3}] Let $M$ be a closed, bounded, and convex subset of $X$ not containing $0$. We have that $$\ba\left(M\right)= \mathrm{co}\left(M \cup \left\{0\right\}\right) \cup \mathrm{co}\left(-M \cup \left\{0\right\}\right).$$ By hypothesis, $$\mathrm{ext}\left(\mathrm{co}\left(M \cup \left\{0\right\}\right)\right)\setminus \left\{0\right\} \neq \varnothing,$$ so it only remains to apply Lemma \ref{superhelp}.
\item[{\em 3} $\Rightarrow$ {\em 1}] Let $M$ be a closed, bounded, and convex subset of $X$. By making a translation if necessary, we can assume without any loss of generality that $0\notin M$. By hypothesis we have that $\mathrm{ext}\left(\ba\left(M\right)\right) \neq \varnothing$. Finally, in virtue of Lemma \ref{superhelp}, we deduce that $\mathrm{ext}\left(M\right) \neq \varnothing$.
\end{enumerate}
\end{proof}

\section{The Bade property}

As we mentioned at the beginning of this chapter, a very interesting question is to determine for which series $\sum x_n$ does $\mathcal{S}\left(\sum x_n\right)$ have the Bade property. According to \cite{D2,DieJar}, the summing multiplier space associated to an uc series in a Banach space has the Bade property. As far as we know, these are the only known examples of series whose associated summing multiplier space has the Bade property.

We have to go back to Bade's dissertation in order to find the origins of the Bade property (see \cite{B}).

\begin{theorem}[Bade, 1950; \cite{B}]
A Hausdorff compact topological space $K$ is $0$-dimensional if and only if $\mathcal{C}\left(K\right)$ has the Bade property.
\end{theorem}

Based upon the previous definition, a semi-normed space $X$ is said to have the Bade property provided that $\mathsf{B}_X$ satisfies the Krein-Milman property, that is, $\mathsf{B}_X=\overline{\mathrm{co}}\left(\mathrm{ext}\left(\mathsf{B}_X\right)\right)$.

An example of a Banach space enjoying the Bade property but not the Krein-Milman property is $c_0$ with an strictly convex renorming (see \cite[Chapter VII. Theorem 1(a)]{D}).

\subsection{Topological impact}

The Bade property has a similar impact on the topology as the weak Krein-Milman property does.

If $X$ is a semi-normed space and $n \in X$ is so that $\left\|n\right\|=0$, then $\left\|m+n\right\|=\left\|m\right\|$ for all $m \in X$. Indeed, observe that $$\left\|m\right\|  = \left|\left\|m\right\|-\left\|-n\right\|\right| \leq \left\|m+n\right\|\leq \left\|m\right\| + \left\|n\right\| = \left\|m\right\|.$$

\begin{theorem}
If $X$ is a non-Hausdorff semi-normed space, then $\mathsf{B}_X$ is free of extreme points.
\end{theorem}

\begin{proof}
Define $N:=\left\{x\in X: \left\|x\right\|=0\right\}$. In accordance with Theorem \ref{n}, $N$ is a bounded closed vector subspace of $X$ whose induced topology is trivial and is topologically complemented with any subspace with which it is algebraically complemented. Let $M$ be an algebraical complement for $N$ in $X$. We will show now that $\mathsf{B}_X$ is free of extreme points. Let $x \in \mathsf{B}_X$. There are $m\in M$ and $n\in N$ such that $x=m + n$. By the observation right above, $\left\|x\right\|= \left\|m\right\|=\left\|m+2n\right\|$, so $m,m+2n\in \mathsf{B}_X$. Finally, $$x= \frac{1}{2}m+\frac{1}{2}\left(m + 2n\right),$$ so $x\notin \mathrm{ext}\left(\mathsf{B}_X\right)$.
\end{proof}

\begin{corollary}
Every semi-normed space with the Bade property is Hausdorff.
\end{corollary}

\subsection{Algebraic non-impact}

The Bade property has no algebraic impact whatsoever on the dimension like the weak Krein-Milman property does (recall Theorem \ref{kmuc}).

\begin{proposition}
Every countably infinite dimensional normed space $X$ can be equivalently renormed to enjoy the Bade property.
\end{proposition}

\begin{proof}
Since $X$ is separable, it is well known that $X$ admits an equivalent rotund renorming (see for instance \cite[Chapter VII. Theorem 1(a)]{D} or \cite{DGZ}). As a consequence, $X$ enjoys the Bade property endowed with this equivalent norm.
\end{proof}

\chapter{Multipliers}

Multipliers in the context of series refers to bounded scalar sequences which make a vector-valued sequence summable under multiplication.

The summing multiplier spaces are involved in the extensions/generalizations of the Orlicz-Pettis Theorem. And the multiplier spaces of summable sequences are involved in the extensions/generalizations of the Hahn-Schur Theorem.

\section{Summing multiplier spaces}

These spaces consist of the multipliers associated to a given series in a topological vector space, and like we said right above will be involved in extensions/generalizations of the Orlicz-Pettis Theorem.

\subsection{The spaces $\mathcal{S}\left(\sum x_i\right)$ and $\mathcal{S}_w\left(\sum x_i\right)$}

The reader should be familiar with the following spaces of sequences associated to a given series, introduced for the first time in \cite{PBA}.

\begin{definition}[P\'erez-Fern\'andez et al., 2000; \cite{PBA}]
The summing multiplier spaces associated to a series $\sum x_i$ in a topological vector space are defined as:
\begin{itemize}
\item $\mathcal{S}\left(\sum x_i\right):=\left\{(a_i)_{i\in\mathbb{N}} \in \ell_\infty:\sum a_i x_i\text{ converges}\right\}$.
\item $\mathcal{S}_w\left(\sum x_i\right):=\left\{(a_i)_{i\in\mathbb{N}} \in \ell_\infty:\sum a_i x_i\text{ $w$-converges}\right\}.$
\end{itemize}
\end{definition}

The elements of such spaces are called multipliers and this is why these spaces are called summing multiplier spaces. Obviously, $$c_{00}\subseteq \mathcal{S}\left(\sum x_i\right)\subseteq \mathcal{S}_w\left(\sum x_i\right).$$

\begin{proposition}
If $X$ is complete and $(x_i)_{i\in\N}$ is bounded, then $\ell_1\subseteq \mathcal{S}\left(\sum x_i\right)$.
\end{proposition}

\begin{proof}
If suffices to notice that if $(b_i)_{i\in\N}\in \ell_1$, then $$\sum\left\|b_i x_i\right\|\leq \left\|(b_i)_{i\in\N}\right\|_1\left\|(x_i)_{i\in\N}\right\|_\infty,$$ which means that $\sum b_i x_i$ is ac.
\end{proof}

We refer the reader to \cite{D2,DieJar,APComenia,PBA} for a wider perspective on the above spaces, where the following results can be found:

\begin{itemize}
\item A normed space is complete if and only if for every series $\sum x_i$ the following conditions are equivalent:
\begin{itemize}
\item $\sum x_i$ is wuC.
\item $\mathcal{S}\left(\sum x_i\right)$ is complete.
\item $c_0\subseteq \mathcal{S}\left(\sum x_i\right)$.
\item $\mathcal{S}_w\left(\sum x_i\right)$ is complete.
\item $c_0\subseteq \mathcal{S}_w\left(\sum x_i\right)$.
\end{itemize}
\item A series $\sum x_n$ in a Banach space is cc and wuC if and only if $c\subseteq \mathcal{S}\left(\sum x_n\right) \subsetneq \ell_\infty$.
\item A series $\sum x_n$ in a Banach space is uc if and only if ${\mathcal S}\left(\sum x_n\right)=\ell_{\infty}$.
\end{itemize}

\subsection{The space $\mathcal{S}_{w^*}\left(\sum x^*_i\right)$}

If we consider the dual of a topological vector space endowed with the $w^*$-topology, then we can define $w^*$-summing multiplier spaces.

\begin{definition}[P\'erez-Fern\'andez et al., 2000; \cite{PBA}]
The summing multiplier space associated to a series $\sum x^*_i$ in the dual of a topological vector space is defined as $$\mathcal{S}_{w^*}\left(\sum x^*_i\right):=\left\{(a_i)_{i\in\mathbb{N}} \in \ell_\infty:\sum a_i x^*_i\text{ $w^*$-converges}\right\}.$$
\end{definition}

Obviously, $$ \mathcal{S}\left(\sum x^*_i\right)\subseteq \mathcal{S}_w\left(\sum x^*_i\right)\subseteq \mathcal{S}_{w^*}\left(\sum x^*_i\right).$$

\subsection{Extremal properties}

Throughout the whole of this section, $X$ will stand for a topological vector space unless otherwise stated. And if $\sum x_n$ is a series in $X$, then we will let $\mathcal{S}:=\mathcal{S}\left(\sum x_n\right)$.

\begin{lemma}\label{ext}
$\mathrm{ext}\left (\mathsf{B}_\mathcal{S}\right) = \mathsf{B}_\mathcal{S}\cap \mathrm{ext}\left (\mathsf{B}_{\ell_{\infty}}\right)$.
\end{lemma}

\begin{proof}
Obviously, $\mathrm{ext}\left (\mathsf{B}_\mathcal{S}\right) \supseteq \mathsf{B}_\mathcal{S}\cap \mathrm{ext}\left (\mathsf{B}_{\ell_{\infty}}\right)$. Let $\left(\varepsilon_n\right)_{n\in \mathbb{N}} \in \mathrm{ext}\left (\mathsf{B}_\mathcal{S}\right)$ and assume that there exists $n_0 \in \mathbb{N}$ such that $\left|\varepsilon_{n_0}\right| < 1$. Take $\delta := \frac{1 - \varepsilon_{n_0}}{2} > 0$ and consider the sequences $\left(a_n\right)_{n\in \mathbb{N}}$ y $\left(b_n\right)_{n\in \mathbb{N}}$ defined by:
$$a_n:= \left\{ \begin{array}{lll}
                \varepsilon_n &{\rm if}& n\neq n_0,\\
                \varepsilon_{n_0}+\delta &{\rm if}& n=n_0,
                \end{array}
        \right.
                $$
and
$$b_n:= \left\{ \begin{array}{lll}
                \varepsilon_n &{\rm if}& n\neq n_0,\\
                \varepsilon_{n_0}-\delta &{\rm if}& n=n_0.
                \end{array}
        \right.
                $$Finally, observe that $\left(a_n\right)_{n\in \mathbb{N}},\left(b_n\right)_{n\in \mathbb{N}}\in \mathsf{B}_{\mathcal{S}}$ and $$\left(\varepsilon_n\right)_{n\in \mathbb{N}} = \frac{1}{2}\left(a_n\right)_{n\in \mathbb{N}}+\frac{1}{2}\left(b_n\right)_{n\in \mathbb{N}}.$$
\end{proof}


\begin{lemma}\label{ccS}
The following conditions are equivalent:
\begin{enumerate}
\item $\sum x_n$ is convergent.
\item $\ext\left(\mathsf{B}_c\right) \subseteq \ext \left(\mathsf{B}_{\mathcal{S}}\right)$.
\item $\ext\left(\mathsf{B}_c\right) \cap \ext \left(\mathsf{B}_{\mathcal{S}}\right)\neq \varnothing$.
\end{enumerate}
\end{lemma}

\begin{proof}
\mbox{}
\begin{enumerate}
\item[{\em 1} $\Rightarrow$ {\em 2}] Observe that $\sum -x_n$ is also convergent. If $\left(\varepsilon_n\right)_{n\in \mathbb{N}}\in \ext\left(\mathsf{B}_c\right)$, then one of the two sets $\left\{n\in\mathbb{N}:\varepsilon_n=1\right\}$ and $\left\{n\in\mathbb{N}:\varepsilon_n=-1\right\}$ must be finite. In either case we have that $\left(\varepsilon_n\right)_{n\in \mathbb{N}}\in \ext\left(\mathsf{B}_{\mathcal{S}}\right)$.

\item[{\em 2} $\Rightarrow$ {\em 3}] Obvious.

\item[{\em 3} $\Rightarrow$ {\em 1}] Let $\left(\varepsilon_n\right)_{n\in \mathbb{N}}\in \ext\left(\mathsf{B}_c\right) \cap \ext \left(\mathsf{B}_{\mathcal{S}}\right)$. Again one the two sets $\left\{n\in\mathbb{N}:\varepsilon_n=1\right\}$ and $\left\{n\in\mathbb{N}:\varepsilon_n=-1\right\}$ must be finite. So we deduce that either $\sum x_n$ or $\sum -x_n$ is convergent.
\end{enumerate}
\end{proof}

\begin{theorem}\label{ccSS}
If $\sum x_n$ is convergent, then:
\begin{enumerate}
\item $c\subseteq \mathrm{cl}\left(\mathcal{S}\right)$.
\item If, in addition, $\sum x_n$ has a non-trivial convergent subseries, then $\ext\left(\mathsf{B}_c\right) \subsetneq \ext \left(\mathsf{B}_{\mathcal{S}}\right)$.
\end{enumerate}
\end{theorem}

\begin{proof}
\mbox{}
\begin{enumerate}

\item In accordance to Lemma \ref{ccS}, we have that $\ext\left(\mathsf{B}_c\right) \subseteq \ext \left(\mathsf{B}_{\mathcal{S}}\right)$, therefore $$\mathsf{B}_c=\overline{\mathrm{co}}\left(\ext\left(\mathsf{B}_c\right)\right)\subseteq \overline{\mathrm{co}}\left(\ext\left(\mathsf{B}_{\mathcal{S}}\right)\right)\subseteq \mathrm{cl}\left(\mathcal{S}\right),$$ in virtue of the fact that $c$ has the Bade property (see \cite{Bade}).
\item There exists $N\subset \mathbb{N}$ infinite such that $M:= \mathbb{N} \setminus N$ is also infinite and $\sum_{n\in N}x_n$ is convergent. Observe that in this situation $\sum_{n\in M} x_n$ is also convergent. Consider the sequences $\left(\chi_N\left(n\right)\right)_{n\in\mathbb{N}}$ and $\left(\chi_M\left(n\right)\right)_{n\in\mathbb{N}}$ where $\chi_N$ and $\chi_M$ denote the characteristic functions of $N$ and $M$, respectively. Note that $\left(\chi_N\left(n\right)-\chi_M\left(n\right)\right)_{n\in\mathbb{N}} \in \ext \left( \mathsf{B}_{\cal S}\right)\setminus \ext \left(\mathsf{B}_c\right)$.
\end{enumerate}
\end{proof}

\begin{corollary}\label{cnotB}
In a Banach space no series verifies that $c = \mathcal{S}$.
\end{corollary}

\begin{proof}
Suppose $c=\mathcal{S}$. According to Lemma \ref{ccS}, $\sum x_n$ is convergent and thus $\left(x_n\right)_{n\in \mathbb{N}}$ must tend to $0$, therefore a non-trivial subsequence $\left(x_{n_k}\right)_{k\in \mathbb{N}}$ can be found in such a way that $\sum x_{n_k}$ is ac and therefore uc. We apply now {\em 2} of Theorem \ref{ccSS} to reach a contradiction.
\end{proof}

\begin{lemma}\label{scS}
The following conditions are equivalent:
\begin{enumerate}
\item $\sum x_n$ is sc.
\item $\ext\left(\mathsf{B}_{\ell_{\infty}}\right) = \ext \left(\mathsf{B}_{\mathcal{S}}\right)$.
\item $\frac{1}{2}\left(\varepsilon_n\right)_{n\in \mathbb{N}}+\frac{1}{2}\left(\delta_n\right)_{n\in \mathbb{N}}\in \mathcal{S}$ for every $\left(\varepsilon_n\right)_{n\in \mathbb{N}},\left(\delta_n\right)_{n\in \mathbb{N}}\in\ext\left(\mathsf{B}_{\ell_\infty}\right)$.
\end{enumerate}
\end{lemma}

\begin{proof}
\mbox{}
\begin{enumerate}

\item[{\em 1} $\Rightarrow$ {\em 2}] Let $\left(\varepsilon_n\right)_{n\in \mathbb{N}}\in\ext\left(\mathsf{B}_{\ell_\infty}\right)$. Denote $$M_+:=\left\{n\in \mathbb{N}:\varepsilon_n=1\right\}\text{  and  }M_-:=\left\{n\in \mathbb{N}:\varepsilon_n=-1\right\}.$$ Observe that $$\sum_{n=1}^\infty \varepsilon_nx_n = \sum_{n\in M_+}x_n -\sum_{n\in M_-}x_n.$$

\item[{\em 2} $\Rightarrow$ {\em 3}] Obvious.

\item[{\em 3} $\Rightarrow$ {\em 1}] Let $M$ be a subset of $\mathbb{N}$. It suffices to consider $\left(\varepsilon_n\right)_{n\in \mathbb{N}}$ and $\left(\delta_n\right)_{n\in \mathbb{N}}$ given by
$$\varepsilon_n:=\left\{\begin{array}{rl} 1 & \text{ if }n\in M\\-1&\text{ if }n\in \mathbb{N}\setminus M\end{array}\right.$$ and $\delta_n :=1$ for all $n\in \mathbb{N}$. Observe that $$\sum_{n\in M}x_n=\sum_{n=1}^\infty \left(\frac{1}{2}\varepsilon_n+\frac{1}{2}\delta_n\right)x_n.$$

\end{enumerate}
\end{proof}

\begin{theorem}\label{scSS}
If $\sum x_n$ is sc, then $\ell_\infty = \mathrm{cl}\left(\mathcal{S}\right)$ and $\mathcal{S}$ enjoys the Bade property.
\end{theorem}

\begin{proof}
In accordance to Lemma \ref{scS}, we have that $\ext\left(\mathsf{B}_{\ell_{\infty}}\right) = \ext \left(\mathsf{B}_{\mathcal{S}}\right)$, then we have that $$\mathsf{B}_{\ell_{\infty}}=\overline{\mathrm{co}}\left(\ext\left(\mathsf{B}_{\ell_{\infty}}\right)\right)= \overline{\mathrm{co}}\left(\ext\left(\mathsf{B}_{\mathcal{S}}\right)\right)\subseteq \mathrm{cl}\left(\mathcal{S}\right),$$ in virtue of the fact that $\ell_{\infty}$ has the Bade property (see \cite{Bade}).
\end{proof}

\begin{lemma}\label{sepc}
\mbox{}
\begin{enumerate}
\item If $\ext \left(\mathsf{B}_{\cal S}\right)\neq \varnothing$, then $\left(x_n\right)_{n\in \mathbb{N}}$ converges to $0$. 
\item If $\left(x_n\right)_{n\in \mathbb{N}}$ has no subsequences converging to $0$, then $\mathcal{S}\subseteq c_0$.
\item If $\sum x_n$ has a subseries $\sum
x_{n_{k}}$ which is subseries convergent, then $\mathcal{S}$ is not separable.
\end{enumerate}
\end{lemma}

\begin{proof}
\mbox{}
\begin{enumerate}
\item Let $\left(\varepsilon_n\right)_{n\in \mathbb{N}} \in \ext \left(\mathsf{B}_{\cal S}\right)$. Then $\sum \varepsilon_n x_n$ is convergent, so $\lim_{n\to \infty} \varepsilon_n x_n=0$. Since there exists a filter base of balanced and absorbing neighborhoods of $0$, we deduce that $\lim_{n\to \infty}x_n=0$.
\item Consider $\left(a_n\right)_{n\in\mathbb{N}}\in \mathcal{S}\setminus c_0$. There exists a subsequence $\left(a_{n_k}\right)_{k\in \mathbb{N}}$ such that $a:=\inf \left\{\left|a_{n_k}\right|:k\in \mathbb{N}\right\}>0$. Since $\sum a_n x_n$ is convergent, we deduce that $\left(a_n x_n\right)_{n\in \mathbb{N}}$ converges to $0$ and so does $\left(a_{n_k} x_{n_k}\right)_{k\in \mathbb{N}}$. We will show that $\left( x_{n_k}\right)_{k\in \mathbb{N}}$ converges to $0$. Let $U$ be a balanced and absorbing neighborhood of $0$. There exists $k_0\in \mathbb{N}$ such that if $k\geq k_0$, then $a_{n_k} x_{n_k} \in aU$. Since $U$ is balanced, we conclude that $x_{n_k}\in U$ for every $k\geq k_0$.
\item It suffices to notice that $$\left\{\left(\chi_N\left(n\right)\right)_{n\in \mathbb{N}}: N\subseteq \left\{n_{k}:k\in \mathbb{N}\right\}\right\}$$ is an uncountable family of elements of $\mathcal{S}$ the distance between every two elements of which is $1$.
\end{enumerate}
\end{proof}

\begin{theorem}\label{sepc2}
If $X$ is a reflexive Banach space, then there exists a wuC series $\sum x_n$ in $X$ such that $\mathcal{S}$ does not have the Bade property.
\end{theorem}

\begin{proof}
By the reflexivity of $X$, there exists a weakly null sequence $(y_n)_{n\in\N}\subset \E_X$. Observe that $(y_n)_{n\in\N}$ is not a Cauchy sequence, therefore, by passing to a subsequence if necessary, we may assume that there exists $\varepsilon>0$ such that $\|y_{2n}-y_{2n-1}\|\geq \varepsilon$ for all $n\in\N$. Define $x_n:=y_{n+1}-y_n$ for all $n\in\N$. Note that $\sum_{n=1}^kx_n=y_{k+1}-y_1$, so $w\sum x_n = y_1$ and thus $\sum x_n$ is wuC. Finally, $(x_n)_{n\in\N}$ is not convergent to $0$ and thus $\ext({\cal S})= \varnothing$ in accordance to {\em 1} of Lemma \ref{sepc}.
\end{proof}

\begin{lemma}\label{acsc}
If $X$ is a Banach space and $(x_n)_{n\in\N}$ has a subsequence convergent to $0$, then $\sum x_n$ has a subseries which is uc.
\end{lemma}

\begin{proof}
We may assume without any loss of generality that $(x_n)_{n\in\N}$ has a subsequence $\left(x_{n_k}\right)_{k\in\N}$ in such a way that $\sum x_{n_k}$ is ac. Therefore, $\sum x_{n_k}$ is uc as well.
\end{proof}

\begin{theorem}
Assume that $X$ is a Banach space. Then:
\begin{enumerate}
\item If $\left(x_n\right)_{n\in\mathbb{N}}$ ha a subsequence converging to $0$, then $\mathcal{S}$ has a complemented subspace linearly isometric to $\ell_{\infty}$.
\item $\left(x_n\right)_{n\in\mathbb{N}}$ has no subsequences convergent to $0$ if and only if $\mathcal{S}$ is separable.
\end{enumerate}
\end{theorem}

\begin{proof}
\mbox{}
\begin{enumerate}
\item By applying Lemma \ref{acsc} we deduce that  $\left(x_n\right)_{n\in\mathbb{N}}$ has a subsequence $\left(x_{n_k}\right)_{k\in\mathbb{N}}$ whose associated series is uc. Then $$\left\{\left(\alpha_n\right)_{n\in\mathbb{N}}\in\ell_\infty:\alpha_n=0\text{ for all }n\in\mathbb{N}\setminus\left\{n_k:k\in\mathbb{N}\right\}\right\}$$ is a complemented subspace of $\mathcal{S}$ linearly isometric to $\ell_{\infty}$. 
\item It is a direct consequence of {\em 1} of this theorem and {\em 2} of Lemma \ref{sepc}.
\end{enumerate}
\end{proof}


Recall that no series $\sum x_n$ in a Banach space $X$ satisfies that $\mathcal{S}\left(\sum x_n\right)=c$ (see Corollary \ref{cnotB}). Other closed subspaces of $\ell_\infty$ with the Bade property can be constructed not be associated with any series.

\begin{itemize}
\item A subset $C$ of $\ell_\infty$ is said to satisfy the first-terms property when $C+c_{00}\subseteq C$.
\item A vector subspace of $\ell_\infty$ verifies the first-terms property if and only if it contains $c_{00}$.
\item $\mathcal{S}\left(\sum x_n\right)$ always verifies the first-terms property.
\end{itemize}

\begin{theorem}
If we let $A:=\left\{\left(\varepsilon_n\right)_{n\in \mathbb{N}}\in \mathrm{ext}\left(\mathsf{B}_{\ell_{\infty}}\right): \varepsilon_{2n}=1\text{ for all }n\in\mathbb{N}\right\}$, then $\cspan\left(A\right)$ is a closed subspace of $\ell_{\infty}$ enjoying the Bade property but failing the first-perms property. Even more, $\cspan\left(A\right)$ is linearly isometric to $\mathbb{R}\oplus_\infty \ell_\infty$.
\end{theorem}

\begin{proof}
In the first place, notice that the closed convex hull of $A$ and the closed absolutely convex hull of $A$ are $$\cco\left(A\right)= \left\{\left(\varepsilon_n\right)_{n\in \mathbb{N}}\in \mathsf{B}_{\ell_{\infty}}: \varepsilon_{2n}=1\text{ for all }n\in\mathbb{N}\right\}$$ and $$\caco\left(A\right)= \left\{\left(\varepsilon_n\right)_{n\in \mathbb{N}}\in \mathsf{B}_{\ell_{\infty}}: \left(\varepsilon_{2n}\right)_{n\in \mathbb{N}} \text{ is constant}\right\},$$ respectively. Therefore, $\caco\left(A\right)=\cco\left(A \cup -A\right)$ and both $\cco\left(A\right)$ and $\caco\left(A\right)$ are closed. On the other hand, $$\cspan\left(A\right)= \left\{\left(\varepsilon_n\right)_{n\in \mathbb{N}}\in \ell_{\infty}: \left(\varepsilon_{2n}\right)_{n\in \mathbb{N}} \text{ is constant}\right\}$$ is also closed and its unit ball is $\caco\left(A\right)$. Furthermore, $$\mathrm{ext}\left(\caco\left(A\right) \right) = A \cup -A,$$ therefore $\cspan\left(A\right)$ has the Bade property. Finally, in order to see that $\cspan\left(A\right)$ is not the space of convergence of any series it is sufficient to realize that $\cspan(A)$ does not contain $c_{00}$.
\end{proof}

\begin{theorem}
Let $$B:=\left\{\left(\varepsilon_n\right)_{n\in \mathbb{N}}\in \mathrm{ext}\left(\mathsf{B}_{\ell_{\infty}}\right): \text{ exists }n_0\in \mathbb{N}\text{ such that }\varepsilon_{2n}=1\text{ for all }n\geq n_0\right\}.$$ Then $\overline{\mathrm{span}}\left(B\right)$ verifies the following:
\begin{enumerate}
\item It contains $c$, satisfies the first-terms property, and enjoys the Bade property. Even more, $\overline{\mathrm{span}}\left(B\right)$ is linearly isometric to $c\oplus_\infty \ell_\infty$.
\item It is not the space of convergence associated to any series in a Banach space.
\end{enumerate}
\end{theorem}

\begin{proof}
Notice that
\begin{eqnarray*}
\overline{\mathrm{co}}\left(B\right)&=& \left\{\left(\varepsilon_n\right)_{n\in \mathbb{N}}\in \mathsf{B}_{\ell_{\infty}}: \left(\varepsilon_{2n}\right)_{n\in\mathbb{N}}\text{ converges to }1\right\},\\
\overline{\mathrm{aco}}\left(B\right)&=& \left\{\left(\varepsilon_n\right)_{n\in \mathbb{N}}\in \mathsf{B}_{\ell_{\infty}}: \left(\varepsilon_{2n}\right)_{n\in\mathbb{N}}\text{ is convergent}\right\},\\
\overline{\mathrm{span}}\left(B\right)&=& \left\{\left(\varepsilon_n\right)_{n\in \mathbb{N}}\in \ell_{\infty}:\left(\varepsilon_{2n}\right)_{n\in\mathbb{N}}\text{ is convergent}\right\}.
\end{eqnarray*}
Assume to the contrary that $\sum x_n$ is a series in a Banach space $X$ verifying that $\mathcal{S}\left(\sum x_n\right)=\overline{\mathrm{span}}\left(B\right)$. Let $M:=\left\{2n:n\in\mathbb{N}\right\}$. Note that $\left(\chi_M\left(n\right)\right)_{n\in \mathbb{N}}\in \overline{\mathrm{span}}\left(B\right)=\mathcal{S}\left(\sum x_n\right)$. Therefore $\sum x_{2n}$ is convergent and there exists a subsequence $\left(x_{2n_k}\right)_{k\in\mathbb{N}}$ such that $\sum_{k=1}^\infty x_{2n_k}$ is absolutely convergent and hence unconditionally convergent. As a consequence, $$\left\{\left(\alpha_n\right)_{n\in\mathbb{N}}\in\ell_\infty:\alpha_n=0\text{ for all }n\in\mathbb{N}\setminus\left\{2n_k:k\in\mathbb{N}\right\}\right\}$$ is a closed subspace of $\mathcal{S}\left(\sum x_n\right)=\overline{\mathrm{span}}\left(B\right)$ linearly isometric to $\ell_{\infty}$. This is a contradiction with the construction of $\overline{\mathrm{span}}\left(B\right)$.
\end{proof}

\subsection{Uniform summability}

By means of the uniform summability we will be able to characterize unconditional convergence in Banach spaces.

\begin{definition}\label{defcu}
A series $\sum x_n$ in a Banach space $X$ is called uniformly convergent (ufc) in ${\cal M} \subseteq\mathcal{S}\left( \sum x_n\right)$ if for every $\varepsilon >0$ there exists $k_0\in \mathbb{N}$ such that for every $k\geq k_0$ and every $\left(a_n\right)_{n\in\mathbb{N}} \in \cal M$ we have that $\left\|\sum_{n=k}^{\infty} a_nx_n\right\|< \varepsilon $.
\end{definition}

\begin{lemma}\label{c1tcolgado}
\mbox{}
\begin{enumerate}
\item If $\sum x_n$ is ufc in ${\cal M}$, then it is ufc in $ \aco\left({\cal M}\right)$.
\item If $(x_n)_{n\in\N}\in\ell_\infty(X)$ and $\sum x_n$ is ufc in ${\cal M}$, then it is ufc in $\cl_1\left(\aco\left({\cal M}\right)\right)$.
\end{enumerate} 
\end{lemma}

\begin{proof}
\mbox{}
\begin{enumerate}
\item It is a straightforward proof.
\item Fix $\varepsilon>0$. There exists $k_0\in \mathbb{N}$ such that for every $k\geq k_0$ and every $\left(a_n\right)_{n\in\mathbb{N}} \in \aco\left({\cal M}\right)$ we have that $\left\|\sum_{n=k}^{\infty} a_nx_n\right\|< \varepsilon/2 $. Fix an arbitrary $(b_n)_{n\in\N}\in \cl_1\left(\aco\left({\cal M}\right)\right)$. There exists $(a_n)_{n\in\N}\in  \aco\left({\cal M}\right)$ such that $(b_n-a_n)_{n\in\N}\in\frac{\varepsilon}{2\left\|(x_n)_{n\in \N}\right\|_\infty}\B_{\ell_1}$. Then for every $k\geq k_0$ we have that
\begin{eqnarray*}
\left\|\sum_{n=k}^{\infty} b_nx_n\right\|&\leq& \left\|\sum_{n=k}^{\infty} (b_n-a_n)x_n\right\| + \left\|\sum_{n=k}^{\infty} a_nx_n\right\|\\
& \leq & \left\|(x_n)_{n\in \N}\right\|_\infty \left\|(b_n-a_n)_{n\in\N}\right\|_1 + \frac{\varepsilon}{2}\\
&\leq& \varepsilon.
\end{eqnarray*}
\end{enumerate}
\end{proof}

\begin{theorem}\label{c2tcolgado}
The following conditions are equivalent for a convergent series $\sum x_n$ in a Banach space $X$:
\begin{enumerate}
\item $\sum x_n$ is uc.
\item $\sum x_n$ is ufc on $\ext \left(\mathsf{B}_{\mathcal{S}}\right)$.
\end{enumerate}
\end{theorem}

\begin{proof}
\mbox{}
\begin{enumerate}

\item[{\em 1} $\Rightarrow$ {\em 2}] Suppose to the contrary that the series $\sum x_n$ is not ufc in $\mathrm{ext} \left(\mathsf{B}_{\ell_{\infty}}\right)$. There exists $\delta>0$ such that for every $i\geq 1$ there are $j>i$ and $\left(\varepsilon^{\left(j\right)}_n\right)_{n\in \mathbb{N}} \in \ext
\left(\mathsf{B}_{\ell_{\infty}}\right) $ verifying that
$\left\|\sum_{n=j}^{\infty}\varepsilon^{\left(j\right)}_n x_n\right\| \geq \delta$. We can consider a strictly increasing sequence of indices
$$i_1<k_1<i_2<k_2< \cdots <i_j<k_j<i_{j+1}< \cdots$$
satisfying that
$$\left\|\sum_{n=i_j}^{\infty} \varepsilon^{\left(j\right)}_n  x_n\right\| > \delta
$$
and
\begin{equation}\label{ec1tcolgado}
\left\|\sum_{n=i_j}^{k_j} \varepsilon^{\left(j\right)}_n  x_n\right\| > \frac{\delta}{2},
\end{equation}
for every $j \in \mathbb{N}$. Next, we will construct a sequence $\left(\overline{\varepsilon}_n\right)_{n\in \mathbb{N}}\in \ext\left(\mathsf{B}_{\ell_{\infty}}\right)$ as follows:
\begin{itemize}
\item If $i\in \left[i_j, k_j\right]$ for some $j\in \mathbb{N}$, then $\overline{\varepsilon}_i= \varepsilon^{\left(j\right)}_i$.
\item The rest of the $\overline{\varepsilon}_i$'s can be either $1$ or $-1$.
\end{itemize}
Clearly $\sum  \overline{\varepsilon}_n  x_n$ is not a Cauchy series, so $\sum x_n$ is not uc.

\item[{\em 2} $\Rightarrow$ {\em 1}] Assume to the contrary that $\sum x_n$ is not uc. There exists $\left(\varepsilon_n\right)_{n\in \mathbb{N}}\in \ext \left(\mathsf{B}_{\mathcal{S}}\right)$ such that $\sum_{n=1}^{\infty}\varepsilon_n x_n$ does not converge. Then there exists $\delta >0$ such that for every $q\in\mathbb{N}$ we can find $p>q$ verifying that $\left\|\sum_{i=q}^{p}\varepsilon_i x_i\right\|>\delta$. Following an inductive process we can construct a strictly increasing sequence of naturals
$$p_1<p_2< \cdots <p_n < \cdots$$
such that for every $n\in\mathbb{N}$ we have that
$\left\|\sum_{i=p_n+1}^{p_{n+1}}\varepsilon_i x_i\right\|>\delta$. Since $\sum_{n=1}^\infty x_n$ is convergent, there exists $n_0\in \mathbb{N}$ such that for every $n\geq n_0$ we have that
$\left\|\sum_{i=n}^{\infty}x_i\right\|<\frac{\delta}{2}$. Next, for every $k\in\mathbb{N}$ we choose $p_{n_k} > \max \left\{n_0, k \right\}$ and define a sequence $\left(\alpha^{\left(k\right)}_n\right)_{n\in \mathbb{N}}$ as follows:
$$\alpha^{\left(k\right)}_n = \left\{\begin{array}{lll}
                        \varepsilon_n, & {\rm if}&
                                n\in\left\{ p_{n_k}+1, \ldots ,
                                p_{n_k+1}\right\}\\
                        1, & {\rm if} &  n\in \mathbb{N} \setminus
                                \left\{ p_{n_k}+1, \ldots , p_{n_k+1} \right\}.
                          \end{array}
                   \right.
$$
Observe that $\left(\alpha^{\left(k\right)}_n\right)_{n\in \mathbb{N}} \in \ext \left(\mathsf{B}_{\mathcal{S}}\right)$ for every $k\in \mathbb{N}$. Now, if $j=p_{n_k}+1$, then $j>k$ and
\begin{eqnarray*}
\left\|\sum_{i=j}^{\infty} \alpha^{\left(k\right)}_i x_i\right\|& =&
\left\|\sum_{i=p_{n_k}+1}^{p_{n_k+1}}\varepsilon_i x_i +
\sum_{i=p_{n_k+1}}^{\infty}x_i\right\|\\
& \geq &
         \left\|\sum_{i=p_{n_k}+1}^{p_{n_k+1}}\varepsilon_i x_i\right\| -
        \left\|\sum_{i=p_{n_k+1}}^{\infty}x_i\right\|\\
&>& \delta - \frac{\delta}{2} \\
&=&\frac{\delta}{2}.
\end{eqnarray*}
This is a contradiction.
\end{enumerate}
\end{proof}

The final corollary clarifies the relation between ufc series and the Bade property.

\begin{corollary}
Every ac series in a Banach space is ufc in $ \mathsf{B}_{\ell_{\infty}} $.
\end{corollary}

\begin{proof}
In the first place, keep in mind that every ac series $\sum x_n$ in a Banach space is uc and thus $\mathcal{S}\left( \sum x_n\right)=\ell_\infty$. In accordance to Theorem \ref{c2tcolgado}, $\sum x_n$ is ufc on $\ext \left(\mathsf{B}_{\ell_\infty}\right) $. By {\em 1} of Lemma \ref{c1tcolgado}, $\sum x_n$ is ufc on $\co\left(\ext \left(\mathsf{B}_{\ell_\infty}\right)\right) $. Since $\ell_\infty$ enjoys the Bade property, we only need to show that $\sum x_n$ is ufc on $\cco\left(\ext \left(\mathsf{B}_{\ell_\infty}\right)\right)$. Fix $\varepsilon>0$. There exists $k_0\in \mathbb{N}$ such that for every $k\geq k_0$ and every $\left(a_n\right)_{n\in\mathbb{N}} \in \co\left(\ext \left(\mathsf{B}_{\ell_\infty}\right)\right)$ we have that $\left\|\sum_{n=k}^{\infty} a_nx_n\right\|< \varepsilon/2 $. Fix an arbitrary $(b_n)_{n\in\N}\in \cco\left(\ext \left(\mathsf{B}_{\ell_\infty}\right)\right)$. There exists $(a_n)_{n\in\N}\in  \co\left(\ext \left(\mathsf{B}_{\ell_\infty}\right)\right)$ such that $(b_n-a_n)_{n\in\N}\in\frac{\varepsilon}{2\left\|(x_n)_{n\in \N}\right\|_1}\B_{\ell_\infty}$. Then for every $k\geq k_0$ we have that
\begin{eqnarray*}
\left\|\sum_{n=k}^{\infty} b_nx_n\right\|&\leq& \left\|\sum_{n=k}^{\infty} (b_n-a_n)x_n\right\| + \left\|\sum_{n=k}^{\infty} a_nx_n\right\|\\
& \leq & \left\|(x_n)_{n\in \N}\right\|_1 \left\|(b_n-a_n)_{n\in\N}\right\|_\infty + \frac{\varepsilon}{2}\\
&\leq& \varepsilon.
\end{eqnarray*}
\end{proof}

\section{Multiplier spaces of summable sequences}

This second section is merely introductory. Our results on this topic involve almost summability, so they appear in the fifth mainmatter chapter. However, we decided not to put this section in the framework and leave it here for the sake of completeness of this second mainmatter chapter.

We remind the reader that the multiplier spaces of summable sequences are the ones involved in the extensions/generalizations of the Hahn-Schur Theorem.

\subsection{The spaces $X\left(\mathcal{S}\right)$ and $X_w\left(\mathcal{S}\right)$}\label{SwSchur}

We have seen the summing multiplier spaces associated to a given series. Now it is the turn of the multiplier spaces of summable sequences, that is, vector-valued sequences which become summable by means of multipliers.

\begin{definition}[Swartz, 2000; \cite{Swartz1}]
If $X$ is a topological vector space and $\mathcal{S}$ is a subspace of $\ell_\infty$, then the ${\cal S}$-multiplier spaces of summable sequences of $X$ are defined as:
\begin{itemize}
\item $X\left(\mathcal{S}\right):=\left\{(x_i)_{i\in\mathbb{N}} \in X^\N:\sum a_i x_i\text{ converges}\right\}$.
\item $X_w\left(\mathcal{S}\right):=\left\{(x_i)_{i\in\mathbb{N}} \in X^\N:\sum a_i x_i\text{ $w$-converges}\right\}$.
\end{itemize}
\end{definition}

According to \cite{Swartz1,Swartz2}, if $X$ is a Banach space and $\mathcal{S}$ is a closed subspace of $\ell_\infty$, then the above spaces are complete when endowed with the norm \begin{equation}\label{ecu1} \left\|\left(x_i\right)_{i\in\N}\right\|_m:= \sup\left\{\left\|\sum_{i=1}^na_ix_i\right\|: n\in\N,\; |a_i|\leq 1,\; i\in\{1,\dots, n\} \right\}.\end{equation} Notice that $\left\|\left(x_i\right)_{i\in\N}\right\|_\infty \leq \left\|\left(x_i\right)_{i\in\N}\right\|_m$. In \cite{Swartz1} the author proves the following version of the Hahn-Schur Theorem.

\begin{theorem}[Swartz, 1983; \cite{Swartz1}]\label{thSw}
If $X$ is a Banach space and $(x^n)_{n\in\N}$ is a sequence in $X(\ell_\infty)$ such that $\lim_{n\to\infty}\sum_{i=1}^\infty a_ix^n_i$ exists in $X$ for each $(a_i)_{i\in\N}\in\ell_\infty$, then there exists $x^0\in X(\ell_\infty)$ such that $\lim_{n\to\infty}\|x^n-x^0\|_m=0$ in $X(\ell_\infty)$. 
\end{theorem}

Extensions of the previous theorem have been obtained in \cite{AizGPPE,AizGuPer,AizPer} for wuC series.

\subsection{The space $X^*_{w^*}\left(\mathcal{S}\right)$}

In a similar way we can define multiplier spaces of $w^*$-summable sequences in duals of topological vector spaces.

\begin{definition}[Swartz, 2000; \cite{Swartz1}]
If $X^*$ denotes the dual of a topological vector space and $\mathcal{S}$ is a subspace of $\ell_\infty$, then the ${\cal S}$-multiplier space of $w^*$-summable sequences of $X^*$ is defined as: $$X^*_{w^*}\left(\mathcal{S}\right):=\left\{(x^*_i)_{i\in\mathbb{N}} \in \left(X^*\right)^\N:\sum a_i x^*_i\text{ $w^*$-converges}\right\}.$$
\end{definition}

If $X$ is a Banach space and ${\cal S}$ is a closed subspace of $\ell_\infty$, then the above space is complete endowed with the norm given in \eqref{ecu1}.

\chapter{Vector-valued Banach limits}

The classical (scalar-valued) Banach limits are simply an application of the Hanh-Banach Extension Theorem to the limit function on the Banach space of all convergence sequences (see \cite[p. 34]{Banach}). There have been some partial attempts of extending the concept of Banach limit to the vector-valued case (see \cite{P, KS, KaS}). Other studies on scalar-valued Banach limits are due to Semenov and Sukochev (see \cite{SS}) where they consider invariant Banach limits and show the existence of a Banach limit $B$ such that $B=B\circ H$ for all $H$ in a subset of all bounded linear operators on $\ell_{\infty}$ of easy verifiable sufficient conditions.

\section{Scalar-valued and vector-valued Banach limits}

A characterization of scalar-valued Banach limits gives us the option to extend that concept to the scope of vector-valued sequences. The idea is to equivalently reformulate Definition \ref{bl} in order to remove the positivity.

\subsection{The classical definition of a Banach limit}

In his 1932 book \textit{Th\'eorie des op\'erations lin\'eaires} (see \cite[p. 34]{Banach}) Banach extended in a natural way the limit function defined on the space of all convergent sequences to the space of all bounded sequences.

\begin{definition}[Banach, 1932; \cite{Banach}]\label{bl}
A linear function $\varphi: \ell_{\infty} \to \R$ is called a Banach limit when:
\begin{enumerate}

\item $\varphi$ is positive: $\varphi\left( \left(x_n\right)_{n \in \mathbb{N}}\right) \geq 0$ if $ \left(x_n\right)_{n \in \mathbb{N}}\in \ell_{\infty}$ and $x_n\geq 0$ for all $n \in \mathbb{N}$,

\item $\varphi$ is invariant under forward shifts: $\varphi\left( \left(x_n\right)_{n \in \mathbb{N}}\right) = \varphi\left( \left(x_{n+1}\right)_{n \in \mathbb{N}}\right)$ if $ \left(x_n\right)_{n \in \mathbb{N}}\in \ell_{\infty}$, and

\item $\varphi$ is unital: $\varphi\left({\bf 1}\right)=1$, where ${\bf 1}$ denotes the constant sequence of term $1$. 
\end{enumerate}
\end{definition}

Banach noticed that a linear function $\varphi: \ell_{\infty} \to \R$ is a Banach limit if and only if $\varphi$ is invariant under the forward shift operator on $\ell_{\infty}$ and $\liminf_{n\to \infty} x_n \leq  \varphi\left( \left(x_n\right)_{n \in \mathbb{N}}\right) \leq \limsup_{n \to \infty} x_n$ for all $ \left(x_n\right)_{n \in \mathbb{N}}\in \ell_{\infty}$. As a consequence, every Banach limit $\varphi: \ell_{\infty} \to \R$ verifies that $\varphi |_c = \lim$ and $\left\|\varphi\right\|= 1$. Therefore, Banach limits are norm-$1$ Hahn-Banach extensions of the limit function on $c$ to $\ell_{\infty}$.

\subsection{Motivation for vector-valued Banach limits}\label{motBL}

The following simple lemma is crucial to extend the classical Banach limits to the vector-valued case.
\begin{lemma}\label{extension}
Consider a linear function $\psi: \ell_{\infty} \to \mathbb{R}$. If $\psi$ has norm $1$ and is unital, then $\psi$ is positive. As a consequence, $\psi$ is a Banach limit if and only if $\psi\in\mathsf{S}_{\mathcal{N}_\R} \cap \mathcal{L}_\R$.
\end{lemma}

\begin{proof}
Simply observe that if $\left(x_n\right)_{n\in\mathbb{N}}\in \mathsf{S}_{\ell_{\infty}}$ with $x_n\geq 0$ for all $n\in \mathbb{N}$, then $\left(1-x_n\right)_{n\in \mathbb{N}} \in \mathsf{B}_{\ell_{\infty}}$ and hence $$1 = \psi\left({\bf 1}\right)= \psi\left(\left(x_n\right)_{n\in\mathbb{N}}\right) + \psi\left(\left(1-x_n\right)_{n\in\mathbb{N}}\right)\leq \psi\left(\left(x_n\right)_{n\in\mathbb{N}} \right) +1.$$
\end{proof}

Now we are on the right position to define Banach limits for vector-valued sequences.

\begin{definition}\label{defBL}
The set of vector-valued Banach limits on a normed space $X$ is defined as $\BL\left(X\right):=\E_{\mathcal{N}_X} \cap \mathcal{L}_X$.
\end{definition}

For convenience, $\BL:=\BL\left(\mathbb{R}\right)$. As a direct consequence of Proposition \ref{prim} together with Corollary \ref{BLcomp}, we obtain the following corollary. We recall the reader that $\mathcal{HB}\left(\lim\right)$ denotes the set of all norm-$1$ Hahn-Banach extensions of the limit function on $c(X)$ to the whole of $\ell_\infty(X)$.

\begin{corollary}\label{Bprim}
\mbox{}
\begin{enumerate}
\item If $\BL(X)\neq \varnothing$, then $X$ is complete.
\item $\BL\left(X\right)=\left\{\varphi\in \mathsf{S}_{\mathcal{N}_X}:\varphi\left({\bf x}\right)=x\text{ for all }x\in X\right\}.$
\item $\BL\left(X\right)=\mathsf{B}_{\mathcal{N}_X} \cap \mathcal{L}_X$.
\item $\BL\left(X\right)=\mathcal{N}_X \cap \mathcal{HB}\left(\lim\right)$.
\item $\BL\left(X\right)$ is convex and closed when $\mathcal{L}\left(\ell_\infty\left(X\right),X\right)$ is endowed with the pointwise convergence topology.
\end{enumerate}
\end{corollary}

Every vector-valued Banach limit trivially induces a scalar-valued Banach limit. Indeed, if $T\in\BL(X)$, then $$\begin{array}{rcl} \ell_\infty&\to&\mathbb{R}\\ \left(\lambda_n\right)_{n\in\N}&\mapsto&
x^*\left(T\left(\left(\lambda_n x\right)_{n\in\N}\right)\right) \end{array}$$ where $x\in\E_X$ and $x^*\in\E_{X^*}$ can be chosen arbitrarily as long as $x^*(x)=1$.

\section{Existence of vector-valued Banach limits}

The class of normed spaces holding vector-valued Banach limits is indeed pretty large. It is closed under arbitrary $\ell_\infty$-sums (see \cite[Theorem 3]{GPPF}) and contains the Banach spaces $1$-complemented in their bidual (\cite[Corollary 1]{AGPPF}) and the $1$-injective Banach spaces (see \cite[Theorem 2.3]{GPbl}).

\subsection{Dual spaces}

We will begin this subsection by making use of an ultrafilter technique to produce vector-valued Banach limits in all dual Banach spaces (see \cite[Corollary 2.1]{AGPPF}). It is fair to mention that this technique was first used in \cite{KaS} although for different purposes.

\begin{theorem}\label{dual}
For every Banach space $X$ there exists a Banach limit $\varphi\in \BL\left(X^*\right)$ verifying that if $\left(x^*_n\right)_{n\in\mathbb{N}}\in\ell_\infty\left(X^*\right)$ is $w^*$-Cesaro convergent to $x^*\in X^*$, then $\varphi\left(\left(x^*_n\right)_{n\in\mathbb{N}}\right)=x^*$.
\end{theorem}

\begin{proof}
Consider any free ultrafilter $\mathcal{F}$ of $\mathbb{N}$. Define $$\begin{array}{rrcl}\varphi:&\ell_{\infty}\left(X^*\right) & \to & X^*\\&\left(x^*_n\right)_{n\in \mathbb{N}}&\mapsto&\displaystyle{\mathcal{F}\lim \frac{x^*_1+\cdots +x^*_n}{n}}.\end{array}$$
Observe that:
\begin{itemize}

\item $\varphi$ is well defined. Indeed, note that $\mathsf{B}_{X^*}\left(0,\left\|\left(x^*_n\right)_{n\in \mathbb{N}}\right\|_{\infty}\right)$ endowed with the $w^*$-topology is compact and Hausdorff. When the ultrafilter $\mathcal{F}$ is pushed forward by the sequence into $\mathsf{B}_{X^*}\left(0,\left\|\left(x^*_n\right)_{n\in \mathbb{N}}\right\|_{\infty}\right)$, it becomes a filter base whose induced filter is an ultrafilter. Since every ultrafilter in a compact Hausdorff topological space is convergent to a unique limit, we deduce that $$\displaystyle{\mathcal{F}\lim  \frac{x^*_1+\cdots +x^*_n}{n}}$$ exists and belongs to $\mathsf{B}_{X^*}\left(0,\left\|\left(x^*_n\right)_{n\in \mathbb{N}}\right\|_{\infty}\right)$.

\item $\varphi$ is linear, continuous, and has norm $1$. Indeed, from the previous item we deduce that $\left\|\varphi\right\|\leq 1$. In order to see that $\left\|\varphi\right\|=1$ it suffices to consider any constant sequence of norm $1$.

\item $\varphi|_{c\left(X^*\right)} =\lim$. Indeed, if $\left(x^*_n\right)_{n\in \mathbb{N}}$ is a sequence in $X^*$ converging to $x^*\in X^*$, then $$\left(\frac{x^*_1+\cdots + x^*_n}{n}\right)_{n\in \mathbb{N}}$$ converges to $x^*$, so $$\left(\frac{x^*_1+\cdots + x^*_n}{n}\right)_{n\in \mathbb{N}}$$ $w^*$-converges to $x^*$ and hence $$\mathcal{F}\lim \frac{x^*_1+\cdots +x^*_n}{n}=x^*$$ by {\em 1} of Lemma \ref{engaya}.

\item $\varphi$ is invariant under the shift operator. Indeed, for every $\left(x^*_n\right)_{n\in\mathbb{N}}\in\ell_{\infty}\left(X^*\right)$ we have that $$\varphi\left(\left(x^*_{n+1}-x^*_n\right)_{n\in\mathbb{N}}\right) = \mathcal{F}\lim\frac{x^*_{n+1}-x^*_1}{n} = 0$$ since $$\left(\frac{x^*_{n+1}-x^*_1}{n}\right)_{n\in\mathbb{N}}$$ converges to $0$.

\item If $\left(x^*_n\right)_{n\in\mathbb{N}}\in\ell_\infty\left(X^*\right)$ is $w^*$-Cesaro convergent to $x^*\in X^*$, then $$\varphi\left(\left(x^*_n\right)_{n\in\mathbb{N}}\right)=x^*.$$ Indeed, $$\left( \frac{x^*_1+\cdots +x^*_n}{n}\right)_{n\in\mathbb{N}}$$ is $w^*$-convergent to $x^*$, so $$\mathcal{F}\lim \frac{x^*_1+\cdots +x^*_n}{n}=x^*.$$

\end{itemize}
\end{proof}

\subsection{Spaces $1$-complemented in their bidual}

In the first place, we would like to make the reader notice that if $Y$ is a subspace of $X$ and $T\in\BL\left(X\right)$, then $T|_{\ell_\infty\left(Y\right)}\in \BL\left(Y\right)$ if and only if $T\left(\ell_\infty\left(Y\right)\right)\subseteq Y$.

The following trivial lemma, whose proof is obviously omitted, allows us to extend Theorem \ref{dual} to the spaces which are $1$-complemented in their bidual.

\begin{lemma}\label{1p}
If $Y$ is a $1$-complemented subspace of $X$ and $\varphi\in \BL\left(X\right)$, then $$
\begin{array}{rcl}
\ell_\infty\left(Y\right)&\to&Y\\
\left(y_n\right)_{n\in\mathbb{N}}&\mapsto&p\left(\varphi\left(\left(y_n\right)_{n\in\mathbb{N}}\right)\right)
\end{array}
$$
is a Banach limit on $Y$, where $p:X\to Y$ is a norm-$1$ linear projection.
\end{lemma}

\begin{corollary}\label{1-cc}
If a Banach space $X$ is $1$-complemented in $X^{**}$, then $\BL\left(X\right)\neq \varnothing$.
\end{corollary}

An example of a Banach space free of vector-valued Banach limits is due now.

\begin{example}\label{kad}
$\mathcal{BL}\left(c_0\right)=\varnothing$. Indeed, if $\varphi\in \mathcal{BL}\left(c_0\right)$, then the following map
$$
\begin{array}{rrcl}
p:&\ell_\infty &\to &c_0\\
&\left(y_n\right)_{n\in\mathbb{N}}&\to&\varphi\left(\left(y_1,0,0,\dots\right),\left(y_2,y_2,0,0,\dots\right),\left(y_3,y_3,y_3,0,0,\dots\right),\dots\right)
\end{array}
$$
is a bounded linear projection, which is not possible since $c_0$ is not complemented in $\ell_\infty$.
\end{example}

The next result in this section is both an approach to the converse of Corollary \ref{1-cc} and a generalization of Example \ref{kad}. Its last item were originally stated in \cite[Theorem 2.2]{izv} under the hypothesis of $X$ having a monotone Schauder basis. During the revision of \cite{izv} for MathSciNet, Prof. Plichko relaxed the previous hypothesis to the one of $X$ being separable and enjoying the bounded approximation property.

\begin{theorem}\label{megaostia}
Let $X$ be a Banach space.
\begin{enumerate}
\item If $X$ is complemented in $X^{**}$, then $\mathcal{N}_X\cap\mathcal{L}_X \neq \varnothing$.
\item If $X$ is separable and has the bounded approximation property and there exists $T\in \mathcal{L}_X$, then $X$ is complemented in its bidual.
\item If $X$ is separable and has the metric approximation property and there exists $T\in \mathsf{S}_{\mathcal{L}\left(\ell_\infty\left(X\right),X\right)}\cap\mathcal{L}_X$, then $X$ is $1$-complemented in its bidual.
\end{enumerate}
\end{theorem}

\begin{proof}
\mbox{}
\begin{enumerate}
\item It follows immediately via combining Theorem \ref{dual} together with Lemma \ref{1p} (assuming of course that $p$ is simply a projection).

\item Let $\left(e_n\right)_{n\in\mathbb{N}}$ be a Schauder basis for $X$ with dual basis $\left(e^*_n\right)_{n\in\mathbb{N}}$. Consider $T\in\mathcal{L}_X$ and define
$$
\begin{array}{rrcl}
P:&X^{**}&\to&X\\
&x^{**}&\mapsto& T\left(\left(\sum_{i=1}^n x^{**}\left(e_i^*\right)e_i\right)_{n\in\mathbb{N}}\right).
\end{array}
$$ It is not difficult to check that the previous map is a bounded linear projection (of norm $1$ in case $\left(e_n\right)_{n\in\mathbb{N}}$ is monotone and $\left\|T\right\|=1$). Indeed:

\begin{itemize}
\item In the first place, $P$ is well defined since for every $n\in\mathbb{N}$ we have that $$P_n^{**}\left(x^{**}\right)= \sum_{i=1}^n x^{**}\left(e_i^*\right)e_i$$ and $\left\|P_n^{**}\right\|=\left\|P_n\right\|$, where $P_n:X\to\mathrm{span}\left\{e_1,\dots,e_n\right\}$ is the usual projection.
\item In the second place, it is obvious that $P|_X$ is the identity on $X$, which already shows that $P$ is a bounded linear projection with $\left\|P\right\|\leq \left\|T\right\|M$ where $M$ is so that $\left\|P_n\right\|\leq M$ for all $n\in \mathbb{N}$.
\end{itemize}

\item Finally, if $\left(e_n\right)_{n\in\mathbb{N}}$ is monotone, then $M=1$ and hence $P$ has norm $1$ in case $T\in \mathsf{S}_{\mathcal{L}\left(\ell_\infty\left(X\right),X\right)}\cap\mathcal{L}_X$.
\end{enumerate}
\end{proof}

From Corollary \ref{1-cc} and Theorem \ref{megaostia} we directly obtain the following corollary.

\begin{corollary}
Let $X$ be a separable Banach space.
\begin{enumerate}
\item If $X$ has the metric approximation property, then the following conditions are equivalent:
\begin{itemize}
\item[$\bullet$] $\BL\left(X\right)\neq\varnothing$.
\item[$\bullet$] $\mathsf{S}_{\mathcal{L}\left(\ell_\infty\left(X\right),X\right)}\cap\mathcal{L}_X\neq \varnothing$.
\item[$\bullet$] $X$ is $1$-complemented in its bidual.
\end{itemize}

\item If $X$ has the bounded approximation property, then the following conditions are equivalent:
\begin{itemize}
\item[$\bullet$] $\mathcal{N}_X\cap\mathcal{L}_X\neq\varnothing$.
\item[$\bullet$] $\mathcal{L}_X\neq \varnothing$.
\item[$\bullet$] $X$ is complemented in its bidual.
\end{itemize}

\end{enumerate}
\end{corollary}

\begin{scholium}
Let $X$ be a Banach space.
\begin{enumerate}
\item If $X$ has a monotone Schauder basis, then the following conditions are equivalent:
\begin{itemize}
\item[$\bullet$] $\BL\left(X\right)\neq\varnothing$.
\item[$\bullet$] $\mathsf{S}_{\mathcal{L}\left(\ell_\infty\left(X\right),X\right)}\cap\mathcal{L}_X\neq \varnothing$.
\item[$\bullet$] $X$ is $1$-complemented in its bidual.
\end{itemize}

\item If $X$ has a Schauder basis, then the following conditions are equivalent:
\begin{itemize}
\item[$\bullet$] $\mathcal{N}_X\cap\mathcal{L}_X\neq\varnothing$.
\item[$\bullet$] $\mathcal{L}_X\neq \varnothing$.
\item[$\bullet$] $X$ is complemented in its bidual.
\end{itemize}

\end{enumerate}
\end{scholium}

\subsection{$1$-injective Banach spaces}

We refer the reader to Subsection \ref{contlin} for a wider perspective on $1$-injective spaces and related topics.

\begin{theorem}
If $X$ is a $1$-injective Banach space, then $\BL\left(X\right)\neq \varnothing$.
\end{theorem}

\begin{proof}
Consider the following operator:
\begin{equation}\label{ope}
\begin{array}{rcl}
\cl_{\ell_\infty\left(X\right)}\left( bps \left(X\right)\right) \oplus {\bf X}& \to &  X\\
\left(x_n\right)_{n\in\mathbb{N}}+ {\bf x}&\mapsto &x
\end{array}
\end{equation}
Notice that, in virtue of {\em 1} in Lemma \ref{pollazo}, we have that $$\left\|x\right\|=d\left({\bf x},\cl_{\ell_\infty\left(X\right)}\left( bps \left(X\right)\right)\right)\leq \left\|\left(x_n\right)_{n\in\mathbb{N}}+ {\bf x}\right\|_\infty$$ for all $x\in X$ and all $\left(x_n\right)_{n\in\mathbb{N}}\in \cl_{\ell_\infty\left(X\right)}\left( bps \left(X\right)\right)$. Finally, Proposition \ref{prim} together with Corollary \ref{Bprim} serve to assure that any norm-$1$ Hahn-Banach extension of the operator given in \eqref{ope} is an element of $\BL\left(X\right)$.
\end{proof}

We recall the reader about \cite[Corollary 1.3]{R} where an example of a $1$-injective Banach space is given which is not isomorphic to a dual space.
\subsection{$\ell_\infty$-sums}

This section is aimed at showing that arbitrary $\ell_\infty$-sums of spaces admitting vector-valued Banach limits also admit such Banach limits.

\begin{theorem}\label{infty}
Let $\left(X_i\right)_{i\in I}$ be an arbitrary family of Banach spaces such that there exists $\varphi_i\in\BL\left(X_i\right)$ for every $i\in I$. The map
$$
\begin{array}{rrcl}
\varphi:&\ell_\infty\left(\left(\prod_{i\in I}X_i\right)_\infty\right)&\to &\left(\prod_{i\in I}X_i\right)_\infty\\
&\left(\left(x^n_i\right)_{i\in I}\right)_{n\in\mathbb{N}}&\mapsto & \left(\varphi_i\left(\left(x^n_i\right)_{n\in\mathbb{N}}\right)\right)_{i\in I}
\end{array}
$$ verifies that $\varphi\in \BL\left(\left(\prod_{i\in I}X_i\right)_\infty\right)$.
\end{theorem}

\begin{proof}
\mbox{}
\begin{enumerate}
\item Let $\left(\left(x^n_i\right)_{i\in I}\right)_{n\in\mathbb{N}}\in \ell_\infty\left(\left(\prod_{i\in I}X_i\right)_\infty\right)$. Then:
\begin{eqnarray*}
\left\| \left(\varphi_i\left(\left(x^n_i\right)_{n\in\mathbb{N}}\right)\right)_{i\in I}\right\|_\infty&=& \sup_{i\in I}\left\|\varphi_i\left(\left(x^n_i\right)_{n\in\mathbb{N}}\right) \right\|\\
&\leq & \sup_{i\in I}\left\|\left(x^n_i\right)_{n\in\mathbb{N}} \right\|_\infty\\
&=&  \sup_{n\in\mathbb{N}}\left\|\left(x^n_i\right)_{i\in I} \right\|_\infty\\
&= &\left\| \left(\left(x^n_i\right)_{i\in I}\right)_{n\in\mathbb{N}}\right\|_\infty.
\end{eqnarray*}
This chain of inequalities shows that $\left\|\varphi\right\|\leq 1$. In the third part of this proof we will show that $\varphi$ is an extension of the limit function on $c\left(\left(\prod_{i\in I}X_i\right)_\infty\right)$ to the whole of $\ell_\infty\left(\left(\prod_{i\in I}X_i\right)_\infty\right)$, which proves that $\left\|\varphi\right\|=1$.
\item Let $\left(\left(x^n_i\right)_{i\in I}\right)_{n\in\mathbb{N}}\in \ell_\infty\left(\left(\prod_{i\in I}X_i\right)_\infty\right)$. Then:
\begin{eqnarray*}
\varphi\left(\left(\left(x^{n+1}_i\right)_{i\in I}\right)_{n\in\mathbb{N}} \right)&=& \left(\varphi_i\left(\left(x^{n+1}_i\right)_{n\in\mathbb{N}}\right)\right)_{i\in I}\\
&=&\left(\varphi_i\left(\left(x^n_i\right)_{n\in\mathbb{N}}\right)\right)_{i\in I}\\
&=&\varphi\left(\left(\left(x^n_i\right)_{i\in I}\right)_{n\in\mathbb{N}} \right).
\end{eqnarray*}
\item Let $\left(\left(x^n_i\right)_{i\in I}\right)_{n\in\mathbb{N}}\in c\left(\left(\prod_{i\in I}X_i\right)_\infty\right)$ converging to $\left(x_i\right)_{i\in I}\in \left(\prod_{i\in I}X_i\right)_\infty$. Observe that $\left(x^n_i\right)_{n\in\mathbb{N}}$ converges to $x_i$ for all $i\in I$. Then:
\begin{equation*}
\varphi\left(\left(\left(x^n_i\right)_{i\in I}\right)_{n\in\mathbb{N}} \right)= \left(\varphi_i\left(\left(x^n_i\right)_{n\in\mathbb{N}}\right)\right)_{i\in I}=
\left(x_i\right)_{i\in I}.
\end{equation*}
\end{enumerate}
\end{proof}

An immediate consequence of the previous theorem is the following corollary.

\begin{corollary}
If $\BL\left(X\right)\neq \varnothing$, then $\BL\left(\ell_\infty\left(\Lambda, X\right)\right)\neq \varnothing$ for every non-empty set $\Lambda$.
\end{corollary}

We would like to make the reader notice that Theorem \ref{infty} works if we substitute $\left(\prod_{i\in I}X_i\right)_\infty$ with the subspace of all $(x_i)_{i\in I}$ with null coordinates except for a countable number of them. However, Theorem \ref{infty} does not work with $\left(\oplus_{i\in I}X_i\right)_\infty$.

\section{Structure of the set of Banach limits}

Once we have found a large class of Banach spaces admitting vector-valued Banach limits, it is the turn now for studying the algebraic and topological structure of the set of vector-valued Banach limits. The first piece of information we find with respect to this issue is provided by Corollary \ref{Bprim}, where it is assured that $\BL\left(X\right)$ is convex and closed when $\mathcal{L}\left(\ell_\infty\left(X\right),X\right)$ is endowed with the pointwise convergence topology.

\subsection{Metric and algebraic structure}

One can easily notice that if $X$ and $Y$ are normed spaces and $T:X\to Y$ is a surjective linear isometry, then the existence of a Banach limit on $X$ implies the existence of a Banach limit on $Y$. Indeed, if $\varphi\in\BL\left(X\right)$, then the map
$$
\begin{array}{rrcl}
\psi:&\ell_\infty\left(Y\right)&\to &Y\\
&\left(y_n\right)_{n\in\mathbb{N}}&\mapsto & T\left(\varphi\left(\left(T^{-1}\left(y_n\right)\right)_{n\in\mathbb{N}}\right)\right)
\end{array}
$$
verifies that $\psi\in\BL\left(Y\right)$.

On the other hand, a left action of $\mathcal{G}_X$ on $\mathsf{S}_{\mathcal{N}_X}$ can be naturally defined by left-composition

\begin{equation}\label{action}
\begin{array}{rcl}
\mathcal{G}_X\times \mathsf{S}_{\mathcal{N}_X} & \to & \mathsf{S}_{\mathcal{N}_X}\\
\left(T,\varphi\right)&\mapsto & T\circ \varphi
\end{array}
\end{equation}

\begin{lemma}\label{act}
If $\varphi \in \BL\left(X\right)$ and $T\in\mathcal{G}_X$, then $T\circ \varphi\in \BL\left(X\right)$ if and only if $T$ is the identity on $X$.
\end{lemma}

\begin{proof}
Notice that $\left\|T\circ \varphi\right\|=1$ and $T\circ \varphi$ is invariant under the shift operator on $\ell_{\infty}\left(X\right)$. Assume that $T\circ \varphi\in \BL\left(X\right)$. Fix an arbitrary $x\in X$. Then $$T\left(x\right)=T\left(\varphi\left(\mathbf{x}\right)\right)=\left(T\circ \varphi\right)\left(\mathbf{x}\right)=x.$$
\end{proof}

Notice that the previous lemma shows that $\BL\left(X\right)$ is $\mathcal{G}_X$-free (we refer the reader to the backmatter chapter for a wide perspective on $G$-spaces).

\begin{theorem}
If $\BL\left(X\right)\neq \varnothing$, then the action \eqref{action} verifies the following:
\begin{enumerate}
\item It is not transitive if $\BL\left(X\right)$ has more than one element.
\item It is faithful.
\item It cannot be restricted to $\BL\left(X\right)$ as a left action.
\end{enumerate}
\end{theorem}

\begin{proof}
\mbox{}
\begin{enumerate}
\item Take $\varphi\neq \psi \in \BL\left(X\right)$. By Lemma \ref{act} we have that $T\circ\varphi \notin \BL\left(X\right)$ for all $T\in \mathcal{G}_X\setminus \left\{\I_X\right\}$, that is, $\mathcal{G}_X\varphi \cap \BL\left(X\right)=\left\{\varphi\right\}$. As a consequence, $\psi\notin\mathcal{G}_X\varphi$, which means that $\mathcal{G}_X\varphi\neq \mathsf{S}_{\mathcal{N}_X}$ and the action is not transitive.
\item Take any $T\in \mathcal{G}_X\setminus \left\{\mathrm{I}_X\right\}$. By hypothesis, there exists $\varphi\in \BL\left(X\right)$. By Lemma \ref{act} we have that $T\circ\varphi \notin \BL\left(X\right)$ and thus $T\circ\varphi \neq \varphi$.
\item It suffices to observe that $\mathcal{G}_X\BL\left(X\right)$ is never contained in $\BL\left(X\right)$ since $-\mathrm{I}_X\circ \varphi \notin \BL\left(X\right)$ for all $\varphi \in \BL\left(X\right)$ in virtue of Lemma \ref{act}.
\end{enumerate}
\end{proof}

\begin{theorem}\label{GBL}
 $$\mathcal{G}_X\BL\left(X\right)=\left\{T\in\mathsf{S}_{\mathcal{N}_X}: T\left(\mathbf{X}\right)=X\text{ and }\left\|T\left(\mathbf{x}\right)\right\|=\left\|x\right\|\text{ for every }x\in X\right\}.$$
\end{theorem}

\begin{proof}
\mbox{}
\begin{enumerate}
\item[$\subseteq$] Let $T\in \mathsf{S}_{\mathcal{N}_X}$ for which there exist $S\in \mathcal{G}_{X}$ and $\varphi \in \BL\left(X\right)$ such that $S\circ \varphi = T$. On the other hand, if $x\in X$, then $S\left(x\right)=\left(S\circ \varphi\right)\left(\mathbf{x}\right)=T\left(\mathbf{x}\right)$, which means that $T\left(\mathbf{X}\right)=X$ and that $$\left\|x\right\|=\left\|S\left(x\right)\right\|=\left\|\left(S\circ \varphi\right)\left(\mathbf{x}\right)\right\|=\left\|T\left(\mathbf{x}\right)\right\|.$$ 
\item[$\supseteq$] Let $T\in \mathsf{S}_{\mathcal{N}_X}$ satisfying that $T\left(\mathbf{X}\right)=X$ and $\left\|T\left(\mathbf{x}\right)\right\|=\left\|x\right\|$ for every $x\in X$. Define the following continuous linear operator:
$$
\begin{array}{rrcl}
S:&X&\to&X\\
&x&\mapsto&S\left(x\right):=T\left(\mathbf{x}\right).
\end{array}
$$ Notice that $S$ is a surjective linear isometry by hypothesis. In accordance to {\em 4} of Proposition \ref{prim} we have that $S^{-1}\circ T\in\BL\left(X\right)$. Finally, $$T=S\circ\left(S^{-1}\circ T\right)\in \mathcal{G}_X\BL\left(X\right).$$
\end{enumerate}
\end{proof}

As a direct consequence of Theorem \ref{GBL} and Corollary \ref{prim2} we obtain the following corollary.

\begin{corollary}
$\mathcal{G}_X\BL\left(X\right) \subsetneq \mathsf{S}_{\mathcal{N}_X}$.
\end{corollary}

\begin{theorem}\label{bothactions}
The action \eqref{action} is free if and only if so is the action \eqref{transitive}.
\end{theorem}

\begin{proof}
\mbox{}
\begin{enumerate}

\item[$\Rightarrow$] Suppose to the contrary that the action \eqref{transitive} is not free, which exactly means that there exists a surjective linear isometry $T:X\to X$ different from $\I_{X}$ with a fixed point $x\in \mathsf{S}_{X}$. Fix any $\left(a_{n}\right)_{n\in \mathbb{N}}\in \ell_{\infty}\left(X\right)$ with $d\left(\left(a_{n}\right)_{n\in \mathbb{N}}, bps\left(X\right) \right)=1$ (see Lemma \ref{pollazo}). A corollary of the Hahn-Banach Extension Theorem allows us to conclude that there exists $f\in\E_{\ell_\infty(X)^*}$ such that $$f\left(\left(a_{n}\right)_{n\in \mathbb{N}}\right)=d\left(\left(a_{n}\right)_{n\in \mathbb{N}}, bps\left(X\right) \right)=1$$ and $$f\left( bps\left(X\right) \right)=\{0\}.$$ Consider the following norm-$1$ continuous linear operator:
\begin{equation*}
\begin{array}{rrcl}
\psi: &\ell_\infty \left(X\right)&\to&X\\
&\left(z_n\right)_{n\in \mathbb{N}}&\mapsto & \psi\left( \left(z_n\right)_{n\in \mathbb{N}}\right):= f\left( \left(z_n\right)_{n\in \mathbb{N}}\right)x .
\end{array}
\end{equation*}
 It is clear that $T\circ \psi =\psi$ but $T\neq \mathrm{Id}_{X}$. This implies that the action \eqref{action} is not free either.

\item[$\Leftarrow$] Suppose to the contrary that the action \eqref{action} is not free. Then there are $T\in\mathcal{G}_X\setminus\left\{\I_X\right\}$ and $\varphi\in\mathsf{S}_{\mathcal{N}_X}$ such that $T\circ \varphi=\varphi$. At this point it suffices to consider any $\left(x_n\right)_{n\in\mathbb{N}}\in\ell_\infty\left(X\right)$ such that $\left\|\varphi\left(\left(x_n\right)_{n\in\mathbb{N}}\right)\right\|=1$, since $$T\left( \varphi\left(\left(x_n\right)_{n\in\mathbb{N}}\right)\right) = \varphi\left(\left(x_n\right)_{n\in\mathbb{N}}\right),$$ which implies that the action \eqref{transitive} is not free either.

\end{enumerate}
\end{proof}

Observe that if $X=\mathbb{R}$, then $\mathcal{G}_X=\left\{\I_X,-\I_X\right\}$ and thus the two actions considered previously are both free.

\subsection{Extremal structure}

The reader may quickly note that if $T: \ell_\infty\left(X\right)\to X$ is just a map, then:
\begin{itemize}
\item[$\bullet$] If $T$ is continuous and $T\left(\mathbf{x}\right)=x$ for all $x$ in some subset $A$ of $X$, then $T\left(\mathbf{x}\right)=x$ for all $x$ in the closure of $A$.
\item[$\bullet$] If $T$ is linear and $T\left(\mathbf{x}\right)=x$ for all $x$ in some subset $A$ of $X$, then $T\left(\mathbf{x}\right)=x$ for all $x$ in the linear span of $A$.
\item[$\bullet$] If $T\in\mathcal{L}\left(\ell_\infty\left(X\right),X\right)$ and $T\left(\mathbf{x}\right)=x$ for all $x$ in some subset $A$ of $X$, then $T\left(\mathbf{x}\right)=x$ for all $x$ in the closed linear span of $A$.
\end{itemize}

\begin{theorem}\label{cc}
Assume that $\BL\left(X\right)\neq \varnothing$.
\begin{enumerate}
\item If $X$ has the Bade property, then $\BL\left(X\right)$ is a face of $\mathsf{B}_{\mathcal{N}_X}$.
\item If there exists $T\in\mathcal{G}_X\setminus \left\{\I_X\right\}$ with a fixed point $x\in\mathsf{S}_X$, then $\BL\left(X\right)$ is not a convex component of $\mathsf{S}_{\mathcal{N}_X}$.
\end{enumerate}
\end{theorem}

\begin{proof}
\mbox{}
\begin{enumerate}
\item Let $\varphi,\psi\in \mathsf{S}_{\mathcal{N}_X}$ and $\alpha\in\left(0,1\right)$ such that $\alpha \varphi + \left(1-\alpha\right)\psi\in \BL\left(X\right)$. For every $x\in \mathsf{S}_X$ we have that $\alpha\varphi\left(\mathbf{x}\right)+\left(1-\alpha\right)\psi\left(\mathbf{x}\right)=x$. Now, if $x\in \mathrm{ext}\left(\mathsf{B}_X\right)$, then $\varphi\left(\mathbf{x}\right)=\psi\left(\mathbf{x}\right)=x$. Finally, by applying first the observation prior to this theorem and then {\em 1} of Corollary \ref{Bprim}, we deduce that $\varphi,\psi\in\BL\left(X\right)$.
\item Let $T\in\mathcal{G}_X\setminus \left\{\I_X\right\}$ with a fixed point $x\in\mathsf{S}_X$. Fix an arbitrary $\varphi \in \BL\left(X\right)$. In accordance to Lemma \ref{act} we have that $T\circ \varphi \notin \BL\left(X\right)$. Finally, in order to see that the drop $\mathrm{co}\left(\left\{T\circ \varphi\right\}\cup\BL\left(X\right)\right)$ is contained in $\mathsf{S}_{\mathcal{N}_X}$, it suffices to realize that
\begin{eqnarray*}
1&=&\left\|x\right\|\\
&=&\left\|\left(t\left(T\circ \varphi\right) + \left(1-t\right)\varphi\right)\left(\mathbf{x}\right)\right\|\\
&\leq & \left\| t\left(T\circ \varphi\right) + \left(1-t\right)\varphi\right\|\\
&\leq&1
\end{eqnarray*}
for all $t\in[0,1]$.
\end{enumerate}
\end{proof}

The following corollary, whose proof is omitted, becomes obvious if taken into account that any Hilbert space of dimension greater than or equal to $3$ enjoys the Bade property (since it is strictly convex) and has a surjective linear isometry other than the identity with non-zero fixed points.

\begin{corollary}
If $X$ is a Hilbert space with $\dim\left(X\right)\geq 3$, then $\BL\left(X\right)$ is a non-maximal face of $\mathsf{B}_{\mathcal{N}_X}$.
\end{corollary}

Bearing in mind the Banach-Stone Theorem, Lemma \ref{asinosva}, and {\em 2} of Theorem \ref{cc}, more spaces $X$ for which $\BL\left(X\right)$ is not a convex component of $\mathsf{S}_{\mathcal{N}_X}$ can be found. We spare the details of the proof to the reader.

\begin{theorem}
Assume that $K$ is a compact Hausdorff topological space with an isolated point $k_0$ satisfying that there exists a homeomorphism $\varphi : K\setminus\left\{k_0\right\}\to K\setminus\left\{k_0\right\}$ different from the identity. Then:
\begin{enumerate}
\item The map
$$
\begin{array}{rcl}
\mathcal{C}\left(K\right) & \to & \mathcal{C}\left(K\right)\\
f & \mapsto & \begin{array}{rcl} K & \to & \mathbb{R} \\ k & \mapsto & \left\{ \begin{array} {rl} f\left(k_0\right) & \text{ if }k=k_0\\ \varphi\left(f\left(k\right)\right) & \text{ if }k\neq k_0 \end{array} \right. \end{array}
\end{array}
$$
is an element of $\mathcal{G}_{\mathcal{C}\left(K\right)}\setminus \left\{\mathrm{I}_{\mathcal{C}\left(K\right)}\right\}$ having $\chi_{\left\{k_0\right\}}$ as a fixed point.
\item $\BL\left(\mathcal{C}\left(K\right)\right)$ is not a convex component of $\mathsf{S}_{\mathcal{N}_{\mathcal{C}\left(K\right)}}$.
\end{enumerate}
\end{theorem}

By relying upon the previous theorem we will show that $\BL\left(\mathcal{C}\left(K\right)\right)$ is not a convex component of $\mathsf{S}_{\mathcal{L}\left(\ell_\infty\left(
\mathcal{C}\left(K\right)\right),\mathcal{C}\left(K\right)\right)}$ for $K$ a extremally disconnected compact Hausdorff topological space with isolated points.

Note that {\em 1} of Lemma \ref{ayuda1} is implicitly used in the statement of the following theorem in the definition of the operators $T_1$ and $S_1$.

\begin{theorem}\label{settings}
Let $x_0\in\mathsf{S}_X$ be an $\mathsf{L}_\infty$-summand vector of $X$ and write $X=\mathbb{R}x \oplus_\infty M$. Consider the continuous linear operators
$$
\begin{array}{rrcl}
T_1:&bps\left(X\right) \oplus \mathbf{X}& \to &  X\\
&\left(\lambda_n x_0 + m_n \right)_{n\in\mathbb{N}}+ \mathbf{x}&\mapsto &T_1\left(\left(\lambda_n x_0 + m_n\right)_{n\in\mathbb{N}}+ \mathbf{x}\right):=\lambda_1 x_0 + x
\end{array}
$$
and
$$
\begin{array}{rrcl}
S_1:&bps\left(X\right) \oplus \mathbf{X}& \to &  X\\
&\left(\lambda_n x_0 + m_n\right)_{n\in\mathbb{N}}+ \mathbf{x}&\mapsto &S_1\left(\left(\lambda_n x_0 + m_n\right)_{n\in\mathbb{N}}+ \mathbf{x}\right):=-\lambda_1 x_0 + x.
\end{array}
$$
Then:
\begin{enumerate}
\item If $B\in\BL\left(X\right)$, then $B|_{bps\left(X\right) \oplus \mathbf{X}}=\frac{T_1+S_1}{2}$.
\item $\left\|\alpha T_1 + \left(1-\alpha\right)S_1\right\|=1$ for all $\alpha \in\left[\frac{1}{2},1\right]$.
\item $ \left\|S_1\right\| = 3$.
\item If $T,S:\ell_\infty\left(X\right)\to X$ are linear and continuous extensions of $T_1$ and $S_1$, respectively, then:
\begin{enumerate}
\item $T,S\notin \mathcal{N}_X\cup\mathcal{L}_X$.
\item $\frac{T+S}{2}\in \mathcal{N}_X\cap\mathcal{L}_X$.
\end{enumerate}
\end{enumerate}
\end{theorem}

\begin{proof}
First off, we would like to make the reader notice that by $\left(\lambda_n x_0 + m_n \right)_{n\in\mathbb{N}}+ \mathbf{x}$ we mean a generic element of $bps\left(X\right) \oplus \mathbf{X}$. Notice that $\left(\lambda_nx_0\right)_{n\in\N}\in bps(\R x_0)$ and $\left(m_n\right)_{n\in\mathbb{N}}\in bps\left(M\right)$ in accordance to {\em 1} of Lemma \ref{ayuda1}. Also notice that {\em 2} of Lemma \ref{ayuda1} does not apply since on $bps(X)$ we are considering the usual sup norm of $\ell_\infty(X)$.

\begin{enumerate}
\item Immediate.

\item We will divide this proof in several steps:
\begin{itemize}
\item $\left\|\frac{T_1+S_1}{2}\right\|=1$. Indeed, if $\left\|\left(\lambda_n x_0 + m_n \right)_{n\in\mathbb{N}}+ \mathbf{x}\right\|_\infty \leq 1$, then
\begin{eqnarray*}
\left\| \frac{T_1+S_1}{2} \left(\left(\lambda_n x_0 + m_n \right)_{n\in\mathbb{N}}+ \mathbf{x}\right)\right\|&=&\|x\|\\
&=&d\left({\bf x},bps(X)\right)\\
&\leq& \left\|\left(\lambda_n x_0 + m_n \right)_{n\in\mathbb{N}}+ \mathbf{x}\right\|_\infty\\
& \leq &1
\end{eqnarray*} in accordance to {\em 3} of Lemma \ref{pollazo}.
\item $\left\|T_1\right\|=1$. Indeed, every $x\in X$ can be written as $x=\gamma x_0 + m$ with $\gamma \in\mathbb{R}$ and $m\in M$. In accordance to {\em 3} of Lemma \ref{pollazo} we have that $\left\|m\right\|=d\left(\mathbf{m},bps\left(M\right)\right)$. Now
\begin{eqnarray*}
\left\|T_1\left(\left(\lambda_n x_0 + m_n \right)_{n\in\mathbb{N}}+ \mathbf{x}\right)\right\|&=&\left\|\lambda_1 x_0 + x\right\|\\
&=& \left\|\left(\lambda_1 + \gamma\right)x_0 + m\right\|\\
&=& \max\left\{\left|\lambda_1 + \gamma\right|,\left\|m\right\|\right\}\\
&\leq & \sup_{n\in\mathbb{N}}\left\{\left|\lambda_n + \gamma\right|,\left\|m+m_n\right\|\right\}\\
&=& \left\|\left(\lambda_n x_0 + m_n\right)_{n\in\mathbb{N}}+ \mathbf{x}\right\|_\infty.
\end{eqnarray*}
\item By using the triangular inequality we obtain that $$\left\|\alpha T_1 + \left(1-\alpha\right)S_1\right\|=\left\|\left(2\alpha-1\right) T_1 + \left(2-2\alpha\right)\frac{T_1 + S_1}{2}\right\|\leq 1$$ for all $\alpha \in\left[\frac{1}{2},1\right]$, which in fact means that $\left\|\alpha T_1 + \left(1-\alpha\right)S_1\right\|=1$ for all $\alpha \in\left[\frac{1}{2},1\right]$.
\end{itemize}

\item By using again {\em 1} of Lemma \ref{pollazo} to accomplish that $\left\|x\right\|\leq \left\|\left(\lambda_n x_0 + m_n\right)_{n\in\mathbb{N}}+ \mathbf{x}\right\|_\infty$ and the fact that $\|T_1\|=1$, we have that
\begin{eqnarray*}
\left\|S_1\left(\left(\lambda_n x_0 + m_n \right)_{n\in\mathbb{N}}+ \mathbf{x}\right)\right\|&=&\left\|-\lambda_1 x_0 + x\right\|\\
&\leq &\left\|-\lambda_1 x_0 - x\right\| + 2\left\|x\right\|\\
&\leq& \left\|\left(\lambda_n x_0 + m_n\right)_{n\in\mathbb{N}}+ \mathbf{x}\right\|_\infty + 2\left\|\left(\lambda_n x_0 + m_n\right)_{n\in\mathbb{N}}+ \mathbf{x}\right\|_\infty\\
&=& 3\left\|\left(\lambda_n x_0 + m_n\right)_{n\in\mathbb{N}}+ \mathbf{x}\right\|_\infty.
\end{eqnarray*}
This shows that $\left\|S_1\right\|\leq 3$. In order to see that $\left\|S_1\right\|=3$ it only suffices to take into consideration that $$S_1\left(\left(-2x_0,0,0,\dots,0,\dots\right) + \mathbf{x_0}\right)=3x_0.$$
\item We will divide this paragraph in two parts:
\begin{enumerate}
\item We will show that $T\notin \mathcal{N}_X \cup \mathcal{L}_{X}$. In a similar way it can be proved that $S\notin \mathcal{N}_X \cup \mathcal{L}_{X}$. Simply notice that $T\left(x_0,0,0\dots,0,\dots\right)=x_{0}\neq 0$, therefore
\begin{itemize}
\item $T\notin \mathcal{L}_{X}$ because $\left(x_0,0,0\dots,0,\dots\right)$ is a convergent sequence to $0$, and 
\item $T\notin \mathcal{N}_{X}$ because $\left(x_0,0,0\dots,0,\dots\right)$ is a sequence with bounded partial sums.
\end{itemize}
\item Let $\left(x_n\right)_{n\in\mathbb{N}}\in bps\left(X\right)$ and write $x_n=\lambda_n x_0 + m_n$ where $\lambda_n\in\mathbb{R}$ and $m_n\in M$ for all $n\in\mathbb{N}$. We have that $$\frac{T+S}{2}\left(\left(x_n\right)_{n\in\mathbb{N}}\right) = \frac{1}{2}\left(\lambda_1 x_0-\lambda_1 x_0\right)=0,$$ therefore $\frac{T+S}{2}\in \mathcal{N}_X$ in virtue of {\em 1} of Proposition \ref{prim}. Notice that $c_0\left(X\right)\subseteq \ker\left(\frac{T+S}{2}\right)$ according to {\em 1} of Proposition \ref{prim}, thus in order to show that $\frac{T+S}{2}\in\mathcal{L}_X$ it only suffices to prove that $\left(\frac{T+S}{2}\right)\left(\mathbf{x}\right)=x$ for all $x\in X$, which is immediate by construction of both $T_1$ and $S_1$.
\end{enumerate}
\end{enumerate}
\end{proof}

\begin{corollary}
$\BL\left(\mathcal{C}\left(K\right)\right)$ is not a convex component of $\mathsf{S}_{\mathcal{L}\left(\ell_\infty\left(\mathcal{C}\left(K\right)\right),\mathcal{C}\left(K\right)\right)}$ for $K$ a extremally disconnected compact Hausdorff topological space with an isolated point $k_0$.
\end{corollary}

\begin{proof}
For simplicity, let $X:=\mathcal{C}(K)$. By Lemma \ref{asinosva} we have that $x_0:=\chi_{\left\{k_0\right\}}$ is an $\mathsf{L}_\infty$-summand vector of $X$. In accordance to Theorem \ref{settings} we have that the continuous linear operators
$$
\begin{array}{rrcl}
T_1:&bps\left(X\right) \oplus \mathbf{X}& \to &  X\\
&\left(\lambda_n x_0 + m_n \right)_{n\in\mathbb{N}}+ \mathbf{x}&\mapsto &T_1\left(\left(\lambda_n x_0 + m_n\right)_{n\in\mathbb{N}}+ \mathbf{x}\right):=\lambda_1 x_0 + x
\end{array}
$$
and
$$
\begin{array}{rrcl}
S_1:&bps\left(X\right) \oplus \mathbf{X}& \to &  X\\
&\left(\lambda_n x_0 + m_n\right)_{n\in\mathbb{N}}+ \mathbf{x}&\mapsto &S_1\left(\left(\lambda_n x_0 + m_n\right)_{n\in\mathbb{N}}+ \mathbf{x}\right):=-\lambda_1 x_0 + x
\end{array}
$$
satisfy that $\left\|T_1\right\|=\left\|\frac{T_1+S_1}{2}\right\|=1$ and $ \left\|S_1\right\|= 3$. In virtue of \cite[Page 123]{D} we have that $X$ is $1$-injective, therefore we can find a norm-preserving Hahn-Banach extension $T:\ell_\infty\left(X\right)\to X$ of $T_1$. Attending to {\em 4(a)} of Theorem \ref{settings}, we deduce that $T\in \mathsf{S}_{\mathcal{L}\left(\ell_\infty\left(X\right),X\right)}$ but $T\notin \BL\left(X\right)$ (in fact, $T\notin \mathcal{N}_X\cup\mathcal{L}_X$). Now we will prove that $\BL\left(X\right)$ is not a convex component of $\mathsf{S}_{\mathcal{L}\left(\ell_\infty\left(X\right),X\right)}$ by showing that the drop $\co\left(\{T\}\cup\BL(X)\right)\subseteq \mathsf{S}_{\mathcal{L}\left(\ell_\infty\left(X\right),X\right)}$. For this it is sufficient to show that if $B\in\BL\left(X\right)$, then the segment joining $T$ and $B$ lies entirely in $\mathsf{S}_{\mathcal{L}\left(\ell_\infty\left(X\right),X\right)}$. Let $\alpha\in\left[0,1\right]$. We know that $\left\|\alpha T + \left(1-\alpha\right)B\right\|\leq 1$ since $T,B\in \mathsf{S}_{\mathcal{L}\left(\ell_\infty\left(X\right),X\right)}$. According to {\em 1} and {\em 2} of Theorem \ref{settings} simply notice that $$\left\|\left(\alpha T + \left(1-\alpha\right)B\right)|_{bps\left(X\right)\oplus\mathbf{X}} \right\|=\left\| \alpha T_1 + \left(1-\alpha\right)\frac{T_1+S_1}{2}\right\|=1,$$ which automatically implies that $\left\|\alpha T + \left(1-\alpha\right)B\right\|= 1$.

\end{proof}

\chapter{Vector-valued almost convergence}

The notion of almost convergence appeared for the first time in the literature of the theory of series in normed spaces in 1948 (see \cite{Lorentz}) and was introduced by Lorentz. Many results, extensions, and generalizations of this concept have been provided ever since. We refer the reader to \cite{Boos3, Boos2, SLS} for a wide perspective on the concept of almost convergence and some generalizations and relations with matrix methods and invariant means. For different applications of almost convergence, we refer the reader to \cite{AAGPPF, AizArPer}, where a series of results are provided involving almost convergence, almost summability, and (weakly) unconditionally Cauchy series.

\section{Notions of almost convergence}

The concept of almost convergence is originally due to Lorentz (see \cite{Lorentz}) and it strongly involves scalar-valued Banach limits. However, Lorentz, through his well-known intrinsic characterization of almost-convergence (see\cite[Theorem 1]{Lorentz}), made possible to extend it to vector-valued sequences. As a consequence, there is only one general concept of almost convergence which includes the Lorentz's scalar-valued almost convergence.

\subsection{Lorentz's (scalar-valued) almost convergence}

In 1948 (see \cite{Lorentz}) Lorentz made use of the concept of Banach limit to introduce the notion of almost convergence.

\begin{definition}[Lorentz, 1948; \cite{Lorentz}]\label{LT}
A bounded sequence $ \left(x_n\right)_{n \in \mathbb{N}}\in \ell_{\infty}$ is called almost convergent when there is a number $y\in \mathbb{R}$ such that $\varphi \left( \left(x_n\right)_{n \in \mathbb{N}}\right) = y$ for all Banach limits $\varphi : \ell_{\infty} \to \mathbb{R}$. Furthermore,
\begin{itemize}
\item $y$ is called the almost limit of $ \left(x_n\right)_{n \in \mathbb{N}}$ and
\item it is usually denoted by $\displaystyle{\AClim_{n\to\infty} x_n = y}$.
\end{itemize}
\end{definition}

In \cite[Theorem 1]{Lorentz} Lorentz provided an intrinsic characterization of almost convergent sequences.

\begin{theorem}[Lorentz, 1948; \cite{Lorentz}]\label{LIC}
Given a bounded sequence $\left(x_n\right)_{n \in \mathbb{N}}\in \ell_{\infty}$ and a real number $y$, $\displaystyle{\AClim_{n\to\infty}x_n=y}$ if and only if $$\lim_{p  \to \infty} \frac{1}{p+1}\sum_{k=0}^{p}x_{n+k} = y$$ uniformly in $n\in \mathbb{N}$.
\end{theorem}

\subsection{Boos' (vector-valued) almost convergence}

By relying upon Theorem \ref{LIC}, the author of \cite{Boos} extended the concept of almost convergence to vector-valued sequences.

\begin{definition}[Boos, 2000; \cite{Boos}]
A sequence $ \left(x_n\right)_{n \in \mathbb{N}}$ in $X$
\begin{itemize}
\item is called almost convergent when there exists $y\in X$ such that $$\lim_{p  \to \infty} \frac{1}{p+1}\sum_{k=0}^{p}x_{n+k} = y$$ uniformly in $n\in \mathbb{N}$
\begin{itemize}
\item $y$ is called the almost limit of $ \left(x_n\right)_{n \in \mathbb{N}}$ and
\item it is usually denoted by $\AClim_{n\to\infty} x_n = y$;
\end{itemize}
\item and $ \left(x_n\right)_{n \in \mathbb{N}}$ is called weakly almost convergent when there exists $y\in X$ such that $$\AClim_{n  \to \infty} f\left(x_n\right) = f\left(y\right)$$ for all $f\in X^*$
\begin{itemize}
\item $y$ is called the weak almost limit of $ \left(x_n\right)_{n \in \mathbb{N}}$ and
\item it is usually denoted by $\displaystyle{\wAClim_{n\to\infty} x_n = y}$.
\end{itemize}
\end{itemize}
\end{definition}

It is fairly obvious that
\begin{itemize}
\item every (weakly) convergent sequence is (weakly) almost convergent;
\item every (weakly) almost convergent sequence is (weakly) Cesaro-convergent.
\end{itemize}

\subsection{Basic properties of almost convergence}

What we will show now (see \cite[Theorem 1.2.18(a)]{Boos}) is that every weakly almost convergent sequence is bounded.

\begin{proposition}[Boos, 2000; \cite{Boos}]\label{bounded}
Every weakly almost convergent sequence $(x_n)_{n\in\mathbb{N}}$ in $X$ satisfies that $(x_n)_{n \in\mathbb{N}}\in \ell_{\infty}(X)$.
\end{proposition}

\begin{proof}
Since every weakly bounded sequence is also bounded, we may assume without any loss of generality that $(x_n)_{n\in\mathbb{N}}$ is almost convergent. Let $$y=\AClim_{n\to\infty} x_n$$ and fix $\varepsilon>0$ and $i_0 \in\mathbb{N}$ satisfying that $$ \left \| \sum_{k=j}^{j+i} \frac{x_k}{i+1} \right \| \leqslant \|y\|+\varepsilon$$ for every $i \geqslant i_0$ and every $j \in\mathbb{N}$. Now, for every $j\in \mathbb{N}$ we have that $$\|x_j\| = \left \| \frac{i_0+2}{i_0+1} \sum_{k=j}^{j+i_0+1} \frac{x_k}{i_0+2} - \sum_{k=j+1}^{j+i_0+1} \frac{x_k}{i_0+1}\right \| \leqslant \left ( \frac{i_0+2}{i_0+1}+1 \right )\left(\|x_0\|+\varepsilon\right),$$ where the last term is a fixed constant, what concludes the proof.
\end{proof}

\begin{lemma}\label{wacden}
If $N$ is a dense vector subspace of $X^*$ and $\left(x_n\right)_{n\in \mathbb{N}}\in\ell_\infty(X)$, then $\displaystyle{\wAClim_{n\to\infty} x_n=0}$ if and only if for each $g \in N$ we have that $\displaystyle{\AClim_{n\to\infty} g\left(x_n\right)=0}$.
\end{lemma}

\begin{proof}
 Indeed, fix an arbitrary $f \in X^*$ and consider $\varepsilon>0$. There exists $A>0$ such that
    $\left\|x_n\right\|<A$ for each $n \in \mathbb{N}$. By the density of
    $N$ in $X^*$ there exists $g \in N$ satisfying that $\left\|f-g\right\|<\frac{\varepsilon}{2A}$.
    Then
    \begin{eqnarray*}
    \frac{1}{i+1} \left| f\left(\sum_{k=0}^{i} x_{j+k}\right)\right| &\leqslant& \frac{1}{i+1} \left ( \left|\left(f-g\right)\left(\sum_{k=0}^{i} x_{j+k}\right)\right|+ \left| g\left(\sum_{k=0}^{i} x_{j+k}\right)\right| \right )\\
&\leqslant& \frac{1}{i+1}\frac{\varepsilon}{2A}\left(i+1\right)A + \frac{1}{i+1} \left| g\left(\sum_{k=0}^{i} x_{j+k}\right)\right|\\
&=& \frac{\varepsilon}{2} + \frac{1}{i+1} \left| g\left(\sum_{k=0}^{i} x_{j+k}\right)\right|
\end{eqnarray*} for every $i,j\in \mathbb{N}$. Since $\displaystyle{\AClim_{n\to\infty} g\left(x_n\right)=0}$, we conclude that $\displaystyle{\AClim_{n\to\infty} f\left(x_n\right)=0}$. The arbitrariness of $f$ tells us that $\displaystyle{\wAClim_{n\to\infty} x_n=0}$.
\end{proof}

In accordance to Lemma \ref{wacden} and the density of $c_{00}$ in $\ell_1$, we deduce that a sequence $\left(x^n\right)_{n\in\mathbb{N}}\subset c_0$ is weakly almost convergent to $x^0
\in c_0$ if and only if $\displaystyle{\AClim_{n\to\infty} x^n_i=x^0_i}$ for every $i \in
\mathbb{N}$. Nevertheless, the sequence $\left(x^n\right)_{n\in\mathbb{N}}\subset c_0$ given by
$$x_i^n=\left\{\begin{array}{lll}
  0 & \hbox{if} & i>n, \\
  1 & \hbox{if} & n=j+(2k+1)i,j \in \{ 0, 1, \ldots, i-1 \}, k \in \mathbb{N},\\
 -1 & \hbox{if} & n=j+(2k+2)i,j \in \{ 0, 1, \ldots, i-1 \}, k \in \mathbb{N},\\
\end{array}\right.
$$
satisfies that $\displaystyle{\wAClim_{n\to\infty}x^n=0}$ but it is not almost convergent.

On the other hand, we would like to make the reader aware about the fact that an almost convergent sequence might have subsequences not almost converging to the almost limit. For instance, the sequence $(-1,1,-1,1,-1,1,\dots)$ is almost convergent to $0$ but not all of its subsequences are almost convergent to $0$.

\begin{proposition}\label{finito}
The following conditions are equivalent for a bounded sequence $\left(x_n\right)_{n\in \mathbb{N}}$ in a finite dimensional normed space:
\begin{enumerate}
\item $\left(x_n\right)_{n\in \mathbb{N}}$ is convergent to $x$.
\item All of its subsequences are almost convergent to $x$.
\item Every subsequence has a further subsequence almost convergent to $x$.
\end{enumerate}
\end{proposition}

\begin{proof}
We only need to show that {\em 3} $\Rightarrow$ {\em 1} since the other two implications are obvious.

Suppose to the contrary that $\left(x_n\right)_{n\in \mathbb{N}}$ is not convergent to $x$. There exist $r>0$ and a subsequence $\left(x_{n_k}\right)_{k\in \mathbb{N}}$ of $\left(x_n\right)_{n\in \mathbb{N}}$ such that $x_{n_k}\notin \mathsf{B}_X\left(x,r\right)$ for all $k\in \mathbb{N}$. Since $\left(x_{n_k}\right)_{n\in \mathbb{N}}$ is bounded, it has a further subsequence $\left(x_{n_{k_j}}\right)_{j\in \mathbb{N}}$ convergent to some $y \in X\setminus \mathsf{U}_X\left(x,r\right)$. By hypothesis, $\left(x_{n_{k_j}}\right)_{j\in \mathbb{N}}$ has a further subsequence $\left(x_{n_{k_{j_p}}}\right)_{p\in \mathbb{N}}$ almost convergent to $x$, which is impossible since $$y=\lim_{p\to\infty}x_{n_{k_{j_p}}}=\AClim_{p\to\infty}x_{n_{k_{j_p}}}=x.$$
\end{proof}

\begin{corollary}\label{nota1}
A sequence is weakly convergent if and only if all of its subsequences are weakly almost convergent to the same limit and if and only if every subsequence has a further subsequence weakly almost convergent to the same limit.
\end{corollary}

Proposition \ref{finito} does not hold in infinite dimensions as expected. Indeed, consider in $c_0$ or $\ell_p \; \left(p>1\right)$ the sequence $\left(e_i\right)_{i\in \mathbb{N}}$ of canonical vectors. For every infinite subset $M \subset\mathbb{N}$ we have that $\displaystyle{\AC\lim_{i \in M} e_i=0}$, but $\left\|e_i\right\|=1$ for every $i \in \mathbb{N}$.

\begin{theorem}\label{CauchyAC}
If $\left(x_n\right)_{n\in\mathbb{N}}$ is a Cauchy sequence in $X$, then $\left(x_n\right)_{n\in\mathbb{N}}$ is convergent if and only if it is almost convergent.
\end{theorem}

\begin{proof}
Suppose that $\left(x_n\right)_{n\in\mathbb{N}}$ is an almost convergent Cauchy sequence in $X$. There exists $x_0 \in X$ such that $\displaystyle{\mathrm{AC}\lim_{i\to\infty} x_i=x_0}$. Without lack of generality, we suppose that $x_0=0$. If $\varepsilon>0$ is given, then there exists $i_0 \in \mathbb{N}$ satisfying $$\frac{1}{i_0+1}\left \| \sum_{k=0}^{i_0} x_{j+k}\right \| \leqslant \frac{\varepsilon}{2}$$ for every $j \in \mathbb{N}$. On the other hand, there exists $j_0 \in \mathbb{N}$ such that if $p,q\geqslant j_0$, then $\left\|x_p-x_q\right\| \leqslant \varepsilon/2$. Therefore $$\left \| x_j+\frac{1}{i_0+1} \sum_{p=j+1}^{j+i_0} (x_p-x_j) \right\| = \frac{1}{i_0+1}\left \| \sum_{k=0}^{i_0} x_{j+k}\right\| \leqslant \frac{\varepsilon}{2}$$ for every $j \geqslant j_0$. If we denote $\displaystyle{v=\sum_{p=j+1}^{j+i_0}\left(x_p-x_j\right)}$, then we have that $$\left\|x_j\right\|-\left\|\frac{v}{i_0+1}\right\| \leqslant \left\|x_j +\frac{v}{i_0+1}\right\|\leqslant \frac{\varepsilon}{2}.$$ However $\left\|\frac{v}{i_0+1}\right\| \leqslant \frac{\varepsilon}{2}$ so we deduce that $\displaystyle{\left\|x_j\right\| \leqslant \left\|\frac{v}{i_0+1}\right\|+ \frac{\varepsilon}{2} \leqslant \varepsilon}$ for each $j \geqslant j_0$.
\end{proof}

\section{Spaces of almost convergent sequences}

Like every time a new convergent method is born, what is due is to consider the corresponding sequence space and to study its properties. The sequence spaces related with the almost convergence will have a strong impact on the structure of the set of Banach limits.

\subsection{The spaces $ac\left(X\right)$ and $wac\left(X\right)$}

We refer the reader to \cite{Boos} where the following spaces are defined.

\begin{definition}[Boos, 2000; \cite{Boos}]
If $X$ is a normed space, then
\begin{itemize}
\item $ac\left(X\right) := \left\{\left(x_i\right)_{i\in \mathbb{N}} \in X^{\mathbb{N}}: \AClim_{i\to\infty} x_i \; \hbox{exists} \right\}$
\item $wac\left(X\right) := \left\{\left(x_i\right)_{i\in \mathbb{N}} \in X^{\mathbb{N}}: \wAClim_{i\to\infty} x_i \; \hbox{exists} \right\}$.
\end{itemize}
\end{definition}

Obviously $$c\left(X\right)\subset ac\left(X\right)\subset wac\left(X\right)\subset\ell_{\infty}\left(X\right).$$

Another usual space of this kind is $$ac_0\left(X\right) = \left\{\left(x_i\right)_{i\in \mathbb{N}} \in X^{\mathbb{N}}: \AClim_{i\to\infty} x_i =0 \right\}.$$

It is fairly obvious that $c_0\left(X\right)\subset ac_0\left(X\right)\subset ac\left(X\right)$. The following lemma directly relies on Lemma \ref{bps}.

\begin{lemma}\label{jaja}
$bps(X)\subset ac_0\left(X\right)$.
\end{lemma}

\begin{proof}
For all $n,p\in \mathbb{N}$ we have $$\left\|\frac{\sum_{k=0}^p\left(z_{n+k+1}-z_{n+k}\right)}{p+1}\right\|=\left\|\frac{z_{n+p+1}-z_{n}}{p+1}\right\|\leq \frac{2\left\|\left(z_n\right)_{n\in \mathbb{N}}\right\|_{\infty}}{p+1}.$$
\end{proof}

\begin{theorem}\label{completitud}
The spaces $ac\left(X\right)$ and $wac\left(X\right)$ are closed in $\ell_\infty\left(X\right)$ endowed with the sup norm provided that $X$ is complete.
\end{theorem}

\begin{proof}
We will only prove the closedness of $wac\left(X\right)$. The closedness of $ac\left(X\right)$ can be shown in a similar way. Let $\left(x^n\right)_{n\in\mathbb{N}}$ be a sequence in $wac\left(X\right)$ and consider $x^0\in \ell_\infty\left(X\right)$ such that $\displaystyle{\lim_{n\to\infty} \left\| x^n-x^0 \right\|_{\infty}=0}$. We will show that $x^0 \in wac\left(X\right)$. For each natural $n$, there exists $x_n\in X$ satisfying that $\displaystyle{\wAClim_{i\to\infty}x_i^n=x_n}$. We will show first that $\left(x_n\right)_{n\in\mathbb{N}}$ is a Cauchy sequence in $X$. Take any $\varepsilon>0$. An $n_0\in\mathbb{N}$ can be found such that for each $p,q \geqslant n_0$ we have that $\left\| x^p-x^q \right\|_{\infty} \leqslant \varepsilon/3$. Fix $p,q \geqslant n_0$ and consider a functional $f \in \mathsf{S}_{X^*}$ such that $\left\| x_p-x_q \right\|= \left|f\left(x_p\right)-f\left(x_q\right)\right|$. There exists $i \in \mathbb{N}$ such that
$$\left| f\left(x_p\right)- \frac{1}{i+1}\left(f\left(x_j^p\right)+ \cdots + f\left(x_{j+i}^p\right)\right) \right| \leqslant \frac{\varepsilon}{3}$$
and
$$\left| f\left(x_q\right)- \frac{1}{i+1}\left(f\left(x_j^q\right)+ \cdots + f\left(x_{j+i}^q\right)\right) \right| \leqslant \frac{\varepsilon}{3}$$
for every $j \in \mathbb{N}$. It follows that
    \begin{eqnarray*}
      \nonumber \left \|x_p-x_q \right \| &\leqslant&\left| f\left(x_p\right)- \frac{1}{i+1}\left(f\left(x_j^p\right)+ \cdots + f\left(x_{j+i}^p\right)\right) \right| \\
      \nonumber &+& \left|\frac{1}{i+1}\left(f\left(x_j^p-x_j^q\right)+\cdots+f\left(x_{j+i}^p-x_{j+i}^q\right)\right)\right| \\
      \nonumber &+& \left| f\left(x_q\right)- \frac{1}{i+1}\left(f\left(x_j^q\right)+ \cdots + f\left(x_{j+i}^q\right)\right) \right| \\ &\leqslant& \varepsilon.
    \end{eqnarray*}
Since $X$ is complete, there exists $x_0 \in X$ such that $\displaystyle{\lim_{n\to \infty} x_n=x_0}$. Finally, we will show that $\displaystyle{\wAClim_{i\to\infty} x_i^0=x_0}$. If $\varepsilon>0$ is given and $f \in X^*\setminus \left\{0\right\}$, then we can fix $p \in \mathbb{N}$ satisfying $\left\| x_0-x_p \right\| \leqslant \frac{\varepsilon}{3 \left\| f \right\|}$ and $\left\| x^0-x^p \right\|_{\infty} \leqslant \frac{\varepsilon}{3 \left \| f \right\|}$. Since $\displaystyle{\wAClim_{i\to\infty}x_i^p=x_p}$, there exists $i_0\in \mathrm{N}$ such that
    $$\left|f\left(x_p\right)-\frac{1}{i+1}\left(f\left(x_j^p\right)+ \cdots + f\left(x_{j+i}^p\right)\right)\right| \leqslant \frac{\varepsilon}{3}$$ for every $i \geqslant i_0$ and every $j \in \mathbb{N}$. Thus
    \begin{eqnarray*}
     && \left|f\left(x_0\right)-\frac{1}{i+1}\left(f\left(x_j^0\right)+ \cdots + f\left(x_{j+i}^0\right)\right)\right|\\ & \leqslant & \left|f\left(x_0\right)-f\left(x_p\right)\right| \\
      & + & \left|f\left(x_p\right)-\frac{1}{i+1}\left(f\left(x_j^p\right)+ \cdots + f\left(x_{j+i}^p\right)\right)\right| \\
      & + & \left|\frac{1}{i+1}\left(f\left(x_j^p-x_j^0\right)+\cdots+f\left(x_{j+i}^p-x_{j+i}^0\right)\right)\right| \\
& \leqslant &
      \varepsilon,
    \end{eqnarray*}
 for every $i \geqslant i_0$ and every $j \in \mathbb{N}$. In order words, $\displaystyle{\wAClim_{i\to\infty} x_i^0 =x_0}$.
\end{proof}

By bearing in mind that $\mathbf{X}$ is a closed subspace of $\ell_\infty\left(X\right)$ even if $X$ is not complete, we immediately deduce the following corollary.

\begin{corollary}
The following conditions are equivalent:
\begin{enumerate}
\item $ac\left(X\right)$ is complete.
\item $wac\left(X\right)$ is complete.
\item $X$ is complete.
\end{enumerate}
\end{corollary}

It is immediate that $ac_0\left(\overline{X}\right)\cap \ell_\infty\left(X\right)=ac_0\left(X\right)$ and thus $ac_0(X)$ is always closed in $\ell_\infty(X)$.

\begin{theorem}\label{jopa}
The following conditions are equivalent:
\begin{enumerate}
\item $ac\left(\overline{X}\right)\cap \ell_\infty\left(X\right)=ac\left(X\right)$.
\item $ac\left(X\right)$ is a closed subspace of $\ell_\infty\left(X\right)$.
\item $X$ is complete.
\end{enumerate}
\end{theorem}

\begin{proof}
\mbox{}
\begin{enumerate}
\item[{\em 1} $\Rightarrow${\em 2}] Immediate if taken into account that $ac\left(\overline{X}\right)$ is closed in $\ell_\infty\left(\overline{X}\right)$ in virtue of Theorem \ref{completitud}.
\item[{\em 2} $\Rightarrow${\em 3}] Assume to the contrary that $X$ is not complete. Consider a non-convergent Cauchy sequence $\left(x_n\right)_{n\in\mathbb{N}}\subset X$. It is obvious that $\left(x_n\right)_{n\in\mathbb{N}}\in c\left(\overline{X}\right)\cap \ell_\infty\left(X\right)\subset ac\left(\overline{X}\right)\cap \ell_\infty\left(X\right)$. In accordance to Lemma \ref{bixic} there exists a sequence of elements of $c\left(X\right)$ converging to $\left(x_n\right)_{n\in\mathbb{N}}$. Since $c\left(X\right)\subset ac\left(X\right)$, we deduce by hypothesis that $\left(x_n\right)_{n\in\mathbb{N}}\in ac\left(X\right)$, which is impossible since the limit of $\left(x_n\right)_{n\in\mathbb{N}}$ in $\overline{X}$ does not belong to $X$.
\item[{\em 3} $\Rightarrow${\em 1}] Obvious.
\end{enumerate}
\end{proof}

It is not excessively hard to check that $ac_0(X)$ is dense in $ac_0\left(\overline{X}\right)$. Indeed, if $(y_n)_{n\in \N}\in ac_0\left(\overline{X}\right)$, then we simply need to find $x_n\in X$ with $\|x_n-y_n\|<\frac{\varepsilon}{2^n}$, and we will have that $$\left\|\frac{1}{p+1}\sum_{k=0}^px_{n+k}\right\|\leq \frac{\varepsilon}{p+1}+\left\|\frac{1}{p+1}\sum_{k=0}^py_{n+k}\right\|,$$ which trivially implies that $\AClim_{n\to \infty} x_n=0$.

The density of $ac(X)$ in $ac\left(\overline{X}\right)$ is not that trivial (unless we rely on the previous sentence).

\begin{lemma}\label{helpazoahi}
$ac\left(X\right)$ is dense in $ac\left(\overline{X}\right)$.
\end{lemma}

\begin{proof}
Let $\left(y_n\right)_{n\in\mathbb{N}}\in ac\left(\overline{X}\right)$ and $y=\displaystyle{\AClim_{n\to\infty} y_n}$. Fix an arbitrary $\varepsilon >0$. There exists a sequence $(x_n)_{n\in\N}\in ac_0(X)$ with $\left\|(x_n)_{n\in\N}-(y_n-y)_{n\in\N}\right\|_\infty<\varepsilon/2$. On the other hand, there exists $x\in X$ with $\left\|x-y\right\|<\varepsilon/2$. Finally, notice that $(x_n+x)_{n\in \N}\in ac(X)$ and $\left\|(x_n+x)_{n\in\N}-(y_n)_{n\in\N}\right\|_\infty<\varepsilon$.
\end{proof}

\begin{corollary}\label{helpazo}
$ac\left(X\right)$ is never dense in $\ell_\infty\left(X\right)$.
\end{corollary}

\begin{proof}
If $ac\left(X\right)$ is dense in $\ell_\infty\left(X\right)$, then $ac\left(\overline{X}\right)$ is dense in $\ell_\infty\left(\overline{X}\right)$. In this situation Theorem \ref{completitud} implies that $ac\left(\overline{X}\right)=\ell_\infty\left(\overline{X}\right)$. On the other hand, it is well known that $\ell_\infty\setminus ac\neq \varnothing$, thus $\ell_\infty\left(\overline{X}\right)\setminus ac\left(\overline{X}\right)\neq \varnothing$. This is a contradiction.
\end{proof}

\subsection{The almost convergent limit function}

 The almost convergent limit function is defined as
\begin{equation}\label{aclim}
\begin{array}{rrcl} \AClim : & ac\left(X\right) & \to & X\\&\left(x_n\right)_{n\in\mathbb{N}}&\mapsto&\displaystyle{\mathrm{AC}\lim_{n\to\infty}x_n}\end{array}
\end{equation} and is a norm-$1$ continuous linear operator such that $\mathrm{AC}\lim|_{c\left(X\right)}=\lim$. It is trivial that $ac_0\left(X\right)=\ker\left(\mathrm{AC}\lim \right)$ and thus $ac_0(X)$ is closed in $ ac(X)$. Recall though that in the previous subsection we mentioned that $ac_0(X)$ is always closed in $\ell_\infty(X)$ since, in fact, $ac_0(X)=ac_0\left(\overline{X}\right)\cap \ell_\infty(X)$. We recall the reader that $\mathcal{HB}\left(\AClim\right)$ stands for the set of all norm-$1$ Hahn-Banach extensions of $ \mathrm{AC}\lim$ to the whole of $\ell_\infty\left(X\right)$.

\begin{lemma}\label{HBCBL}
$\mathcal{HB}\left(\AClim\right)\subseteq \BL\left(X\right)\subseteq \mathcal{HB}\left(\lim\right)$.
\end{lemma}

\begin{proof}
By definition it is clear that $\BL\left(X\right)\subseteq \mathcal{HB}\left(\lim\right)$. Let $T\in \mathcal{HB}\left(\AClim\right)$. All we need to show is that $T\in\mathcal{HB}\left(\lim\right)\cap\mathcal{N}_X$ if we bear in mind Corollary \ref{Bprim}. It is pretty obvious that $ T\in \mathcal{HB}\left(\lim\right)$. Finally, it can be shown that $T\in\mathcal{N}_X$ by taking into consideration Proposition \ref{prim} together with Lemma \ref{jaja}.
\end{proof}

From now on until the end of this subsection we will assume that $X=\mathbb{R}$.

\begin{lemma}\label{jojo}
For every $0<\varepsilon < \delta$ there exists a sequence $\left(x_n\right)_{n\in \mathbb{N}}\in\ell_{\infty}$ such that:
\begin{enumerate}
\item $\left\|\left(x_n\right)_{n\in \mathbb{N}}\right\|_{\infty}=\delta + \varepsilon$.
\item $\left(x_n\right)_{n\in \mathbb{N}}$ has infinitely many positive terms and infinitely many negative terms.
\item $\displaystyle{\AClim_{n\to\infty}x_n =\varepsilon}$.
\end{enumerate}
\end{lemma}

\begin{proof}
Let $\left(z_n\right)_{n\in \mathbb{N}}\in \ell_{\infty}$ be the sequence $$\left(\frac{\delta}{2},\frac{-\delta}{2},\frac{\delta}{2},\frac{-\delta}{2},\dots\right).$$ Take $\left(u_n\right)_{n\in \mathbb{N}}\in\ell_{\infty}$ to be the sequence $\left(z_{n+1}-z_n\right)_{n\in \mathbb{N}}$, that is, the sequence $$\left(-\delta,\delta,-\delta,\delta,\dots\right).$$ Note that by Lemma \ref{jaja}, $\left(u_n\right)_{n\in \mathbb{N}}\in ac_0$. Finally, the desired sequence $\left(x_n\right)_{n\in\mathbb{N}}$ will be $$\left(-\delta+\varepsilon,\delta+\varepsilon,-\delta+\varepsilon,\delta+\varepsilon,\dots\right).$$
\end{proof}

\begin{theorem}
$ \mathcal{HB}\left(\mathrm{AC}\lim\right) \subsetneq \mathcal{HB}\left(\lim\right)$.
\end{theorem}

\begin{proof}
According to Lemma \ref{jojo}, we may fix an almost convergent sequence $\left(x_n\right)_{n\in \mathbb{N}}\in ac\setminus c$ such that $\left(x_n\right)_{n\in \mathbb{N}}$ has infinitely many positive terms and infinitely many negative terms, and $$0<\AClim_{n\to\infty}x_n < \left\|\left(x_n\right)_{n\in \mathbb{N}}\right\|_{\infty}.$$ Define the linear function:
\begin{equation*}
\begin{array}{rrcl}
\varphi:&c \oplus \mathbb{R}\left(x_n\right)_{n\in \mathbb{N}}&\to&\mathbb{R}\\
&\left(c_n\right)_{n\in \mathbb{N}}+\lambda\left(x_n\right)_{n\in \mathbb{N}}&\mapsto&\varphi\left(\left(c_n+\lambda x_n\right)_{n\in \mathbb{N}}\right):= \displaystyle{\lim_{n\to\infty}c_n}.
\end{array}
\end{equation*}
We will show now that $\varphi$ is continuous and has norm $1$. Let $\left(c_n\right)_{n\in \mathbb{N}}\in c$ and $\lambda \in \mathbb{R}$. We may assume without loss of generality that $\lim_{n\to\infty}c_n \geq 0$. If $\lambda \geq 0$, then there exists a subsequence $\left(x_{n_k}\right)_{k\in \mathbb{N}}$ of $\left(x_n\right)_{n\in \mathbb{N}}$ of positive terms. Therefore
\begin{eqnarray*}
\left|\varphi\left(\left(c_n+\lambda x_n\right)_{n\in \mathbb{N}}\right)\right|&=& \lim_{n\to\infty}c_n \\
&=& \lim_{k\to\infty}c_{n_k} \\
&\leq& \lim_{k\to\infty}c_{n_k}+\lambda x_{n_k}\\
& \leq& \sup\left\{\left|c_{n_k}+\lambda x_{n_k}\right|:k\in \mathbb{N}\right\}\\
&\leq& \left\|\left(c_n+\lambda x_n\right)_{n\in \mathbb{N}}\right\|_{\infty}.
\end{eqnarray*}
If $\lambda \leq 0$, we apply a similar reasoning by considering a subsequence $\left(x_{n_k}\right)_{k\in \mathbb{N}}$ of $\left(x_n\right)_{n\in \mathbb{N}}$ of negative terms. Next, $\varphi|_c=\lim$, therefore $\left\|\varphi\right\|=1$. Finally, in accordance with the Hahn-Banach Extension Theorem, $\varphi$ may be extended linearly, continuously, and preserving its norm to the whole of $\ell_{\infty}$. To simplify, we will keep denoting this extension by $\varphi$. By construction, $$\varphi\left(\left(x_n\right)_{n\in \mathbb{N}}\right)=0<\AClim_{n\to\infty}x_n,$$ thus $\varphi|_{ac}\neq \AClim$.
\end{proof}

We refer the reader to Subsection \ref{geometry} for the basics on geometry of Banach spaces, such as smoothness and rotundity.

\begin{theorem}
The set $\mathcal{HB}\left(\lim\right)$ of all norm-$1$ Hahn-Banach extensions of the limit function on $c$ to $\ell_{\infty}$ is a $w^*$-exposed face of $\mathsf{B}_{\ell_{\infty}^*}$.
\end{theorem}

\begin{proof}
It suffices to apply Lemma \ref{polla1} to $X:=\ell_{\infty}$, $Y:=c$, $y^*:=\lim$, and $y:=\left(\frac{n-1}{n}\right)_{n\in\mathbb{N}}$. It only remains to show that $y$ is indeed a smooth point of $\mathsf{B}_c$. In order to prove this, we first notice that $c=\mathcal{C}\left(\omega \cup \left\{\infty\right\}\right)$, where $\omega \cup \left\{\infty\right\}$ denotes the one-point compactification of the natural numbers $\omega$. Now, we call on Theorem \ref{bsmo} to conclude that $y$ is a smooth point of $\mathsf{B}_c$ because $y$ attains its absolute maximum value at only $\infty$.
\end{proof}

\subsection{The space $w^*ac\left(X^*\right)$}

In a natural way we can consider the almost convergence of sequences in dual spaces endowed with the weak star topology, which takes us to the concept of weak-star almost convergence.

\begin{definition}[Boos, 2000; \cite{Boos}]
A sequence $\left(x^*_i\right)_{i\in\mathbb{N}}$ in $X^*$ is called weakly-star almost convergent when there exists $y^* \in X^*$ such that $$\displaystyle{\AClim_{i\to\infty}x^*_i\left(x\right)=y^*\left(x\right)}$$ for all $x\in X$
\begin{itemize}
\item $y^*$ is called the weak-star almost limit of $\left(x^*_i\right)_{i\in\mathbb{N}}$
\item it is usually denoted by $\displaystyle{\wsAClim_{i\to\infty} x_i^*=y^*}$.
\end{itemize}
The space of $w^*$-almost convergent sequences is defined as $$w^*ac\left(X^*\right):=\left\{\left(f_i\right)_{i\in\mathbb{N}} \in \left(X^*\right)^\mathbb{N} :  \displaystyle{\wsAClim_{i\to\infty} f_i} \; \hbox{exists}\right\}.$$
\end{definition}

It is fairly obvious that every weakly-star convergent sequence is weakly-star almost convergent and every weakly almost convergent sequence in a dual space is weakly-star almost convergent.

\begin{theorem}
$w^*ac\left(X^*\right)\cap \ell_\infty\left(X^*\right)$ is a closed subspace of $\ell_\infty\left(X^*\right)$.
\end{theorem}

\begin{proof}
Let $\left(f^n\right)_{n\in\mathbb{N}}\subset w^*ac\left(X^*\right)\cap \ell_\infty\left(X^*\right)$ satisfying that $\displaystyle{\lim_{n\to\infty} \left\| f^n-f^0 \right\|_{\infty}=0}$ for some $f^0 \in \ell_\infty\left(X^*\right)$. Our purpose is to prove that $f^0 \in w^*ac\left(X^*\right)$. For each natural $n$ there exists $f_n \in X^*$ such that $\displaystyle{\wsAClim_{i\to\infty}\left(f_i^n\right)=f_n}$. We will show next that $\left(f_n\right)_{n\in\mathbb{N}}$ is a Cauchy sequence. If an $\varepsilon>0$ is given, there exists $n_0\in\mathbb{N}$ such that $\left\| f^p-f^q \right\|_{\infty} \leqslant \varepsilon/6$ for any $p,q \geqslant n_0$. Fix $p,q \geqslant n_0$. We can find a vector $x \in \mathsf{S}_X$ satisfying $$\left\| f_p-f_q \right\| - \frac{\varepsilon}{2} <\left|\left(f_p-f_q\right)\left(x\right)\right|\leqslant\left\| f_p-f_q \right\|.$$ Consider a natural $i$ such that for every $j \in \mathbb{N}$ we have $$\left|f_p\left(x\right)-\frac{1}{i+1}\left(f_j^p\left(x\right)+\cdots+f_{j+i}^p\left(x\right)\right)\right| \leqslant \frac{\varepsilon}{6}$$
and
$$\left|f_q\left(x\right)-\frac{1}{i+1}\left(f_j^q\left(x\right)+\cdots+f_{j+i}^q\left(x\right)\right)\right| \leqslant \frac{\varepsilon}{6}.$$
It follows that
\begin{eqnarray*}
\left\| f_p-f_q \right\| - \frac{\varepsilon}{2} &\leqslant& \left|f_p\left(x\right)-\frac{1}{i+1}\left(f_j^p\left(x\right)+\cdots +f_{j+i}^p\left(x\right)\right)\right| \\
&+& \left| \frac{1}{i+1}\left(\left(f_j^p-f_j^q\right)\left(x\right)+\cdots+\left(f_{j+i}^p-f_{j+i}^q\right)\left(x\right)\right)\right| \\
&+& \left|\frac{1}{i+1}\left(f_j^q\left(x\right)+\cdots+f_{j+i}^q\left(x\right)\right)-f_q\left(x\right)\right|\\
&\leqslant&\frac{ \varepsilon}{2},
\end{eqnarray*}
that is, $\left\| f_p-f_q \right\| \leqslant \varepsilon$ for each $p,q \geqslant n_0$. Then there exists $f_0 \in X^*$ such that $\displaystyle{\lim_{n\to \infty} \left\| f_n-f_0 \right\| =0}$. We will show now that $\displaystyle{\wsAClim_{i\to\infty} f_i^0=f_0}$. Consider $x \in X\setminus\left\{0\right\}$ and $\varepsilon >0$. We can fix $p \in\mathbb{N}$ such that $\left\| f^p-f^0 \right\|_{\infty} \leqslant\frac{\varepsilon}{3\left \|x\right\|}$ and $\left\|f_p\left(x\right)-f_0\left(x\right)\right\| \leqslant \frac{\varepsilon}{3}$. Since $\displaystyle{\wsAClim_{i\to\infty} f_i^p=f_0}$, there  exists $i_0\in\mathbb{N}$ such that for each $i \geqslant i_0$ it is satisfied that $\left|f_p\left(x\right)- \frac{1}{i+1}\left(f_j^p\left(x\right)+\cdots+f_{j+i}^p\left(x\right)\right)\right| \leqslant\frac{\varepsilon}{3}$ for every $j \in \mathbb{N}$. Therefore, if $i \geqslant i_0$, then
\begin{eqnarray*}
&&\left|f_0\left(x\right)- \frac{1}{i+1}\left(f_j^0\left(x\right)+ \cdots+f_{j+i}^0\left(x\right)\right)\right|\\&\leqslant& \left|f_0\left(x\right)-f_p\left(x\right)\right| \\
&+&\left |f_p\left(x\right)-\frac{1}{i+1}\left(f_j^p\left(x\right)+\cdots+f_{j+i}^p\left(x\right)\right)\right| \\
&+& \left | \frac{1}{i+1}\left(\left(f_j^p-f_j^0\right)\left(x\right)+\cdots+\left(f_{j+i}^p-f_{j+i}^0\right)\left(x\right)\right)\right| \\
&\leqslant& \varepsilon
\end{eqnarray*}
for every $j \in \mathbb{N}$. Thus $\displaystyle{\wsAClim_{i\to\infty} f_i=f_0}$.
\end{proof}

By noticing that if $X$ is barrelled, then every weak-star almost convergent sequence in $X^*$ is bounded and hence $w^*ac\left(X^*\right)\subseteq \ell_\infty(X^*)$, we immediately deduce the following corollary.

\begin{corollary}
If $X$ is barrelled, then $w^*ac\left(X^*\right)$ is a closed subspace of $\ell_\infty(X^*)$.
\end{corollary}

We recall the reader an example of a weakly-star convergent sequence which is not bounded. Indeed, consider $X:=c_{00}$ and $\left(ne_n\right)_{n\in\mathbb{N}}\subset \ell_1 = X^*$. It is easy to see that $\left(ne_n\right)_{n\in\mathbb{N}}$ is weak-star convergent in $X^*$ to $0$ but it is not bounded.

\begin{lemma}\label{vayapu}
If $M$ is a dense vector subspace of $X$ and $\left(f_n\right)_{n\in \mathbb{N}}$ is a bounded sequence in $X^*$, then $\displaystyle{\wsAClim_{n\to\infty} f_n=0}$ if and only if for each $y \in M$ we have that $\displaystyle{\AClim_{n\to\infty} f_n\left(y\right)=0}$.
\end{lemma}

In virtue of Lemma \ref{vayapu}, whose proof we omit due to its similarity to that of Lemma \ref{wacden}, and the density of $c_{00}$ in $c_0$, we have that a bounded sequence $\left(f^n\right)_{n\in\mathbb{N}}\subset \ell_1$ is $w^*$-almost convergent to $f^0 \in \ell_1$ if and only if $\displaystyle{\AClim_{n\to\infty} f^n_i=f^0_i}$ for each $i \in \mathbb{N}$. On the other hand, the sequence $\left(f^n\right)_{n\in\mathbb{N}}\subset \ell_1$ given by
$$f_i^n= \left \{ \begin{array}{lll}
                        \displaystyle{\frac{1}{2^{i-n+1}}} & \hbox{if} & i \geqslant n, \\
                        & & \\
                        \displaystyle{(-1)^{i+n+1} \; \frac{1}{2^i}} & \hbox{if} & i<n,
                      \end{array}
     \right.$$verifies that $\displaystyle{\wsAClim_{n\to\infty} f^n=0}$ (observe that $\left\|f^n\right\| < 2$ for each $n \in \mathbb{N}$). However, it is not almost convergent. In the next subsection we will show that $wac\left(\ell_1\right)=ac\left(\ell_1\right)$, so $\left(f^n\right)_{n\in\mathbb{N}}$ is not weakly almost convergent either.

Finally, in accordance to Proposition \ref{finito}, we may assure that a sequence is weakly-star convergent if and only if all of its subsequences are weakly-star almost convergent to the same limit and if and only if every subsequence has a further subsequence weakly-star almost convergent to the same limit.
     
\subsection{Almost convergence and classical properties}

This subsection is aimed at characterizing some classical properties in terms of the almost convergence. We will start off with a characterization of completeness, which is an immediate consequence of Theorem \ref{CauchyAC}.

\begin{corollary}
A normed space $X$ is complete if and only if every Cauchy sequence in $X$ is almost convergent.
\end{corollary}

\begin{theorem}
A normed space $X$ is reflexive if and only if every bounded sequence in $X$ has a weakly almost convergent subsequence.
\end{theorem}

\begin{proof}
Assume that every bounded sequence in $X$ has a weakly almost convergent subsequence. We will distinguish two parts:
\begin{enumerate}
\item $X$ must be complete: Indeed, let $y\in \overline{X}$. There exists a sequence $\left(x_n\right)_{n\in\mathbb{N}}\subset X$ which converges to $y$. Now, $\left(x_n\right)_{n\in\mathbb{N}}$ is bounded so by hypothesis there exists a weakly almost almost convergent subsequence $\left(x_{n_k}\right)_{k\in\mathbb{N}}$ to some $x\in X$. Observe then that $y=x\in X$.
\item Every functional on $X$ is norm-attaining: Indeed, let $f \in X^*$. For each natural $n$ we can find $x_n \in \mathsf{S}_X$ such that $\left\|f\right\|-\frac{1}{n} \leqslant f\left(x_n\right) \leqslant \left\|f\right\|.$ Since $\left(x_n\right)_{n\in\mathbb{N}}$ is bounded, there exists a subsequence
$\left(x_{n_j}\right)_{j\in\mathbb{N}}$ such that $\displaystyle{\wAClim_{j\to\infty} x_{n_j}=x_0}$ for some $x_0 \in X$. Notice that $x_0 \in \mathrm{B}_X$. On the other hand, $$\mathrm{AC}\lim_{j\to\infty} \left ( \left\|f\right\|-\frac{1}{n_j} \right ) \leqslant \mathrm{AC}\lim_{j\to\infty} f\left(x_{n_j}\right) \leqslant \mathrm{AC}\lim_{j\to\infty} \left\|f\right\|,$$ that is, $\displaystyle{\left\|f\right\| \leqslant \mathrm{AC}\lim_{j\to\infty} f\left(x_{n_j}\right) \leqslant \left\|f\right\|}$. Since $\displaystyle{\mathrm{AC}\lim_{j\to\infty} f\left(x_{n_j}\right)=f\left(x_0\right)}$, we conclude that $f\left(x_0\right)=\left\|f\right\|$.
\end{enumerate}
In accordance to the James' characterization of reflexivity (see \cite{J}), we deduce that $X$ is reflexive.
\end{proof}

\begin{theorem}
A normed space $X$ has the Schur property if and only if $X$ enjoys the following two properties:
\begin{itemize}
\item Every sequence in $X$ whose subsequences are almost convergent to the same limit is convergent.
\item Every weakly almost convergent sequence in $X$ is almost convergent.
\end{itemize}
\end{theorem}

\begin{proof}
In the first place, suppose that $X$ enjoys both properties. Let $\left(x_n\right)_{n\in\mathbb{N}}$ be a weakly convergent sequence in $X$. From the second property we deduce that $\left(x_n\right)_{n\in\mathbb{N}}$ is almost convergent, and thus, all of its subsequences are almost convergent to the same limit. By the first property, $\left(x_n\right)_{n\in\mathbb{N}}$ is convergent.

Conversely, suppose that $X$ has the Schur property. In accordance with {\em 1} of Corollary \ref{nota1}, $X$ enjoys the first property. We will show now that $X$ enjoys the second one. Let $\left(a_i\right)_{i\in\mathbb{N}}$ be a sequence in $X$ which is weakly almost convergent. Without lack of generality we may assume that $\displaystyle{ \wAC\lim_{i\to\infty} a_i=0}$. For each $f \in X^*$ it is verified that $\displaystyle{\lim_{i \to \infty} f\left(x_i^n \right)=0}$ uniformly in $n \in \mathbb{N}$, where we have established that $\displaystyle{x^n_i:=\frac{1}{i+1} \sum_{k=0}^i a_{n+k}}$ for every $n,i\in\mathbb{N}$ in order to simplify. Since $X$ has the Schur property, $\displaystyle{\lim_{i \to \infty} x_i^n=0}$ for every $n \in \mathbb{N}$, so we will conclude the proof by showing that $\displaystyle{\lim_{i \to \infty} x_i^n=0}$ uniformly in $n \in \mathbb{N}$. Suppose not. Then there exists $\varepsilon>0$ for which the following holds:
\begin{enumerate}
\item Take $n_1=1$. There exists $i_1 \in \mathbb{N}$ such that if $i \geqslant i_1$, then $\left\|x^{n_1}_i\right\|<\varepsilon$.
\item There is $n_2>n_1$ such that $y_1:=x^{n_2}_i$ satisfies that $\left\|y_1\right\| \geqslant \varepsilon$ for some $i>i_1$. Besides, there exists $i_2>i_1$ such that if $i\geqslant i_2$, then $\left\|x^{n_2}_i\right\|< \varepsilon$.
\item There is $n_3>n_2$ such that $y_2:=x^{n_3}_i$ satisfies that $\left\|y_2\right\| \geqslant \varepsilon$ for some $i>i_2$. Besides, there exists $i_3>i_2$ such that if $i\geqslant i_3$, then $\left\|x^{n_3}_i\right\|< \varepsilon$.
\item And so on.
\end{enumerate}
In this manner, two sequences $\left(i_k\right)_{k\in\mathbb{N}}\subset \mathbb{N}$ and $\left(y_k\right)_{k\in\mathbb{N}}\subset X$ are found enjoying the following:
\begin{enumerate}
\item $\left\|y_k\right\| \geqslant \varepsilon$ for each $k \in\mathbb{N}$.
\item $\left(i_k\right)_{k\in\mathbb{N}}$ is strictly increasing.
\item For every $k\in \mathbb{N}$ there exists $i \in\left[i_k,i_{k+1}\right]$ such that $y_k=x^{n_{k+1}}_i$.
\end{enumerate}
We will show next that $\displaystyle{w\lim_{k\to\infty} y_k=0}$. Take $f \in X^*$ and $\eta>0$. By hypothesis, there exists $i' \in \mathbb{N}$ so that if $i \geqslant i'$, then $\left|f\left(x^n_i\right)\right|< \eta$ for every $n \in \mathbb{N}$. On the other hand, there will exist $k_0\in\mathbb{N}$ such that if $k \geqslant k_0$, then $i_k > i'$. So $\left|f\left(y_k\right)\right|< \eta$ for $k \geqslant k_0$. Since $X$ has the Schur property, we have that $\displaystyle{\lim_{k\to\infty}y_k=0}$, which contradicts that $\left\|y_k\right\|\geqslant \varepsilon$ for each $k \in \mathbb{N}$.
    \end{proof}

\begin{theorem}
A normed space $X$ has the Grothendieck property if and only if $w^*ac\left(X^*\right)=wac\left(X^*\right)$.
\end{theorem}
\begin{proof}
Firstly, suppose that $w^*ac\left(X^*\right)=wac\left(X^*\right)$. Let $\left(f_n\right)_{n\in\mathbb{N}} \subset X^*$ be such that $\displaystyle{\ws_{n\to\infty} f_n=f_0}$. For every subsequence $\left(f_{n_j}\right)_{j\in\mathbb{N}}$ we have that $\displaystyle{\ws_{j\to\infty} f_{n_j}=f_0}$, so $\displaystyle{\wsAC\lim_{j\to\infty} f_{n_j}=f_0}$, which means that $\displaystyle{\wAClim_{j\to\infty} f_{n_j}=f_0}$. We conclude by Corollary \ref{nota1} that $\displaystyle{\w_{n\to\infty}f_n=f_0}$. As a consequence, $X$ has the Grothendieck property.

Conversely, assume that $X$ has the Grothendieck property. Consider a sequence $\left(f_i\right)_{i\in\mathbb{N}} \subset X^*$ such that $\displaystyle{\wsAClim_{i\to\infty} f_i=0}$. For each $x \in X$ it is verified that $\displaystyle{\lim_{i \to \infty} F_i^n\left(x\right) =0}$ uniformly in $n \in \mathbb{N}$, where we have established that $\displaystyle{F^n_i:=\frac{1}{i+1} \sum_{k=0}^i f_{n+k}}$ for every $n,i\in \mathbb{N}$ in order to simplify. We will conclude this part of the proof by showing that $\displaystyle{\lim_{i \to \infty}g\left(F_i^n\right)=0}$ uniformly in $n \in \mathbb{N}$ for each $g \in X^{**}$. Suppose not. Then there exist $g \in X^{**}$ and $\varepsilon>0$ satisfying the following:\begin{enumerate}
\item Take $n_1=1$. There exists $i_1 \in \mathbb{N}$ so that if $i \geqslant i_1$, then $\left|g\left(F^1_i\right)\right|<\varepsilon$.
\item There is $n_2>n_1$ such that $y_1:=F^{n_2}_i$ satisfies that $\left|g\left(y_1\right)\right| \geqslant \varepsilon$ for some $i>i_1$. Besides, there exists $i_2>i_1$ such that if $i\geqslant i_2$, then $\left|g\left(F^{n_2}_i\right)\right|< \varepsilon$.
\item There is $n_3>n_2$ such that $y_2:=F^{n_3}_i$ satisfies that $\left|g\left(y_2\right)\right| \geqslant \varepsilon$ for some $i>i_2$. Besides, there exists $i_3>i_2$ such that if $i\geqslant i_3$, then $\left|g\left(F^{n_3}_i\right)\right|< \varepsilon$.
\item And so on.
\end{enumerate}
In this manner, two sequences $\left(i_k\right)_{k\in\mathbb{N}}\subset \mathbb{N}$ and $\left(y_k\right)_{k\in\mathbb{N}}\subset X$ are found satisfying the following:
\begin{enumerate}
\item $\left|g\left(y_k\right)\right| \geqslant \varepsilon$ for each $k \in\mathbb{N}$.
\item $\left(i_k\right)_{k\in\mathbb{N}}$ is strictly increasing.
\item For each $k \in \mathbb{N}$ there exists $i \in\left[i_k,i_{k+1}\right]$ such that $y_k=F^{n_{k+1}}_i$.
\end{enumerate}
We will prove now that $\displaystyle{\ws_{k\to\infty} y_k=0}$. Take $x \in X$ and $\eta>0$. By hypothesis, there exists $i' \in \mathbb{N}$ such that if $i \geqslant i'$, then $\left|F^n_i\left(x\right)\right|< \eta$ for every $n \in \mathbb{N}$. On the other hand, there will exist $k_0\in\mathbb{N}$ such that if $k \geqslant k_0$, then $i_k > i'$. So $\left|y_k\left(x\right)\right|< \eta$ for $k \geqslant k_0$. Since $X$ has the Grothendieck property, we have that $\displaystyle{\w_{k\to\infty}y_k=0}$, which contradicts that $\left|g\left(y_k\right)\right| \geqslant \varepsilon$ for each $k \in \mathbb{N}$.
\end{proof}

\begin{theorem}
A Banach space $X$ has a copy of $c_0$ if and only if there exists a sequence $\left(x_i\right)_{i\in\mathbb{N}}\in \ell_\infty\left(X\right)\setminus c_0\left(X\right)$ satisfying that for every infinite set $M \subset \mathbb{N}$ there exists $P \subset M$ infinite such that $\displaystyle{\sum_{i \in P} x_i}$ has bounded partial sums.
\end{theorem}

\begin{proof}
If $X$ has a copy of $c_0$, then $\left(x_i\right)_{i\in\mathbb{N}}$ can be taken as the canonical basis of $c_0$.

Conversely, assume the existence of a sequence $\left(x_i\right)_{i\in\mathbb{N}}\in X^{\mathbb{N}}\setminus c_0\left(X\right)$ satisfying that for every infinite set $M \subset \mathbb{N}$ there exists $P \subset M$ infinite such that $\displaystyle{\sum_{i \in P} x_i}$ has bounded partial sums. By Corollary \ref{nota1} and by the fact that $bps\left(X\right)\subset ac_0\left(X\right)$ (see Lemma \ref{jaja}), we conclude that $\displaystyle{\w_{i\to\infty} x_i=0}$. Since $\left(x_i\right)_{i\in\mathbb{N}}$ is not convergent to $0$, there exists $A\subset \mathbb{N}$ infinite and $\delta >0$ such that $\left\|x_i\right\| > \delta$ for every $i \in A$. According to a result proved by Bessaga-Pelczynski (see \cite{Bessaga}), there exists $B \subset A$ such that $\left(x_i\right)_{i \in B}$ is a basic sequence. A dichotomy result by Odell (see \cite{Odell}) tells us that only one of the following conditions is satisfied:
\begin{enumerate}
\item There exists $C \subset B$ infinite such that $\left(x_i\right)_{i \in C}$ is equivalent to the basis of canonical vectors in $c_0$.
\item There exists $C \subset B$ infinite such that for every sequence $\left(\alpha_i\right)_{i \in C}$ of real numbers which is not convergent to zero, the series $\sum_{i \in C} \alpha_ix_i$ does not have bounded partial sums.
\end{enumerate}
Thus, we deduce that $X$ has a copy of $c_0$.
\end{proof}

\section{Vector-valued Lorentz's theorem}\label{secLIC}

The purpose of this section is to obtain a vector-valued version of Lorentz's almost convergence intrinsic characterization (see Theorem \ref{LIC} or \cite[Theorem 1]{Lorentz}). Our technique is totally different from the original one used by Lorentz in \cite[Theorem 1]{Lorentz} and much more simple since it does not involve the usual order relation in the real line.

\subsection{A total extension for the $ac_0(X)$ case}

For the null almost convergence we obtain a full extension of Theorem \ref{LIC} to the vector-valued case.

\begin{theorem}\label{aqui}
Let $\left(x_n\right)_{n\in \mathbb{N}}$ be a bounded sequence in $X$. The following conditions are equivalent:
\begin{enumerate}
\item $\left(x_n\right)_{n\in \mathbb{N}}$ is a null almost convergent sequence in $X$.
\item $T\left( \left(x_n\right)_{n\in \mathbb{N}} \right) = 0$ for all $T\in  \mathcal{N}_X$.
\end{enumerate}
\end{theorem}

\begin{proof}
\mbox{}
\begin{enumerate}
\item[{\em 1} $\Rightarrow$ {\em 2}] Fix $T\in \mathcal{N}_X\setminus\left\{0\right\}$. Denote $s:=T\left( \left(x_n\right)_{n\in \mathbb{N}} \right)$. We will show that $\left\|s\right\|<\varepsilon$ for every $\varepsilon >0$. So, fix an arbitrary $\varepsilon >0$. There exists $p\in \mathbb{N}$ such that $$\left\|\frac{1}{p+1}\sum_{k=0}^px_{n+k}\right\|<\frac{\varepsilon}{\left\|T\right\|}$$ for all $n\in \mathbb{N}$. Observe that
$$
\begin{array}{rcl}
s & = & T\left(x_1,x_2,x_3,x_4,x_5,\dots\right)\\
s & = & T\left(x_2,x_3,x_4,x_5,x_6,\dots\right)\\
s & = & T\left(x_3,x_4,x_5,x_6,x_7,\dots\right)\\
\vdots & \vdots & \vdots \\
s & = & T\left(x_{p+1},x_{p+2},x_{p+3},x_{p+4},x_{p+5},\dots\right).
\end{array}
$$ Therefore $$\left(p+1\right)s=T\left(x_1+\cdots +x_{p+1},x_2 + \cdots +x_{p+2},\dots\right),$$ that is $$s=T\left(\frac{x_1+\cdots +x_{p+1}}{p+1},\frac{x_2 + \cdots +x_{p+2}}{p+1},\dots\right),$$ and hence
\begin{eqnarray*}
\left\|s\right\|&=&\left\|T\left(\frac{x_1+\cdots +x_{p+1}}{p+1},\frac{x_2 + \cdots +x_{p+2}}{p+1},\dots\right)\right\|\\
& \leq & \left\|T\right\| \left\| \left( \frac{1}{p+1}\sum_{k=0}^px_{n+k} \right)_{n\in \mathbb{N}}\right\|_{\infty}\\
&< & \varepsilon.
\end{eqnarray*}
\item[{\em 2} $\Rightarrow$ {\em 1}] Suppose to the contrary that $\left(x_n\right)_{n\in\mathbb{N}}$ is not almost convergent to $0$. Note that then $$d\left(\left(x_n\right)_{n\in\mathbb{N}},bps(X)\right)>0$$ in virtue of Lemma \ref{jaja} and the fact that $ac_0(X)$ is always closed in $\ell_\infty(X)$. By bearing in mind a corollary of the Hahn-Banach Extension Theorem there exists $f\in\E_{\ell_\infty(X)^*}$ verifying that $$f\left(\left(x_n\right)_{n\in\mathbb{N}}\right)=d\left(\left(x_n\right)_{n\in\mathbb{N}},bps(X)\right)$$
and $$f\left(bps(X)\right)=\{0\}.$$ Consider the following norm-$1$ continuous linear operator
$$
\begin{array}{rrcl}
S : &  \ell_\infty\left(X\right) &\to&  X\\
&\left(z_n\right)_{n\in\mathbb{N}} &\mapsto & f\left(\left(z_n\right)_{n\in\mathbb{N}}\right) x,
\end{array}
$$
where $x\in \E_X$ can be chosen arbitrarily. Note that $S \in \mathcal{N}_X$ in virtue of Proposition \ref{prim}, which is a contradiction since $S\left(\left(x_n\right)_{n\in \mathbb{N}}\right) =f\left(\left(x_n\right)_{n\in\mathbb{N}}\right)x \neq 0$.
\end{enumerate}
\end{proof}

\begin{corollary}\label{aquitepillo3}
\mbox{}
\begin{enumerate}
\item $\displaystyle{ac_0(X)=\bigcap_{T\in\mathcal{N}_X}\ker(T)}$.
\item If $\left(x_n\right)_{n\in \mathbb{N}}$ is a sequence in $X$ almost convergent to $x\in X$, then $T\left( \left(x_n\right)_{n\in \mathbb{N}} \right) = x$ for all $T\in \mathcal{N}_X\cap \mathcal{L}_X$. In particular, $\varphi\left( \left(x_n\right)_{n\in \mathbb{N}} \right) = x$ for all $\varphi\in \BL\left(X\right).$
\item $\BL\left(X\right)=\mathcal{HB}\left(\AClim\right)$.
\end{enumerate}
\end{corollary}

\subsection{A partial extension for the $ac(X)$ case}

In the previous subsection we have proved that if $\left(x_n\right)_{n\in \mathbb{N}}$ is almost convergent to $x$, then $\varphi\left( \left(x_n\right)_{n\in \mathbb{N}} \right) = x$ for all $\varphi\in \BL\left(X\right).$ We will obtain now an approach to the converse of the previous assertion.

We refer the reader to Subsection \ref{contlin} for a wide perspective on injective Banach spaces.

\begin{theorem}\label{pollo}
Let $\left(x_n\right)_{n\in\mathbb{N}}$ be a bounded sequence in an injective Banach space $X$ and consider $x\in X$. If $T\left(\left(x_n\right)_{n\in\mathbb{N}}\right)=x$ for every $T\in\mathcal{N}_X\cap\mathcal{L}_X$, then $\left(x_n\right)_{n\in\mathbb{N}}$ is almost convergent to $x$.
\end{theorem}

\begin{proof}
Notice that we can assume without any loss that $x=0$. Suppose to the contrary that $$\left(x_n\right)_{n\in\mathbb{N}}\in \bigcap \left\{\ker\left(T\right):T\in\mathcal{N}_X\cap\mathcal{L}_X\right\}\setminus ac_0\left(X\right).$$ Notice that $\left(x_n\right)_{n\in\mathbb{N}}\notin ac\left(X\right)$ in virtue of Corollary \ref{aquitepillo3}. In order to reach a contradiction, the natural thing is to consider the following continuous linear map
\begin{equation}\label{S}
\begin{array}{rrcl}
S : &  ac\left(X\right) \oplus \mathbb{R}\left(x_n\right)_{n\in\mathbb{N}} &\to& X\\
&\left(y_n +\lambda x_n \right)_{n\in\mathbb{N}}&\mapsto & \displaystyle{\mathrm{AC}\lim_{n\to \infty}y_n+ \lambda a},
\end{array}
\end{equation}
where $a\in X\setminus \left\{0\right\}$ can be chosen arbitrarily. By hypothesis we can extend $S$ to a continuous linear operator $\widehat{S}:\ell_{\infty}\left(X\right)\to X$. We will finish the proof by showing that $\widehat{S} \in \mathcal{N}_X\cap \mathcal{L}_X$ (which contradicts the fact that $S\left(\left(x_n\right)_{n\in \mathbb{N}}\right) =a \neq 0$):
\begin{itemize}
\item[$\bullet$] $\widehat{S}\in \mathcal{N}_X$. Indeed, $\widehat{S}|_{ac\left(X\right)}=\mathrm{AC}\lim$, therefore $\widehat{S}$ is invariant under the shift operator on $\ell_\infty\left(X\right)$ in virtue of the fact that $bps\left(X\right)\subseteq ac_{0}\left(X\right)$ (see Lemma \ref{jaja}).
\item[$\bullet$] $\widehat{S}\in \mathcal{L}_X$. Indeed, it only suffices to realize that $$\widehat{S}|_{c\left(X\right)}=\left(\widehat{S}|_{ac\left(X\right)}\right)|_{c\left(X\right)}=\mathrm{AC}\lim|_{c\left(X\right)}=\lim.$$
\end{itemize}
\end{proof}

\subsection{A partial extension for the $ac\left(X^*\right)$ case}

We refer the reader to Subsection \ref{filters} for the necessary background involved in the following theorem.

\begin{theorem}\label{pollo2}
Let $\left(x^*_n\right)_{n\in\mathbb{N}}$ be a bounded sequence in $X^*$ and consider $x^*\in X^*$. If $\varphi\left(\left(x^*_n\right)_{n\in\mathbb{N}}\right)=x^*$ for every $\varphi\in\BL\left(X^*\right)$, then $x^*\in \mathrm{cl}_{w^*}\left(\left\{\frac{x^*_1+\cdots+x^*_n}{n}:n\in\mathbb{N}\right\}\right)$.
\end{theorem}

\begin{proof}
Let $\mathcal{U}$ be any free ultrafilter of $\mathbb{N}$. According to Theorem \ref{dual} we have that $\varphi \in \BL\left(X^*\right),$ where $$\begin{array}{rrcl}\varphi:&\ell_{\infty}\left(X^*\right) & \to & X^*\\&\left(y^*_n\right)_{n\in \mathbb{N}}&\mapsto&\displaystyle{\mathcal{U}\lim \frac{y^*_1+\cdots +y^*_n}{n}}.\end{array}$$ By hypothesis, $$x^*=\varphi\left(\left(x^*_n\right)_{n\in\mathbb{N}}\right)=\mathcal{U}\lim \frac{x^*_1+\cdots + x^*_n}{n}.$$ Finally, by taking into consideration {\em 2} of Lemma \ref{engaya}, we deduce that $$x^*\in \mathrm{cl}_{w^*}\left(\left\{\frac{x^*_1+\cdots + x^*_n}{n}:n\in\mathbb{N}\right\}\right).$$
\end{proof}

We remind the reader that duals of separable spaces have $w^*$-metrizable balls.

\begin{corollary}
If $X$ is separable and $\left(x^*_n\right)_{n\in\mathbb{N}}$ is a bounded sequence in $X^*$ almost convergent to $x^*\in X^*$, then there exists a subsequence of $\left(\frac{x^*_1+\cdots + x^*_n}{n} \right)_{n\in\mathbb{N}}$ $w^*$-convergent to $x^*$.
\end{corollary}

\begin{proof}
In virtue of Corollary \ref{aquitepillo3}(2) we have that $\varphi\left(\left(x^*_n\right)_{n\in\mathbb{N}}\right)=x^*$ for every $\varphi\in\BL\left(X^*\right)$. By applying Theorem \ref{pollo2} we deduce that $x^*\in \mathrm{cl}_{w^*}\left(\left\{\frac{x^*_1+\cdots + x^*_n}{n}:n\in\mathbb{N}\right\}\right).$ Finally, by bearing in mind that balls of duals of separable spaces are metrizable when endowed with the $w^*$-topology, we finally conclude the existence of a subsequence of $\left(\frac{x^*_1+\cdots + x^*_n}{n} \right)_{n\in\mathbb{N}}$ $w^*$-convergent to $x^*$.
\end{proof}

\begin{scholium}
In a finite dimensional normed space, every almost convergent sequence verifies that the sequence of its means has a further subsequence convergent to the almost limit.
\end{scholium}

\section{Almost convergence and Banach limits}

In general terms, if $\mathcal{C}$ is a non-empty subset of $\mathcal{L}\left(\ell_\infty(X),X\right)$, then we can define the $\mathcal{C}$-convergence as follows: {\em a sequence $(x_n)_{n\in\N}$ is $\mathcal{C}$-convergent to $x$ provided that $T\left((x_n)_{n\in\N}\right)=x$ for all $T\in\mathcal{C}$}.

This procedure is the one followed by Lorentz in \cite{Lorentz} to define the scalar-valued almost convergence.

The convergence space associated to $\mathcal{C}$, or $\mathcal{C}$-convergence space, is given by
\begin{eqnarray*}
c_{\mathcal{C}}(X)&=&\bigcup\left\{ W\subseteq \ell_\infty\left(X\right): T,S\in\mathcal{C}\Rightarrow T|_W=S|_W\right\}\\
&=&\bigcap\left\{\ker\left(T-S\right):T,S\in\mathcal{C}\right\}.
\end{eqnarray*}

As the reader can observe, the $\mathcal{C}$-convergence space is a closed vector subspace of $\ell_\infty(X)$.

The $\mathcal{C}$-limit function is defined as \begin{equation}
\begin{array}{rrcl} \mathcal{C}\lim : & c_{\mathcal{C}}\left(X\right) & \to & X\\&\left(x_n\right)_{n\in\mathbb{N}}&\mapsto&\displaystyle{\mathcal{C}\lim_{n\to\infty}x_n:=T\left(\left(x_n\right)_{n\in\mathbb{N}}\right),}\end{array}
\end{equation} where $T$ is any element of $\mathcal{C}$, and it is a continuous linear operator such that  $$\mathcal{C}\cap\E_{\mathcal{L}\left(\ell_\infty(X),X\right)}\subseteq
\mathcal{HB}\left(\mathcal{C}\lim\right)$$ provided that $\left\|\mathcal{C}\lim\right\|=1$.

\subsection{A superspace of $ac(X)$ defined by the Banach limits}

On those spaces admitting vector-valued Banach limits, the $\BL(X)$-convergence space makes sense to be taken into account: 
\begin{eqnarray*}
qc\left(X\right)&:=& \bigcup\left\{ W\subseteq \ell_\infty\left(X\right): T,S\in\BL\left(X\right)\Rightarrow T|_W=S|_W\right\}\\
&=& \bigcap\left\{\ker\left(T-S\right):T,S\in\BL\left(X\right)\right\}.
\end{eqnarray*}

Note that we have chosen a different notation for the above space just for simplicity. When $X=\mathbb{R}$ we will simply write $qc$. Easy verifiable properties of the previously defined space follow in the next proposition, the details of its proof we spare to the reader (see Corollary \ref{aquitepillo3}, Corollary \ref{helpazo}, and Theorem \ref{LIC} for help).

\begin{proposition}\label{qc}
Assume that $\BL\left(X\right)\neq \varnothing$.
\begin{enumerate}
\item $ac\left(X\right)\subseteq qc\left(X\right)$.
\item $qc\left(X\right)=\ell_\infty\left(X\right)$ if and only if $\BL\left(X\right)$ is a singleton.
\end{enumerate}
\end{proposition}

As of today, we are unaware of the existence of normed spaces admitting only one vector-valued Banach limit.

The following theorem, whose proof we also omit (keep in mind Theorem \ref{jopa}), declares the importance of $qc(X)$.

\begin{theorem}\label{cagada}
The following conditions are equivalent:
\begin{itemize}
\item $ac\left(X\right)= qc\left(X\right)$.
\item $\displaystyle{\AClim_{n\to\infty} x_n = x}$ if and only if $T\left(\left(x_n\right)\right)_{n\in\mathbb{N}}=x$ for all $T\in\BL(X)$.
\end{itemize}
\end{theorem}

Another immediate corollary of the previous theorem and Theorem \ref{LIC} is the fact that $ac= qc$.

The details of the proof of the next result are also spared to the reader because of its obviousness if bearing in mind Corollary \ref{aquitepillo3}.

\begin{proposition}
Assume that $\BL\left(X\right)\neq \varnothing$.
\begin{enumerate}
\item The $\mathrm{QC}$-limit function, defined as
\begin{equation}\label{qclim}
\begin{array}{rrcl} \mathrm{QC}\lim : & qc\left(X\right) & \to & X\\&\left(x_n\right)_{n\in\mathbb{N}}&\mapsto&\displaystyle{\mathrm{QC}\lim_{n\to\infty}x_n:=T\left(\left(x_n\right)_{n\in\mathbb{N}}\right),}\end{array}
\end{equation} where $T$ is any element of $\BL\left(X\right)$, is a norm-$1$ continuous linear operator such that $\mathrm{QC}\lim|_{ac\left(X\right)}=\mathrm{AC}\lim$.
\item $\mathcal{HB}\left(\mathrm{QC}\lim\right)=\BL(X)=\mathcal{HB}\left(\mathrm{AC}\lim\right)$.
\end{enumerate}
\end{proposition}

\subsection{Separating sets}

The concept of separating set was originally introduced in \cite{SS2} for scalar-valued Banach limits. However, they can be defined the same way for vector-valued Banach limits.

\begin{definition}[Semenov and Sukovech, 2013; \cite{SS2}]
A non-empty subset $G$ of $\ell_\infty\left(X\right)$ is called separating provided that the following condition holds: if $T,S\in\BL\left(X\right)$ are so that $T|_G=S|_G$, then $T=S$.
\end{definition}

\begin{proposition}\label{Gsep}
Assume that $\BL\left(X\right)\neq \varnothing$. Let $G$ be a non-empty subset of $\ell_\infty\left(X\right)$. Then:
\begin{enumerate}
\item If $qc\left(X\right)+\mathrm{span}\left(G\right)$ is dense in $\ell_\infty\left(X\right)$, then $G$ is separating.
\item If $G\subseteq qc\left(X\right)$ and $q\left(X\right)\neq \ell_\infty\left(X\right)$, then $G$ is not separating.
\end{enumerate}
\end{proposition}

\begin{proof}
\mbox{}
\begin{enumerate}
\item Let $T,S\in\BL\left(X\right)$ such that $T|_G=S|_G$. Notice in first place that $T|_{qc\left(X\right)}=S|_{qc\left(X\right)}$ by definition of $qc\left(X\right)$. Now by linearity and continuity we immediately deduce that $T=S$.
\item It is sufficient to consider any $\left(x_n\right)_{n\in\mathbb{N}}\in \ell_\infty\left(X\right)\setminus qc\left(X\right)$. By definition of $qc\left(X\right)$ there must exist $T, S\in \BL\left(X\right)$ such that $T\left( \left(x_n\right)_{n\in\mathbb{N}}\right)\neq S\left( \left(x_n\right)_{n\in\mathbb{N}}\right)$. However, $T|_G=S|_G$ because $G\subseteq qc\left(X\right)$, therefore $G$ is not separating.
\end{enumerate}
\end{proof}

The following result is a direct consequence of Proposition \ref{qc} together with Proposition \ref{Gsep}.

\begin{corollary}
Assume that $\BL\left(X\right)\neq \varnothing$. The following conditions are equivalent:
\begin{enumerate}
\item $\BL\left(X\right)$ is a singleton.
\item Every non-empty subset of $\ell_\infty\left(X\right)$ is separating.
\item Every non-empty subset of $qc\left(X\right)$ is separating.
\item Every non-empty subset of $ac\left(X\right)$ is separating.
\item There exists a non-empty subset of $ac\left(X\right)$ which is separating
\item There exists a non-empty subset of $qc\left(X\right)$ which is separating.
\end{enumerate}
\end{corollary}

As we have already mentioned, as of today we are unaware of the existence of spaces with only one vector-valued Banach limit.

\begin{theorem}
Assume that $\BL\left(X\right)$ is neither empty nor a singleton. A non-empty subset $G$ of $\ell_\infty\left(X\right)$ is separating if and only if $G\cap \left(\ell_\infty\left(X\right)\setminus \ker\left(T-S\right)\right)\neq \varnothing$ for all $T\neq S\in\BL\left(X\right)$.
\end{theorem}

\begin{proof}
Assume first that $G$ is separating. Let $T\neq S\in\BL\left(X\right)$. By assumption there must exist $\left(x_n\right)_{n\in\mathbb{N}}\in G$ such that $T\left(\left(x_n\right)_{n\in\mathbb{N}}\right)\neq S\left(\left(x_n\right)_{n\in\mathbb{N}}\right)$, which means that $\left(x_n\right)_{n\in\mathbb{N}}\in G \cap  \left(\ell_\infty\left(X\right)\setminus \ker\left(T-S\right)\right)$.

Conversely, assume that $G\cap \left(\ell_\infty\left(X\right)\setminus \ker\left(T-S\right)\right)\neq \varnothing$ for all $T\neq S\in\BL\left(X\right)$. Let $T, S\in\BL\left(X\right)$ such that $T|_G=S|_G$ and $T\neq S$. By hypothesis we can find $\left(x_n\right)_{n\in\mathbb{N}}\in G \cap  \left(\ell_\infty\left(X\right)\setminus \ker\left(T-S\right)\right)$, which implies the contradiction that $T|_G\neq S|_G$.
\end{proof}

As an immediate corollary we obtain a sharp characterization of non-separating sets.

\begin{corollary}
Assume that $\BL\left(X\right)$ is neither empty nor a singleton. A non-empty subset $G$ of $\ell_\infty\left(X\right)$ is not separating if and only if $G\subseteq \ker\left(T-S\right)$ for some $T\neq S\in\mathcal{BL}\left(X\right)$.
\end{corollary}

\chapter{Vector-valued almost summability}

The concept of almost summability is nothing but considering the almost convergence for the sequence of partial sums of a given sequence. We refer the reader to \cite{AAGPPF,AizArPer} upon which this chapter relies.

\section{Notion of almost summability}

Since there is only one notion of almost convergence (due to both Lorentz and Boos), there is only one notion of summability (due to Boos, since Lorentz seemed not to be interested in studying the almost convergence of series).

\subsection{Boos' (vector-valued) almost summability}

The almost convergence turns to the almost summability in the context of series.

\begin{definition}[Boos, 2000; \cite{Boos}]
A sequence $ \left(x_n\right)_{n\in\mathbb{N}}$ in $X$
\begin{itemize}
\item is called almost summable provided that the sequence of its partial sums is almost convergent
\begin{itemize}
\item the almost limit of the sequence of partial sums is the almost sum of $ \left(x_n\right)_{n \in \mathbb{N}}$ and
\item it is usually denoted by $\displaystyle{\ACS_{n=1}^\infty x_n = y}$;
\end{itemize}
\item and $ \left(x_n\right)_{n \in \mathbb{N}}$ is called weakly almost summable provided that $\left(f\left(x_n\right) \right)_{n\in\mathbb{N}}$ is almost summable for all $f\in X^*$
\begin{itemize}
\item the weak almost limit of the sequence of partial sums is the weak almost sum of $ \left(x_n\right)_{n \in \mathbb{N}}$ and
\item it is usually denoted by $\displaystyle{\wACS_{n=1}^\infty x_n = y}$.
\end{itemize}
\end{itemize}
\end{definition}

\subsection{Basic properties of almost summability}

Since the almost summability is a particular case of the almost convergence, all the basic properties of the almost convergence apply to the almost summability (when the sequence is question is simply the sequence of partial sums of another sequence).

It is easy to check that, given a series $\sum_{i=1}^{\infty} x_i$ in a normed space $X$ and an element $y\in X$, then:
\begin{itemize}
\item $\displaystyle{\ACS_{i=1}^\infty x_i=y}$ if and only if $$\displaystyle{\lim_{p \to \infty} \left( \sum_{k=1}^nx_k + \frac{1}{p+1} \sum_{k=1}^p
\left(p-k+1\right)\; x_{n+k} \right )=y}$$ uniformly in $n \in \mathbb{N}$.
\item $\displaystyle{\wACS_{i=1}^\infty x_i=y}$ if and only if $$\displaystyle{\lim_{p \to \infty} \left( \sum_{k=1}^nf\left(x_k\right) + \frac{1}{p+1} \sum_{k=1}^p
\left(p-k+1\right)\; f\left(x_{n+k}\right) \right )=f\left(y\right)}$$ uniformly in $n \in \mathbb{N}$, for every $f\in X^*$.
\end{itemize}

\section{Spaces of almost summable sequences}

The almost summability also arises new spaces of particular interest.

\subsection{The spaces $sac\left(X\right)$ and $wsac\left(X\right)$}

We refer the reader to \cite{AizArPer} where the following spaces are defined.

\begin{definition}
If $X$ is a normed space, then:
\begin{itemize}
\item $sac\left(X\right)=\left\{\left(x_i\right)_{i\in\mathbb{N}} \in X^\mathbb{N}: \ACS_{i=1}^{\infty}x_i  \textrm{ exists}\right\}$.
\item $wsac\left(X\right)=\left\{\left(x_i\right)_{i\in \N} \in X^\mathbb{N}: \wACS_{i=1}^\infty x_i  \textrm{ exists}\right\}$.
\end{itemize}
\end{definition}

Proposition \ref{bounded} leads us to the following chain of inclusions: $$sac\left(X\right) \subset wsac\left(X\right)\subset bps\left(X\right).$$

\begin{theorem}
If $X$ is complete, then $sac\left(X\right)$ and $wsac\left(X\right)$ are closed in $bps\left(X\right)$ when endowed with the norm given in \eqref{99}.
\end{theorem}
\begin{proof}
Consider a sequence $\left(x^n\right)_{n\in\mathbb{N}}\subset sac\left(X\right)$ and $x^0\in bps\left(X\right)$ such that $$\displaystyle{\lim_{n\to\infty} \left\| x^n-x^0\right\|=0}.$$ For each $n \in \mathbb{N}$ fixed, we define the sequence $\left(y_i^n\right)_{i\in\mathbb{N}}$ in $X$ given by $y_i^n=\sum_{j=1}^ix_j^n$ for every $i\in\mathbb{N}$. We also define the sequence $\left(y^0_i\right)_{i\in\mathbb{N}}$ in $X$ given by $y^0_i=\sum_{j=1}^i x_j^0$ for every $i\in \mathbb{N}$. We have that $\left(y^n\right)_{n\in\mathbb{N}}\subset ac\left(X\right)$, $\left(y^0_i\right)_{i\in\mathbb{N}} \in \ell_\infty\left(X\right)$, and $\displaystyle{\lim_{n\to\infty} \left\| y^n-y^0\right\|_\infty=0}$. Therefore, $y^0 \in ac\left(X\right)$ in virtue of Theorem \ref{completitud} and hence $x^0 \in sac\left(X\right)$. Similarly, $w sac\left(X\right)$ is closed in $bps\left(X\right)$ endowed with the norm given in \eqref{99}.
\end{proof}

Any normed space $X$ is linearly isometric to $\left\{\left(\frac{1}{2^n}x\right)_{n\in\mathbb{N}}:x\in X\right\}$, which, in fact, is a closed subspace of $bps\left(X\right)$ endowed with the norm given in (\ref{99}). This simple observation saves us from providing the proof of the following corollary.

\begin{corollary}
The following conditions are equivalent:
\begin{enumerate}
\item $sac\left(X\right)$ is complete.
\item $wsac\left(X\right)$ is complete.
\item $X$ is complete.
\end{enumerate}
\end{corollary}

\subsection{The space $w^*sac\left(X^*\right)$}

The weak-star version of the almost summability is due now.

\begin{definition}[Boos, 2000; \cite{Boos}]
A sequence $\left(x^*_i\right)_{i\in\mathbb{N}}$ in $X^*$ is called weakly-star almost summable when there exists $y^* \in X^*$ such that $$\displaystyle{\ACS_{i=1}^\infty x^*_i\left(x\right)=y^*\left(x\right)}$$ for all $x\in X$
\begin{itemize}
\item $y^*$ is called the weak-star almost sum of $\left(x^*_i\right)_{i\in\mathbb{N}}$
\item it is usually denoted by $\displaystyle{\wsACS_{i=1}^\infty x_i^*=y^*}$.
\end{itemize}
The space of $w^*$-almost summable sequences is defined as: $$w^* sac\left(X^*\right):=\left\{\left(x^*_i\right)_{i\in\mathbb{N}} \in \left(X^*\right)^\mathbb{N}: \wsACS_{i=1}^\infty x^*_i \text{ exists}\right\}.$$
\end{definition}

Obviously, every weakly-star summable sequence is weakly-star almost summable and every weakly almost summable sequence in a dual space is weakly-star almost summable.

If $X$ is barrelled, then it is easy to prove that $w^* sac\left(X^*\right)$ is a closed subspace of $bps\left(X^*\right)$. In case $X$ is not, then $w^* sac\left(X^*\right) \cap bps\left(X^*\right)$ is closed in $bps\left(X^*\right)$.

\section{Almost summing multiplier spaces}

Following \cite{PBA} (see Subsection \ref{vectorseries}) we can define spaces of almost summing multipliers, which turn out to be subspaces of $\ell_\infty$. They are not to be mistaken with the spaces of almost summable sequences previously defined.

\subsection{The spaces $\mathcal{S}_{\AC}\left(\sum x_i\right)$ and $\mathcal{S}_{\wAC}\left(\sum x_i\right)$}

It is the time now to define the almost summing multiplier spaces.

\begin{definition}
If $\sum x_i$ is a series in a normed space $X$, then we can define the following subspaces of $\ell_\infty$:
\begin{itemize}
\item $ \mathcal{S}_{\AC}\left(\sum x_i\right)=\left\{ (a_i)_{i \in\N} \in \ell_\infty : \ACS a_ix_i \; \textrm{exists} \right\}$.
\item $\mathcal{S}_{\wAC}\left(\sum x_i\right)=\left\{ (a_i)_{i\in\N} \in \ell_\infty : \wACS a_ix_i \; \textrm{exists} \right\}$.
\end{itemize}
\end{definition}

Obviously, $c_{00}\subseteq \mathcal{S}_{\AC}\left(\sum x_i\right)\subseteq \mathcal{S}_{\wAC}\left(\sum x_i\right)$.

\begin{lemma}\label{resumen}
If $c_0\subseteq \mathcal{S}_{\wAC} \left(\sum x_i\right)$, then $\sum x_i$ is wuC.
\end{lemma}

\begin{proof}
Suppose to the contrary that $\sum x_i$ is not wuC. In accordance with Theorem \ref{DiestelwuC}, there exists $f \in X^*$ verifying that $\sum_{i=1}^\infty |f(x_i)| = +\infty$. Let us proceed as follows:
\begin{itemize}
\item We can choose a natural $n_1$ such that $\sum_{i=1}^{n_1} |f(x_i)| > 2 \cdot 2$ and for $i \in \{ 1, \ldots, n_1 \}$ define $$a_i=\left\{\begin{array}{rl} 1/2& \text{ if } f(x_i) \geqslant 0\\  - 1/2 & \text{ if } f(x_i)<0.\end{array}\right.$$
\item There exists $n_2>n_1$ such that $\sum_{i=n_1+1}^{n_2} |f(x_i)| > 3 \cdot 3$ and for $i \in \{ n_1+1,\ldots, n_2\}$ define $$a_i=\left\{\begin{array}{rl} 1/3& \text{ if } f(x_i) \geqslant 0\\  - 1/3 & \text{ if } f(x_i)<0.\end{array}\right.$$
\item In this manner we obtain an increasing sequence $(n_k)_{k\in\mathbb{N}}$ in $\mathbb{N}$ and a sequence $(a_i)_{i\in\mathbb{N}}\in c_0$ such that $\sum_{i=1}^\infty a_i f(x_i)= +\infty$.
\end{itemize}
Since $(a_i)_{i\in\mathbb{N}} \in \mathcal{S}_{\wAC}(\sum x_i)$ by hypothesis, it follows that $\wACS_{i=1}^\infty a_i x_i$ exists and therefore $\left ( \sum_{i=1}^n a_i f(x_i)\right )_{n\in\mathbb{N}}$ is a bounded sequence, which is a contradiction.
\end{proof}

\begin{theorem}\label{ACSc0}
The following conditions are equivalent:
\begin{enumerate}
\item $X$ is complete.
\item The below assertions are equivalent:
\begin{enumerate}
\item $\sum x_i$ is wuC.
\item $\mathcal{S}_{\AC} \left(\sum x_i\right)$ is complete.
\item $c_0\subseteq \mathcal{S}_{\AC}(\sum x_i)$.
\item $\mathcal{S}_{\wAC} \left(\sum x_i\right)$ is complete.
\item $c_0\subseteq \mathcal{S}_{\wAC}(\sum x_i)$.
\end{enumerate}
\end{enumerate}
\end{theorem}

\begin{proof}
\mbox{}
\begin{enumerate}
\item[{\em 1} $\Rightarrow$ {\em 2}] Attending to Lemma \ref{resumen} and to the chain of inclusions $c_{00}\subseteq \mathcal{S}_{\AC}\left(\sum x_i\right)\subseteq \mathcal{S}_{\wAC}\left(\sum x_i\right)$, it only suffices to show $(a) \Rightarrow (b)$ and $(a) \Rightarrow (d)$. For similarity reasons, we will only show $(a) \Rightarrow (b)$. We will, in fact, prove that $\mathcal{S}_{\AC} \left(\sum x_i\right)$ is closed in $\ell_\infty$. Let $(a^n)_{n\in\N}$ be a sequence in $\mathcal{S}_{\AC} \left(\sum x_i\right)$, with $a^n=(a^n_i)_{i\in \N}$ for each $n \in \mathbb{N}$, and consider $a^0 \in\ell_\infty$ in such a way that $\displaystyle{\lim_{n\to\infty} \|a^n-a^0\|=0}$. We will show that $a^0 \in \mathcal{S}_{\AC} \left(\sum x_i\right)$. In accordance with Theorem \ref{DiestelwuC}, there exists $H>0$ such that $$H \geqslant \sup \left\{ \left\|\sum_{i=1}^n a_ix_i\right\| : n \in \mathbb{N}, |a_i| \leqslant 1, i \in \{1, \ldots, n \} \right\}.$$ For each natural $n$, there exists $y_n \in X$ such that $y_n=\ACS_{i=1}^\infty a^n_ix_i$.
\begin{itemize}
\item We claim that $(y_n)_{n\in\N}$ is a Cauchy sequence. Indeed, if $\varepsilon >0$ is given, there exists an $n_0$ such that if $p,q \geqslant n_0$, then $\| a^p-a^q \|<\varepsilon/3H$. If $p,q \geqslant n_0$ are fixed, there exists $i \in \mathbb{N}$ verifying
\begin{eqnarray*}
\left\|y_p-\left(\sum_{k=1}^j a_k^p x_k + \frac{1}{i+1} \sum_{k=1}^i (i-k+1) a_{j+k}^p x_{j+k}\right)\right\| &<& \frac{\varepsilon}{3}\\
\left\|y_q-\left(\sum_{k=1}^j a_k^q x_k + \frac{1}{i+1} \sum_{k=1}^i (i-k+1) a_{j+k}^q x_{j+k}\right)\right\| &<&\frac{\varepsilon}{3} 
\end{eqnarray*}
for each $j \in N$. Then, if $p,q \geqslant n_0$ we have that
\begin{eqnarray*}
 \|y_p-y_q\| &\leqslant & \left\|y_p-\left(\sum_{k=1}^j a_k^p x_k + \frac{1}{i+1} \sum_{k=1}^i (i-k+1) a_{j+k}^p x_{j+k}\right)\right\|  \\
&+& \left\|y_q-\left(\sum_{k=1}^j a_k^q x_k + \frac{1}{i+1} \sum_{k=1}^i (i-k+1) a_{j+k}^q x_{j+k}\right)\right\| \\
&+& \left\|\sum_{k=1}^j (a_k^p-a_k^q)x_k + \sum_{k=1}^i \frac{i-k+1}{i+1} (a_{j+k}^p-a_{j+k}^q)x_{j+k}\right\| \\
&<& \frac{\varepsilon}{3}+\frac{\varepsilon}{3}+\frac{\varepsilon}{3}\\
&=& \varepsilon.
\end{eqnarray*}
\end{itemize}
Since $X$ is a Banach space, there exists $y_0 \in X$ such that $\displaystyle{\lim_{n\to\infty} \|y_n - y_0\|=0}$. Finally, we will check that $\ACS_{i=1}^\infty a_i^0 x_i=y_0$, that is, $$\lim_{i \to \infty} \left( \sum_{k=1}^j a_k^0 x_k+ \frac{1}{i+1} \sum_{k=1}^i (i-k+1) a_{j+k}^0 x_{j+k}\right )=y_0, \hspace{5pt} \textrm{uniformly in} \; j \in \mathbb{N}.$$ If $\varepsilon>0$ is given, we can fix a natural $n$ such that $\|a^n-a^0\|< \varepsilon/3 H$ and $\|y_n-y_0\|<\varepsilon/3$. Now, we can also fix $i_0$ such that for every $i \geqslant i_0$ it holds that $$\displaystyle{\left\|y_n-\left(\sum_{k=1}^j a_k^n x_k+ \frac{1}{i+1} \sum_{k=1}^i (i-k+1) a_{j+k}^n x_{j+k}\right)\right\| <   \frac{\varepsilon}{3}}$$ for every $j \in \mathbb{N}$. Next, if $i \geqslant i_0$, then it is satisfied that
\begin{eqnarray*}
&&\left \|y_0-\left(\sum_{k=1}^j a_k^0 x_k+ \frac{1}{i+1} \sum_{k=1}^i (i-k+1) a_{j+k}^0 x_{j+k}\right)\right\| \\
&\leqslant & \|y_0-y_n\|  +\left \|y_n-\left(\sum_{k=1}^j a_k^n x_k+ \frac{1}{i+1} \sum_{k=1}^i (i-k+1) a_{j+k}^n x_{j+k}\right)\right\| \\
&+& \left\|\sum_{k=1}^j (a_k^n-a_k^0) x_k + \frac{1}{i+1} \sum_{k=1}^i (i-k+1) (a_{j+k}^n-a_{j+k}^0) x_{j+k}\right\|  \\
&\leqslant & \frac{2 \varepsilon}{3}  + \|a^n-a^0\| \left( \sum_{k=1}^j \frac{(a_k^n-a_k^0)}{\|a^n-a^0\|} x_k + \sum_{k=1}^i \frac{(i-k+1) (a_{j+k}^n-a_{j+k}^0)}{(i+1) \|a^n-a^0\|} x_{j+k} \right) \\
&\leqslant & \frac{2\varepsilon}{3}+\frac{\varepsilon}{3H} H \\
&= &\varepsilon
\end{eqnarray*}
for every $j \in \mathbb{N}$.

\item[{\em 2} $\Rightarrow$ {\em 1}] Suppose to the contrary that $X$ is not complete. There exists a series $\sum x_i$ in $X$ such that $\|x_i\|<\frac{1}{i 2^i}$ and $\sum x_i = x^{**} \in X^{**}\setminus X$. As a consequence, $\ACS x_i = x^{**}$. It is well known that the series $\sum i x_i$ is wuC. Then, by hypothesis, $x^{**}=\ACS \frac{1}{i}ix_i \in X$ since $\left(\frac{1}{i}\right)_{i\in\mathbb{N}}\in c_0$. This is a contradiction.
\end{enumerate}
\end{proof}

\begin{proposition}
The following conditions are equivalent:
\begin{enumerate}
\item $\sum x_i$ is wuC.
\item The linear operator
\begin{equation*}
\begin{array}{rcl}
 \mathcal{S}_{\AC} (\sum x_i)& \to& X\\
(a_i)_{i\in\mathbb{N}}&\mapsto& \ACS_{i=1}^\infty a_i x_i.
\end{array}
\end{equation*}
is continuous.
\item The linear operator
\begin{equation*}
\begin{array}{rcl}
 \mathcal{S}_{\wAC} (\sum x_i)& \to& X\\
(a_i)_{i\in\mathbb{N}}&\mapsto& \wACS_{i=1}^\infty a_i x_i.
\end{array}
\end{equation*}
is continuous.
\end{enumerate}
\end{proposition}

\begin{proof}
For similarity reasons, we will only show the equivalence of {\em 1} and {\em 2}.
\begin{enumerate}
\item[{\em 1} $\Rightarrow$ {\em 2}] If $\sum x_i$ is $wuC$ and $H:= \sup\left\{ \left\|\sum_{i=1}^n a_i x_i\right\| : n \in \mathbb{N}, |a_i| \leqslant 1, i \in \{ 1, \ldots, n\}\right\}$, then $\left\| \ACS_{i=1}^\infty a_i x_i \right\| \leqslant H \|a\|$.

\item[{\em 2} $\Rightarrow$ {\em 1}] Conversely, if $\{ a_1, \ldots, a_j\} \subset [-1,1]$ (and considering $a_i=0$ if $i>j$), then $\left\|\sum_{i=1}^j a_i x_i \right\| =\left \|\ACS_{i=1}^\infty a_i x_i\right\|$ is less than or equal to the norm of the above operator. Finally Theorem \ref{DiestelwuC} applies.

\end{enumerate}

\end{proof}

\subsection{The space $\mathcal{S}_{\wsAC}\left(\sum x^*_i\right) $}

\begin{definition}
If $\sum x_i^*$ is a series in the dual $X^*$ of a normed space, then we can define the following subspace of $\ell_\infty$:
$$\mathcal{S}_{\wsAC}\left(\sum x^*_i\right)=\left\{ (a_i)_{i\in\N} \in \ell_\infty : \wsACS_{i=1}^\infty a_i x^*_i \;\textrm{exists} \right\}.$$
\end{definition}

In a dual space we obviously have that $$c_{00}\subseteq \mathcal{S}_{\AC}\left(\sum x^*_i\right)\subseteq \mathcal{S}_{\wAC}\left(\sum x^*_i\right)\subseteq \mathcal{S}_{\wsAC}\left(\sum x^*_i\right).$$

By bearing in mind the $w^*$-compacity of $\B_{X^*}$, the reader can quickly realize that:

\begin{itemize}
\item If $\sum x^*_i$ is $wuC$, then $\mathcal{S}_{\wsAC}\left(\sum x^*_i\right)=\ell_\infty$.
\item If $\mathcal{S}_{\wsAC}\left(\sum x^*_i\right)=\ell_\infty$, then $\ACS_{i \in M}x^*_i(x)$ exists for all $x \in X$ and all $M \subset \mathbb{N}$.
\end{itemize}

\begin{theorem}
If $X$ is barrelled, then the following conditions are equivalent:
\begin{enumerate}
\item $\sum x^*_i$ is $wuC$.
\item $\mathcal{S}_{\wsAC}\left(\sum x^*_i\right)=\ell_\infty$.
\item $\ACS_{i \in M}x^*_i(x)$ exists for all $x \in X$ and all $M \subset \mathbb{N}$.
\end{enumerate}
\end{theorem}

\begin{proof}
We only need to show {\em 3} $\Rightarrow $ {\em 1}. Effectively, our goal is to use Theorem \ref{DiestelwuC} by proving that $$E:=\left\{ \sum_{i=1}^n a_i x^*_i : n \in \mathbb{N}, |a_i| \leqslant 1, i \in \{1, \ldots, n \}\right\}$$ is bounded. Since $X$ is barrelled, it is sufficient to show that $E$ is pointwise bounded. So, suppose to the contrary that $E$ is not pointwise bounded, that is, there exists $x_0 \in X$ such that $\sum_{i=1}^\infty |x^*_i(x_0)|=+\infty$. Then, we can choose a subset $M \subset \mathbb{N}$ such that $\sum_{i \in M} x^*_i(x_0)=\pm \infty$. However, by hypothesis, $\ACS_{i \in M}x^*_i(x_0)$ exists, which is a contradiction.
\end{proof}

\subsection{The Almost Convergence Orlicz-Pettis Theorem}

An almost convergence version of Theorem \ref{OP} will be provided to conclude the third chapter of this book.

\begin{theorem}\label{ACOPT}
If $X$ is complete and $\sum x_i$ is a series in $X$ such that $\wACS_{i \in M} x_i$ exists for each $M \subset \mathbb{N}$, then $\sum x_i$ is $uc$.
\end{theorem}

\begin{proof}
We will proceed in two steps:
\begin{itemize}
\item First off, we will show that $\sum x_i$ is $wuC$. Indeed, if not, then we can find $M \subset
\mathbb{N}$ and $f \in X^*$ such that $\sum_{i \in M} f(x_i)=+\infty$. By hypothesis, there exists $x_0 \in X$ such that
$\ACS_{i \in M} f(x_i)=f(x_0)$, which is a contradiction.

\item Secondly, we will show that if $M \subset \mathbb{N}$, then
$w\sum_{i \in M} x_i$ exists, which already implies that $\sum x_i$ is $uc$ by applying the classic Orlicz-Pettis Theorem (Theorem \ref{OP}). Effectively, if $M \subset \mathbb{N}$, then there exists $x_0 \in X$ such that $\wACS_{i \in M} x_i =x_0$, but if $f \in X^*$, then $\sum_{i \in M} f(x_i)$ exists and $$\sum_{i \in M} f(x_i)=\ACS_{i \in M} f(x_i)=f(x_0).$$ As a consequence, $w \sum_{i \in M} x_i=x_0$.
\end{itemize}
\end{proof}

The following corollary is deduced as an evident consequence.

\begin{corollary}
If $X$ is complete and $\sum x_i$ is a series in $X$, then the following assertions are equivalent:
\begin{enumerate}
\item $\sum x_i$ is $uc$.
\item $\mathcal{S}_{\AC}\left( \sum x_i\right)=\ell_\infty$.
\item $\mathcal{S}_{\wAC}\left( \sum x_i\right)=\ell_\infty$.
\end{enumerate}
\end{corollary}

\section{Multiplier spaces of almost summing sequences}

In this section we will deal with multiplier spaces of almost summing sequences. Throughout the whole of this section $X$ will be a Banach space and $\mathcal{S}$ a closed subspace of $\ell_\infty$ containing $c_0$. The reason for these impositions on $X$ and $\mathcal{S}$ is to able to rely on Theorem \ref{ACSc0}.

\subsection{The spaces $X_{\AC}({\cal S})$ and $X_{\wAC}({\cal S})$}

As we mentioned right above, the following definition finds part of its origins in Theorem \ref{ACSc0}.

\begin{definition}
If $\mathcal{S}$ is a vector subspace of $\ell_\infty$ containing $c_0$ and $X$ is a Banach space, then the following spaces can be defined:
\begin{enumerate}
\item $X_\AC\left(\mathcal{S}\right) = \left\{ \left(x_i\right)_{i\in\mathbb{N}} \in X^{\mathbb{N}} : \ACS_{i=1}^\infty a_i x_i \; \hbox{exists if} \; \left(a_i\right)_{i\in\mathbb{N}} \in \mathcal{S} \right\}.$
\item $X_{\wAC}\left(\mathcal{S}\right) = \left\{ \left(x_i\right)_{i\in\mathbb{N}} \in X^{\mathbb{N}} : \wACS_{i=1}^\infty a_i x_i \; \hbox{exists if} \; \left(a_i\right)_{i\in\mathbb{N}} \in \mathcal{S} \right\}.$
\end{enumerate}
\end{definition}

Notice that Theorem \ref{ACSc0} allows us to deduce that $$X\left(\ell_\infty\right)\subset X_\AC\left(\mathcal{S}\right)\subset X_\wAC \left(\mathcal{S}\right)\subset X\left(c_0\right),$$ and thus on $X_\AC\left(\mathcal{S}\right)$ and $X_\wAC \left(\mathcal{S}\right)$ we can consider the norm \eqref{ecu1}.

\begin{theorem}
The spaces $X_AC\left(\mathcal{S}\right)$ and $X_\wAC \left(\mathcal{S}\right)$ are complete endowed with the norm given in \eqref{ecu1}.
\end{theorem}

\begin{proof}
We will only prove the completeness of $X_\wAC \left(\mathcal{S}\right)$, for which it will suffice to show that $X_\wAC \left(\mathcal{S},\right)$ is closed in $X\left(c_0\right)$. The completeness of $X_\AC \left(\mathcal{S},\right)$ may be proved in a similar way. Let $\left(x^n\right)_{n\in\mathbb{N}}$ be a sequence in $X_\wAC \left(\mathcal{S}\right)$ for which there exists $x^0 \in X\left(c_0\right)$ such that $\displaystyle{\lim_{n\to \infty} \left\|x^n-x^0\right\|=0}$. Fix $a= \left(a_i\right)_{i\in\mathbb{N}} \in \mathcal{S}\setminus \left\{0 \right \}$ and $n\in \mathbb{N}$. There exists $x_n \in X$ such that for every $f \in X^*$ $$\lim_{i \to \infty} \left( \sum_{k=1}^j a_k f\left(x_k^n\right) + \frac{1}{i+1} \sum_{k=1}^i \left(i-k+1\right) a_{j+k} f\left(x_{j+k}^n\right) \right ) = f\left(x_n\right)$$ uniformly in $j \in \mathbb{N}$. We will show now that $\left(x_n\right)_{n\in\mathbb{N}}$ is a Cauchy sequence. Given $\varepsilon>0$, there must exist $n_0 \in \mathbb{N}$ such that if $p,q \geqslant n_0$, then $\left\|x^p-x^q\right\| \leqslant \frac {\varepsilon}{3 \left\|a\right\|}$. If $p,q \geqslant n_0$ are fixed, then a functional $f \in \mathsf{S}_{X^\ast}$ can be found such that $$\left\|x_p-x_q\right\|= \left|f\left(x_p\right)-f\left(x_q\right)\right|.$$ Besides, there exists $i \in \mathbb{N}$ such that
\begin{equation}
\left| f\left(x_p\right)- \left(\sum_{k=1}^j a_k f\left(x_k^p\right)+ \frac{1}{i+1} \sum_{k=1}^i \left(i-k+1\right) a_{j+k} f\left(x_{j+k}^p\right)\right) \right| < \frac{\varepsilon}{3} \label{ecuacion4}
\end{equation}
and
\begin{equation}
\left| f\left(x_q\right)- \left(\sum_{k=1}^j a_k f\left(x_k^q\right)+ \frac{1}{i+1} \sum_{k=1}^i \left(i-k+1\right) a_{j+k} f\left(x_{j+k}^q\right)\right)\right| < \frac{\varepsilon}{3} \label{ecuacion5}
\end{equation}
are satisfied for every $j \in \mathbb{N}$. Thus
\begin{eqnarray*}
&&\left \|x_p-x_q\right\| \\
&=& \left|f\left(x_p\right)-f\left(x_q\right)\right| \\
&\leqslant& (\ref{ecuacion4}) + (\ref{ecuacion5})\\
& +& \left| \sum_{k=1}^j a_k f\left(x_k^p -x_k^q\right) + \frac{1}{i+1} \sum_{k=1}^i \left(i-k+1\right) a_{j+k}f\left(x_{j+k}^p-x_{j+k}^q\right)\right|  \\
        &\leqslant& \frac{\varepsilon}{3}+ \frac{\varepsilon}{3}+ \left\|x^p-x^q\right\| \|a\| \\
        &\leqslant&        \varepsilon.
\end{eqnarray*}
Since $X$ is complete, there exists $x_0 \in X$ such that $\displaystyle{\lim_{n \to \infty} \left\|x_n-x_0\right\|=0}$. Now, fix $f \in X^* \setminus \left\{0\right\}$ and $\varepsilon>0$. There exists $n \in \mathbb{N}$ such that $$\left\|x^n-x^0\right\| < \frac{\varepsilon}{3 \left\|a\right\|\left\|f\right\|} \text{ and }\left\|x_n-x_0\right\|< \frac {\varepsilon}{3 \left\|f\right\|}.$$ On the other hand, there exists $i_0\in \mathbb{N}$ such that if $i \geqslant i_0$, then $$\left|\left(\sum_{k=1}^j a_k f\left(x_k^n\right)+ \frac{1}{i+1} \sum_{k=1}^i \left(i-k+1\right) a_{j+k} f\left(x_{j+k}^n\right)\right)-f\left(x_n\right) \right| < \frac{\varepsilon}{3}$$ for every $j \in \mathbb{N}$. Finally, we conclude that
    \begin{eqnarray*}
        &&\left|\left(\sum_{k=1}^j a_k f\left(x_k^0\right)+ \frac{1}{i+1} \sum_{k=1}^i \left(i-k+1\right) a_{j+k} f\left(x_{j+k}^0\right)\right)-f\left(x_0\right)\right|  \\
        &\leqslant& \left| \sum_{k=1}^j a_k f\left(x_k^0 -x_k^n\right) + \frac{1}{i+1} \sum_{k=1}^i \left(i-k+1\right) a_{j+k}f\left(x_{j+k}^0-x_{j+k}^n\right)\right| \\
        &+& \left|\left(\sum_{k=1}^j a_k f\left(x_k^n\right)+ \frac{1}{i+1} \sum_{k=1}^i \left(i-k+1\right) a_{j+k}f\left(x_{j+k}^n\right)\right)-f\left(x_n\right)\right| \\
        &+& \left|f\left(x_n\right)-f\left(x_0\right)\right| \\
        &\leqslant& \left\|x^n-x^0\right \| \left\|a\right\| \left\|f\right\|+ \frac{\varepsilon}{3} + \left\|f\right\| \left\|x_n-x_0\right\|\\
        & \leqslant&   \varepsilon
    \end{eqnarray*}
    for each $j \in \mathbb{N}$. In other words, $\wACS_{i=1}^\infty a_i x_i^0=x_0$ and hence $x^0 \in X_\wAC\left(\mathcal{S}\right) $.

\end{proof}

\begin{proposition}\label{cont}
If $\left(x_i\right)_{i\in\mathbb{N}} \in X_\AC\left(\mathcal{S}\right)$ and $\left(y_i\right)_{i\in\mathbb{N}} \in X_\wAC\left(\mathcal{S}\right)$, then the linear maps are $$\begin{array}{rcl}  \mathcal{S}& \to & X\\   \left(a_i\right)_{i\in\mathbb{N}}&\mapsto &\ACS_{i=1}^\infty a_i x_i \end{array}\text{ and } \begin{array}{rcl}  \mathcal{S} & \to & X\\  \left(a_i\right)_{i\in\mathbb{N}}&\mapsto &\wACS_{i=1}^\infty a_i y_i\end{array}$$ are continuous.
\end{proposition}

\begin{proof}
We will only prove the continuity of the second map, since the other one can be proved in a similar way. Suppose that $ \left(a_i\right)_{i\in\mathbb{N}}\in \mathcal{S}$. There exists $f \in \mathsf{S}_{X^\ast}$ such that
\begin{eqnarray*}
\left\| \wACS_{i=1}^\infty a_i y_i \right\| &=& \left|f\left( \wACS_{i=1}^\infty a_i y_i\right)\right|\\
&=&\left|\ACS_{k=1}^\infty a_i f\left(x_i\right)\right|\\
&=&  \lim_{i\to \infty} \left|\sum_{k=1}^j a_k f\left(x_k\right)+ \frac{1}{i+1} \sum_{k=1}^i \left(i-k+1\right) a_{k+j} f\left(x_{k+j}\right)\right|
\end{eqnarray*} uniformly in $j \in \mathbb{N}$. Now, if $i,j \in \mathbb{N}$, then
\begin{eqnarray*}
&&\left|\sum_{k=1}^j a_k f\left(x_k\right)+ \frac{1}{i+1} \sum_{k=1}^i \left(i-k+1\right) a_{k+j}f\left(x_{k+j}\right)\right|\\
&\leqslant& \left\|x\right\|_\infty \left\|a\right\|_\infty \\
&\leqslant& \left\|x\right\|_m \left\|a\right\|_\infty.
\end{eqnarray*}
That is, $\left\| \wACS_{i=1}^\infty a_i y_i \right\| \leqslant \left\|x\right\|_m \left\|a\right\|_\infty$.
\end{proof}

\subsection{Uniform almost summability}

In the following results we deduce the existence of uniform almost convergence from certain situations of point-wise almost convergence. As a consequence, we obtain Hahn-Schur Theorem-like results that generalize other results of uniform convergence of sequences in $X\left(\ell_\infty\right)$ and in $X\left(c_0\right)$ that appear in \cite{AizGPPE, AizGuPer, AizPer, BuWu, Swartz1}. Even more, by making use of any regular matrix summability methods (see \cite[Theorem 3.3 and Theorem 3.5]{AizGPPE}) we show that Hahn-Shur Theorem-like results for sequences in $X\left(c_0\right)$ remain valid for other generalizations of the concept of convergence, for instance, the (Banach-Lorentz) almost convergence, even though this type of convergence is not representable by a matrix summability method.

\begin{theorem}\label{teor}
Consider a Banach space $X$ and a closed vector subspace $\mathcal{S}$ of $\ell_\infty$ containing $c_0$ such that $\mathcal{S}$ is $\ell_\infty$-Grothendieck. If $\left(x^n\right)_{n\in \mathbb{N}}$ is a sequence in $X\left(c_0\right)$ verifying that for all $a=\left(a_i\right)_{i\in \mathbb{N}} \in \mathcal{S}$ we have that $\displaystyle{\lim_{n\to \infty} \wACS_{i=1}^{\infty} a_ix^n_i}$ exists, then there exists $x^0 \in X\left(c_0\right)$ such that $\displaystyle{\lim_{n\to \infty} \left\|x^n-x^0\right\|=0}$ in $X\left(c_0\right)$.
\end{theorem}

\begin{proof}
Assume that $\left(x^n\right)_{n\in \mathbb{N}}$ is not a Cauchy sequence in $X\left(c_0\right)$. We will reach a contradiction. Let $\varepsilon>0$ and consider a strictly increasing sequence $\left(n_k\right)_{k\in \mathbb{N}}$ of natural numbers verifying that $\left\|z^k\right\|>\varepsilon$ for all $k\in \mathbb{N}$ where $z^k:=x^{n_k}-x^{n_{k+1}}$. Fix an arbitrary $k\in \mathbb{N}$. Be aware of the existence of $f_k \in \mathsf{B}_{X^\ast}$ such that $\sum_{j=1}^\infty \left | f_k\left(z^k_j\right)\right | > \varepsilon$. On the other hand, for each $k \in \mathbb{N}$ we will consider the continuous linear map from Proposition \ref{cont}(2.) given by 
\begin{equation*}
\begin{array}{rrcl}
\alpha_{z^k}:&\mathcal{S}& \to &X \\
&\left(a_i\right)_{i\in\mathbb{N}}&\mapsto & \alpha_k\left(\left(a_i\right)_{i\in\mathbb{N}}\right)=\wACS_{i=1}^\infty a_i z^k_i.
\end{array}
\end{equation*}
To simplify, we will denote $\alpha_{z^k}$ by $\alpha_k$. The reader may immediately realize that $$\lim_{k \to \infty} \alpha_k\left(\left(a_i\right)_{i\in\mathbb{N}}\right)= \lim_{k \to \infty} \wACS_{i=1}^\infty a_i x^{n_k}_i- \lim_{k \to \infty} \wACS_{i=1}^\infty a_i x^{n_{k+1}}_i=0$$ for all $\left(a_i\right)_{i\in\mathbb{N}}\in \mathcal{S}$. It follows that $\left(f_k \circ \alpha_k\right)_{k\in \mathbb{N}}$ is a sequence in $\mathcal{S}^\ast$ which is $\omega^\ast$ convergent to zero. By hypothesis, it will be $\sigma\left(\mathcal{S}^\ast,\ell_\infty\right)$-convergent to zero. Equivalently $$\lim_{k \to \infty} \left(f_k \circ \alpha_k\right) \left(\left(a_i\right)_{i\in\mathbb{N}}\right)=0$$ for each $\left(a_i\right)_{i\in\mathbb{N}} \in \ell_\infty$. Hence, we have that $$\lim_{k \to \infty} \sum_{j=1}^\infty a_j \left(f_k \circ \alpha_k\right)\left(e^j\right)=\lim_{k \to \infty} \sum_{j=1}^\infty a_jf_k\left(z^k_j\right)=0$$ for all $\left(a_i\right)_{i\in\mathbb{N}} \in \ell_\infty$. Therefore $\left(\left(f_k\left( z^k_j\right)\right)_{j\in\mathbb{N}}\right)_{k\in\mathbb{N}}$ is a sequence in $\ell_1$ which is $w$-convergent to $0$. In accordance with the classical Schur Theorem, we deduce that $\left(\left(f_k\left( z^k_j\right)\right)_{j\in\mathbb{N}}\right)_{k\in\mathbb{N}}$ is convergent to $0$ in the norm of $\ell_1$. However, the latter assertion contradicts the fact that $\sum_{j=1}^\infty \left | f_k\left(z^k_j\right)\right | > \varepsilon$ for every $k\in \mathbb{N}$.
\end{proof}

The previous theorem is still valid if we change almost convergence for weak almost convergence. Also, the previous theorem generalizes Theorem \ref{thSw} and some other results in \cite{AizGPPE, AizGuPer, AizPer,   BuWu} for sequences in $X\left(c_0\right)$ by means of (weak) almost convergence summability methods.

\subsection{Boolean-algebra almost summability}

In the next theorem we obtain a sufficient condition for the convergence of sequences in $X\left(c_0\right)$ by means of the pointwise convergence of almost sums in a natural Boolean algebra.

\begin{theorem}\label{teor2}
Consider a Banach space $X$, $\left(x^n\right)_{n\in \mathbb{N}}\subset X\left(c_0\right)$, and a natural Boolean algebra $\mathcal{F}$ with the Vitali-Hahn-Saks property. If $ \lim_{n\to \infty} \wACS_{i \in B} x^n_i$ exists for each $B \in \mathcal{F}$, then then there is $x^0 \in X\left(c_0\right)$ such that $\lim_{n\to \infty}\left \|x^n-x^0\right\|=0$ in $X\left(c_0\right)$.
\end{theorem}

\begin{proof}
 Let $T$ be the Stone space of $\mathcal{F}$. In accordance with Subsection \ref{vhs} we have that $\mathcal{C}\left(T\right)$ is Grothendieck and $\mathcal{C}_0\left(T\right)$ is barreled, so we can identify $\mathcal{C}\left(T\right)$ linearly and isometrically with a closed subspace $\mathcal{S}$ of $\ell_\infty$ containing $c_0$. Obviously, $\mathcal{S}$ will be Grothendieck. For each natural $n$, we define the map
\begin{equation*}
\begin{array}{rrcl}
\sigma_n:& \mathcal{S}&\to& X\\
&a=\left(a_i\right)_{i\in \mathbb{N}}&\mapsto& \sigma_n\left(a\right)=\wACS_{i=1}^{\infty} a_ix_i,
\end{array}
\end{equation*}
and will denote by $\sigma_n^0$ to the corresponding restriction of $\sigma_n$ to $\mathcal{S}_0$, where $\mathcal{S}_0$ stands for the subspace of $\mathcal{S}$ composed of all finite-valued sequences. Now, if $b\in \mathcal{S}_0$, then $\lim_{n\to \infty} \wACS_{i=1}^{\infty} b_ix_i^n$ exists. However, $\mathcal{S}_0$ is barreled (because $\mathcal{S}_0$ corresponds to $\mathcal{C}_0\left(T\right)$ in the previous identification), so there exists $H>0$ such that $\left\|\sigma_n\right\|=\left\|\sigma_n^0\right\|<H$ for all $n \in \mathbb{N}$. Next, due to the density of $\mathcal{S}_0$ in $\mathcal{S}$, we deduce that $\lim_{n\to \infty} \wACS_{i=1}^{\infty} a_i x_i^n$ exists for all $a=\left(a_i\right)_{i\in \mathbb{N}} \in \mathcal{S}$. From the previous theorem, it follows that an $x^0 \in X\left(c_0\right)$ exists satisfying that $\displaystyle{\lim_{n\to \infty}\left\|x^n-x^0\right\|=0}$ in $X\left(c_0\right)$.
\end{proof}

\begin{corollary}\label{cor}
Consider a Banach space $X$ and an $\ell_\infty$-Grothendieck closed vector subspace $\mathcal{S}$ of $\ell_\infty$ containing $c_0$. If $\left(x^n\right)_{n\in \mathbb{N}}$ is a sequence in $X_{\wAC}\left(\mathcal{S}\right)$, then $\left(x^n\right)_{n\in \mathbb{N}}$ is convergent in $X_{\wAC}\left(\mathcal{S}\right)$ if and only if $ \lim_{n\to \infty} \wACS_{i=1}^{\infty} a_ix_i^n$ exists for each $\left(a_i\right)_{i\in \mathbb{N}} \in \mathcal{S}$.
\end{corollary}

\subsection{The space $X^*_{\wsAC}({\cal S})$}

In a similar way, multiplier spaces of $w^*$-almost summable sequences can be defined.

\begin{definition}
If $X$ is a Banach space and $\mathcal{S}$ is a closed subspace of $\ell_\infty$, then the ${\cal S}$-multiplier space of $w^*$-almost summable sequences of $X^*$ is defined as: $$X^*_{\wsAC}\left(\mathcal{S}\right):=\left\{(x^*_i)_{i\in\mathbb{N}} \in \left(X^*\right)^\N:\wsACS a_i x^*_i\text{ exists}\right\}.$$
\end{definition}

\backmatter
\appendix
\chapter{$G$-spaces}\label{g}

Recall that a monoid is a set endowed with a binary operation which is associative and has an identity element. When a monoid is additive we usually call the identity element the neutral element, and when it is multiplicative the identity element is usually called the unity element.

Given a monoid $G$ and a non-empty set $M$, the expression that $M$ be a $G$-set or a $G$-space refers to a (left) action $G\curvearrowright M$, that is, a map $G\times M\to M$ which is:
\begin{itemize}
\item associative ($g\left(hm\right)=\left(gh\right)m$ for all $g,h\in H$ and all $m\in M$) and
\item verifies the identity condition ($em=m$ for all $m\in M$ where $e$ is the indentity of $G$).
\end{itemize}

A morphism between two $G$-sets, $_GM$ and $_GN$, is simply a $G$-homogeneous map $f:M\to N$, that is, $f\left(gm\right)=gf\left(n\right)$ for all $g\in G$ and all $m\in M$. A morphism of $G$-sets which is bijective verifies that its inverse is also a morphism of $G$-sets. A $G$-subset of a $G$-set is simply a subset which is closed under the action. Morphisms of $G$-sets preserve both images and pre-images of $G$-subsets. We refer the reader to \cite{ZIM} for wider perspective on the category of $G$-sets.

Given a $G$-set, the non-empty intersection of a family of $G$-subsets is another $G$-subset, and thus the generated $G$-subset by a non-empty subset can be defined as the intersection of all those $G$-subsets containing that given subset.

Notable subsets of a $G$-set $M$ are the orbits $Gm$ of each element $m\in M$. Notable submonoids of $G$ are the stabilizers $G_m:=\{g\in G:gm=m\}$ of each element $m\in M$.

Recall that a left action $G\curvearrowright M$ is said to be:
\begin{itemize}
\item {\em trivial} if each stabilizer is the whole monoid $G$ or equivalently all the orbits are trivial, that is, $Gm=\{m\}$.
\item {\em transitive} if the orbit of each element is the given set $M$;
\item {\em faithful} if the following morphism of groups is a monomorphism (that is, injective):
\begin{equation*}
\begin{array}{rcl}
G & \to & M!\\
g &\mapsto & \begin{array}{rcl}
M & \to & M\\
m & \mapsto & gm,
\end{array}
\end{array}
\end{equation*}
where $M!$ denotes the symmetric group of $M$, that is, the group of permutations of $M$;
\item {\em free} if the map
\begin{equation*}
\begin{array}{rcl}
G & \to & M\\
g &\mapsto & gm
\end{array}
\end{equation*}
is injective for every $m\in M$.
\end{itemize}

The reader may notice that when $G$ is a group, then an action is free if and only if each stabilizer if the trivial group, that is, $G_m=\{e\}$.

Every $G$-space can be endowed with a pre-order as follows: $m\leq n$ if and only if $m\in Gn$. If $G$ is a group, this pre-order is also an equivalence relation whose quotient set is the partition of $M$ into the orbits.

Throughout the whole of this appendix, $G$ stands for a monoid acting from the left on a non-empty set $M$. Sometimes $G$ will be a group, situation that will be explicitly stated.

\section{$G$-density}

The $G$-density basically consists of finding, given a (left) action, the smallest subsets which generate the whole set by orbiting through the monoid.

\subsection{$G$-density and $G$-generator sets}

We will immediately proceed to define and characterize the $G$-density. This definition is purely algebraical, however it somehow dyes as a topological definition.

\begin{definition}
A subset $N$ of a $G$-set $M$ is $G$-dense in $M$ provided that $N\cap Gm \neq \varnothing$ for all $m\in M$.
\end{definition}

The first fact we would like to point out about the $G$-density is that in order to check that a certain non-empty subset is $G$-dense it is not necessary to go through all the orbits.

\begin{lemma}\label{esoeh}
A subset $N$ of $M$ is $G$-dense in $M$ if and only if $N\cap Gm \neq \varnothing$ for all $m\in M\setminus N$.
\end{lemma}

\begin{proof}
Indeed, $\left\{n\right\}\subseteq N\cap Gn$ for all $n\in N$ so if $N\cap Gm \neq \varnothing$ for all $m\in M\setminus N$, then we trivially deduce that $N$ is $G$-dense in $M$.
\end{proof}

Next, we will characterize the $G$-density when $G$ is a group. However, we would first like to state and prove the following proposition which in some sense settles what type of characterization we will be making use of.

\begin{proposition}
If $N$ is a non-empty subset of $M$, then:
\begin{enumerate}
\item The $G$-subset generated by $N$ is $GN$.
\item $N$ is a $G$-generator of $M$ (that is, $GN=M$) if and only if $M\setminus N\subseteq GN$.
\end{enumerate}
\end{proposition}

\begin{proof}
\mbox{}
\begin{enumerate}
\item In the first place notice that $GN$ is a $G$-subset of $M$ which contains $N$. Therefore by definition we have that $GN$ contains the $G$-subset of $M$ generated by $N$. Now let any $gn\in GN$ with $g\in G$ and $n\in N$. Consider any $G$-subset $P$ of $M$ containing $N$. Since $n\in P$ we obtain that $gn\in P$ and so $GN\subseteq P$. This shows that $GN$ is contained in the $G$-subset generated by $N$.
\item Simply observe that if $M\setminus N\subseteq GN$, then $M=N\cup M\setminus N \subseteq GN\subseteq M$.
\end{enumerate}
\end{proof}

The $G$-density can be characterized in terms of generated $G$-subsets provided that $G$ is a group.

\begin{theorem}\label{dense}
If $G$ is a group, then the following conditions are equivalent:
\begin{enumerate}
\item $N$ is $G$-dense in $M$.
\item $N$ is a $G$-generator of $M$, that is, $GN=M$.
\end{enumerate}
\end{theorem}

\begin{proof}
\mbox{}
\begin{enumerate}
\item[{\em 1} $\Rightarrow$ {\em 2}] Let $m\in M$. By hypothesis we can find $n\in N\cap Gm$, so let $g\in G$ such that $n=gm$. Finally, $m= g^{-1}n\in GN$.
\item[{\em 2} $\Rightarrow$ {\em 1}] Let $m\in M\setminus N$. By hypothesis we can find $n\in N$ and $h\in G$ such that $hn=m$. Finally, $h^{-1}m=n\in N$. Now it only suffices to apply Lemma \ref{esoeh}.
\end{enumerate}
\end{proof}

The next example shows that Theorem \ref{dense} does not remain true for monoids.

\begin{example}\label{con}
Consider the additive monoid $G:=\left[0,+\infty\right)$ acting from the left on itself, $M:=\left[0,+\infty\right)$, by right translation:
\begin{equation*}
\begin{array}{rcl}
\left[0,+\infty\right)\times \left[0,+\infty\right)&\to& \left[0,+\infty\right)\\
\left(t,x\right)&\mapsto&t+x
\end{array}
\end{equation*}
We have the following:
\begin{itemize}
\item The orbit of any element $x\in\left[0,+\infty\right)$ is $Gx=\left[x,\infty\right)$.
\item A non-empty subset $A$ of $\left[0,+\infty\right)$ is $G$-dense if and only if $\sup\left(A\right)=+\infty$.
\item As a consequence, there is no minimal $G$-dense subset.
\item If $N$ is a non-empty subset of $\left[0,+\infty\right)$, then $$GN=\bigcup_{n\in N}Gn=\bigcup_{n\in N}\left[n,+\infty\right)=\left\{
\begin{array}{lll}
\left(\inf\left(N\right),+\infty\right)&\text{ if }&\inf\left(N\right)\notin N\\
\left[\inf\left(N\right),+\infty\right)&\text{ if }&\inf\left(N\right)\in N
\end{array}
\right.
$$
\item As a consequence, a non-empty subset $N$ of $\left[0,+\infty\right)$ is a $G$-generator if and only if $0\in N$.
\end{itemize}
\end{example}

The last result in this subsection shows that $G$-dense subsets are preserved under morphisms of $G$-sets.

\begin{proposition}
Assume that $G$ is a group and $f: \mbox{}_GM\to \mbox{}_GN$ is a morphism of $G$-sets.
\begin{enumerate}
\item If $P$ is a $G$-dense subset of $N$, then $f^{-1}\left(P\right)$ is a $G$-dense subset of $M$.
\item If $f$ is surjective and $P$ is a $G$-dense subset of $M$, then $f\left(P\right)$ is a $G$-dense subset of $N$.
\end{enumerate}
\end{proposition}

\begin{proof}
\mbox{}
\begin{enumerate}
\item Let $m\in M$. By hypothesis there are $g\in G$ and $p\in P$ such that $gp=f\left(m\right)$. Notice that $p=g^{-1}f\left(m\right)=f\left(g^{-1}m\right)$ and hence $g^{-1}m\in f^{-1}\left(P\right)$. Finally, $m=g\left(g^{-1}m\right)\in Gf^{-1}\left(P\right)$.
\item Let $n\in N$. By hypothesis there exists $m\in M$ such that $f\left(m\right)=n$. We can also find $p\in P$ and $g\in G$ such that $gp=m$ and hence $n=f\left(m\right)=f\left(gp\right)=gf\left(p\right)\in Gf\left(P\right)$.
\end{enumerate}
\end{proof}

To finish this subsection we would like to single out that any subset containing a $G$-dense subset or a $G$-generator is trivially $G$-dense or a $G$-generator, respectively (for this to hold it is not necessary that $G$ be a group).

\subsection{$G$-free sets}

This subsection is devoted to define and study the $G$-free sets. Later on we will relate this type of sets with the so called ``$G$-fundamental domains''.

\begin{definition}
A non-empty subset $N$ of $M$ is said to be $G$-free provided that $Gn\cap N=\left\{n\right\}$ for all $n\in N$.
\end{definition}

At this stage we consider crucial to single out the following two basic facts on $G$-free sets.

\begin{itemize}
\item If $G$ is a group, then $N$ is $G$-free if and only if for all $g_1,g_2\in G$ and all $n_1,n_2\in N$, the condition $g_1n_1=g_2n_2$ implies that $n_1=n_2$. Notice that the fact that $g_1=g_2$ is not necessarily implied, contrary to what happens with linearly independent sets in vector spaces, unless the action is free of course. These issues will be treated and compared against in Subsection \ref{mod} and Subsection \ref{vsp}.
\item If $N$ is $G$-free, then all of its subsets are also $G$-free (for this to hold it is not necessary that $G$ be a group).
\end{itemize}

\begin{proposition}
Let $f: \mbox{}_GM\to \mbox{}_GN$ be a monomorphism of $G$-sets.
\begin{enumerate}
\item If $P$ is a $G$-free subset of $N$, then $f^{-1}\left(P\right)$ is a $G$-free subset of $M$.
\item If $P$ is a $G$-free subset of $M$, then $f\left(P\right)$ is a $G$-free subset of $N$.
\end{enumerate}
\end{proposition}

\begin{proof}
\mbox{}
\begin{enumerate}
\item Let $m_1,m_2\in f^{-1}\left(P\right)$ and $g\in G$ such that $gm_1=m_2$. Then $gf\left(m_1\right)=f\left(gm_1\right)=f\left(m_2\right)$. By hypothesis $f\left(m_2\right)=f\left(m_1\right)$ and hence $m_1=m_2$.
\item Let $p,q\in P$ and $g\in G$ such that $gf\left(p\right)=f\left(q\right)$. Then $f\left(gp\right)=f\left(q\right)$ therefore $gp=q$ by the injectivity of $f$. Finally, $q\in Gp\cap P=\left\{p\right\}$ by hypothesis and thus $f\left(q\right)=f\left(p\right)$.
\end{enumerate}
\end{proof}

Finally, $G$-free sets are characterized by the following extension property when $G$ is a group and the action is free.

\begin{proposition}\label{extensionp}
If $G$ is a group acting freely from the left on $M$ and $F$ is a non-empty subset of $M$, then the following conditions are equivalent:
\begin{enumerate}
\item $F$ is $G$-free.
\item If $_GN$ is any $G$-set and $\alpha :F\to N$ is any map, then there exists a map $\beta :GF\to N$ which is unique verifying that $\beta|_F=\alpha$ and $\beta$ is a morphism of $G$-sets.
\item If $_GN$ is any $G$-set and $\alpha :F\to N$ is any map, then there exists a map $\beta :GF\to N$ verifying that $\beta|_F=\alpha$ and $\beta$ is a morphism of $G$-sets.
\end{enumerate}
\end{proposition}

\begin{proof}
\mbox{}
\begin{enumerate}
\item[{\em 1} $\Rightarrow$ {\em 2}] It suffices to define $\beta\left(gf\right):=g\alpha\left(f\right)$ for all $g\in G$ and all $f\in F$.
\item[{\em 2} $\Rightarrow$ {\em 3}] Fairly obvious.
\item[{\em 3} $\Rightarrow$ {\em 1}] Assume that $F$ is not $G$-free. In this case we may find $g\in G$ and $f_1\neq f_2\in F$ such that $gf_1=f_2$. Now it is easy to find a $G$-set $_GN$ with more than one element and a map $\alpha:F\to N$ such that $\alpha\left(f_2\right)\neq g\alpha \left(f_1\right)$. By hypothesis there exists a map $\beta :GF\to N$ verifying that $\beta|_F=\alpha$ and $\beta$ is a morphism of $G$-sets. Then $$\alpha\left(f_2\right)=\beta\left(f_2\right)=\beta\left(gf_1\right)=g\beta\left(f_1\right)=g\alpha\left(f_1\right)$$ which constitutes a contradiction.
\end{enumerate}
\end{proof}

The reader may notice that in the previous result for the implication {\em 3} $\Rightarrow$ {\em 1} to remain true it is not necessary that $G$ be a group or the action be free.

\subsection{$G$-bases and $G$-fundamental domains}

First we will characterize the minimal $G$-dense sets.

\begin{lemma}\label{minimal}
If $G$ is a group and $N$ is a $G$-dense subset of $M$, then the following conditions are equivalent:
\begin{enumerate}
\item $N$ is minimal among the $G$-dense subsets of $M$.
\item $N$ is $G$-free.
\end{enumerate}
\end{lemma}

\begin{proof}
\mbox{}
\begin{enumerate}
\item[{\em 1} $\Rightarrow$ {\em 2}] Assume that there exist $n\in N$ and $g\in G$ such that $gn\in N\setminus\left\{n\right\}$. In this case note that $N\setminus\left\{gn\right\}$ is $G$-dense.
\item[{\em 2} $\Rightarrow$ {\em 1}] Let $P\subseteq N$ be such that $P$ is $G$-dense in $M$. Let $n\in N$. By hypothesis $Gn \cap N =\left\{n\right\}$. By assumption there exist $g\in G$ and $p\in P$ such that $gp=n$. Now $p=g^{-1}n\in Gn\cap P\subseteq Gn\cap N=\left\{n\right\}$. This means that $n=p\in P$ and hence $N=P$.
\end{enumerate}
\end{proof}

The previous lemma can be versioned for maximal $G$-free sets.

\begin{lemma}\label{maximal}
If $G$ is a group and $N$ is a $G$-free subset of $M$, then the following conditions are equivalent:
\begin{enumerate}
\item $N$ is maximal among the $G$-free subsets of $M$.
\item $N$ is $G$-dense.
\end{enumerate}
\end{lemma}

\begin{proof}
\mbox{}
\begin{enumerate}
\item[{\em 1} $\Rightarrow$ {\em 2}] Let $m\in M\setminus GN$. Notice that $N\cup\left\{m\right\}$ is $G$-free and strictly contains $N$.
\item[{\em 2} $\Rightarrow$ {\em 1}] Let $P\supseteq N$ be such that $P$ is $G$-free in $M$. Let $p\in P$. By hypothesis there are $n\in N$ and $g\in G$ such that $p=gn$. Now $n=g^{-1}p\in Gp \cap P =\left\{p\right\}$, which means that $p=n\in N$.
\end{enumerate}
\end{proof}

Both Lemma \ref{minimal} and Lemma \ref{maximal} motivate the following expected definition.

\begin{definition}
A subset $B$ of $M$ is said to be a $G$-basis of $M$ provided that $B$ is $G$-free, $G$-dense, and a $G$-generator of $M$.
\end{definition}

Obviously, if $G$ is a group, then a $G$-basis is a $G$-free, $G$-dense subset in virtue of Theorem \ref{dense}.

We recall the reader that, given a (left) action of a group $G$ on a non-empty set $M$, a $G$-fundamental domain of $M$ is a subset of $M$ which contains exactly one element from each orbit. It is not difficult to show that the $G$-fundamental domains are exactly the $G$-bases (when $G$ is a group). The following result shows the existence of $G$-bases. We will strongly rely on Lemma \ref{maximal}.

\begin{theorem}\label{pollaenvinagre}
If $G$ is a group, then:
\begin{enumerate}
\item Every $G$-free subset $F$ of $M$ is contained in a $G$-basis $B$ of $M$.
\item Every $G$-dense subset $D$ of $M$ contains a $G$-basis $B$ of $M$.
\item If $B$ and $C$ are $G$-bases of $M$, then $\mathrm{card}\left(B\right)=\mathrm{card}\left(C\right)$.
\end{enumerate}
\end{theorem}

\begin{proof}
\mbox{}
\begin{enumerate}
\item Consider the non-empty set $\mathcal{L}:=\left\{P\subseteq M: F\subseteq P\text{ and }P\text{ is }G\text{-free} \right\}$ partially ordered by the inclusion. We will show that $\mathcal{L}$ is an inductive set. Let $\left(P_i\right)_{i\in I}$ be a chain in $\mathcal{L}$. Notice that $\bigcup_{i\in I}P_i \in \mathcal{L}$. Indeed, it is obvious that $F\subseteq \bigcup_{i\in I}P_i$. Now let $p\in \bigcup_{i\in I}P_i$ and $g\in G$ such that $gp\in \bigcup_{i\in I}P_i$. Because of the total order there must exist $i\in I$ such that $gp,p\in P_i$. Therefore $gp=p$. In accordance to the Zorn's Lemma there exists a maximal element in $\mathcal{L}$ which we will denote by $B$. In virtue of Lemma \ref{maximal} we have that $N$ is a $G$-basis of $M$.
\item Consider the non-empty set $\mathcal{L}:=\left\{P\subseteq M: D\supseteq P\text{ and }P\text{ is }G\text{-free} \right\}$ partially ordered by the inclusion. We will show that $\mathcal{L}$ is an inductive set. Let $\left(P_i\right)_{i\in I}$ be a chain in $\mathcal{L}$. Notice that $\bigcup_{i\in I}P_i \in \mathcal{L}$. Indeed, it is obvious that $D\supseteq \bigcup_{i\in I}P_i$. Now let $p\in \bigcup_{i\in I}P_i$ and $g\in G$ such that $gp\in \bigcup_{i\in I}P_i$. Because of the total order there must exist $i\in I$ such that $gp,p\in P_i$. Therefore $gp=p$. In accordance to the Zorn's Lemma there exists a maximal element in $\mathcal{L}$ which we will denote by $B$. In order to show that $B$ is a $G$-basis of $M$ it suffices to proof that $B$ is $G$-dense in $M$, for which it is enough to show that $D\subseteq GB$. Suppose to the contrary that there exists $d\in D\setminus GB$. Notice that $B\cup \left\{d\right\}$ is $G$-free, contained in $D$ and strictly contains $B$, which contradicts the fact that $B$ is maximal.
\item By hypothesis for every $c\in C$ there exists a unique element $b_c\in B$ such that $c\in Gb_c$. The Axiom of Choice allows us to consider the set $\left\{b_c\in B:c\in C\right\}$ and to define the map
\begin{equation}\label{surje}
\begin{array}{rcl}
C&\to&B\\
c&\mapsto&b_c
\end{array}
\end{equation}
It is not hard to see that the previous map is injective (recall a previous observation about free sets). In order to see that it is surjective it only suffices to realize that the set $\left\{b_c:c\in C\right\}$ is a $G$-basis of $M$.
\end{enumerate}
\end{proof}

Finally, following the idea behind Proposition \ref{extensionp} we find that $G$-bases can also be characterized by the following extension property when $G$ is a group and the action is free.

\begin{proposition}\label{extensionpo}
If $G$ is a group acting freely from the left on $M$ and $B$ is a non-empty subset of $M$, then the following conditions are equivalent:
\begin{enumerate}
\item $B$ is $G$-basis.
\item If $_GN$ is any $G$-set and $\alpha :B\to N$ is any map, then there exists a map $\beta :M\to N$ which is unique verifying that $\beta|_B=\alpha$ and $\beta$ is a morphism of $G$-sets.
\end{enumerate}
\end{proposition}

\begin{proof}
\mbox{}
\begin{enumerate}
\item[{\em 1} $\Rightarrow$ {\em 2}] Inmediate in virtue of Proposition \ref{extensionp}.
\item[{\em 2} $\Rightarrow$ {\em 1}] Obviously, $B$ is $G$-free in accordance to Proposition \ref{extensionp}. Suppose to the contrary that $GB\subsetneq M$. By applying (1) of Theorem \ref{pollaenvinagre} there exists a $G$-basis $V$ of $M$ such that $B\subsetneq V$. Now consider $_GN$ to be any $G$-set with more than one element and $\alpha: B\to N$ any map. Since $B\subsetneq V$ and $N$ has more than one element, we can find two different extensions $\alpha_1 \neq \alpha_2 : V\to N$ of $\alpha$. Now Proposition \ref{extensionp} assures the existence of $\beta_1\neq\beta_2: M\to N$ which are morphisms of $G$-sets and extensions of $\alpha_1$ and $\alpha_2$, repesctively. This is a contradiction.
\end{enumerate}
\end{proof}

The reader may notice that in the previous result for the implication [{\em 2} $\Rightarrow$ {\em 1}] to remain true it is not necessary that the action be free.

\subsection{$G$-density character and $G$-dimension}

This subsection is devoted to define and study the $G$-density character, which is nothing else but an index to indicate the smallest size of the $G$-dense subsets. In virtue of Theorem \ref{pollaenvinagre} we feel motivated to define both the $G$-density character and the $G$-dimension in the following way.

\begin{definition}
\mbox{}
\begin{itemize}
\item The $G$-density character of $M$ is defined as the smallest cardinal that is $G$-dense in $M$ (if it exists) and it is denoted by $\mathrm{dchar}\left(_GM\right)$.
\item The $G$-dimension of $M$ is defined as the smallest cardinal that is a $G$-generator of $M$ (if it exists) and it is denoted by $\mathrm{dim}\left(_GM\right)$.
\end{itemize}
\end{definition}

Unfortunately there is not any type of order relation between the $G$-density character and the $G$-dimension, as shown in the following example (we recall the reader about Example \ref{con} and refer him or her to Subsection \ref{mod} and Subsection \ref{vsp}).

\begin{example}\label{con2c}
\mbox{}
\begin{itemize}
\item Consider the additive monoid $G:=\left[0,+\infty\right)$ acting from the left on itself, $M:=\left[0,+\infty\right)$, by right translation:
\begin{equation*}
\begin{array}{rcl}
\left[0,+\infty\right)\times \left[0,+\infty\right)&\to& \left[0,+\infty\right)\\
\left(t,x\right)&\mapsto&t+x
\end{array}
\end{equation*}
In accordance to what was revealed in Example \ref{con} we have that
$$1=\dim\left(_GM\right)<\aleph_0=\mathrm{dchar}\left(_GM\right).$$
\item Consider $\mathbb{R}^2$ as a real vector space. From (3) of Proposition \ref{elcuerpo} and from Theorem \ref{qr} one can infer that $$1=\mathrm{dchar}\left(_\mathbb{R}\mathbb{R}^2\right)<\aleph_1=\dim\left(_\mathbb{R}\mathbb{R}^2\right).$$
\end{itemize}
\end{example}

We will finish this subsection with an immediate corollary of Theorem \ref{pollaenvinagre}.

\begin{corollary}
If $G$ is a group, then the $G$-density character of $M$ coincides with the $G$-dimension of $M$, that is, the cardinal of any $G$-basis of $M$.
\end{corollary}

\subsection{Extreme cases of $G$-density}

This subsection is on two results that characterize the cases of maximum and minimum $G$-density. We will begin with the situation of minimum $G$-density. We recall the reader that a left action of a monoid on a given non-empty set is said to be transitive provided that all the orbits equal the whole set.

\begin{theorem}
The following conditions are equivalent:
\begin{enumerate}
\item The action is transitive.
\item Every non-empty subset of $M$ is a $G$-generator.
\item Every non-empty subset of $M$ is $G$-dense.
\end{enumerate}
In this situation both the $G$-density character and the $G$-dimension of $M$ are equal to $1$.
\end{theorem}

\begin{proof}
\mbox{}
\begin{enumerate}
\item[{\em 1} $\Rightarrow$ {\em 2}] Let $N$ be a non-empty subset of $M$. If $n\in N$, then by definition we have that $M=Gn\subseteq GN\subset M.$
\item[{\em 2} $\Rightarrow$ {\em 3}] Let $N$ be a non-empty subset of $M$ and consider any $m\in M$. Since $\left\{m\right\}$ is a $G$-generator of $M$ by hypothesis, then we have that $Gn=M$ and so $N\cap Gn=N \cap M = N\neq \varnothing$.
\item[{\em 3} $\Rightarrow$ {\em 1}] Fix an arbitrary $m\in M$. Consider any other $n\in M$. By hypothesis $\left\{n\right\}$ is $G$-dense in $M$, therefore $\left\{n\right\}\cap Gm\neq \varnothing$, which meas that $n\in GM$. As a consequence, $Gm=M$ and the action is transitive by the arbitrariness of $m$.
\end{enumerate}
\end{proof}

In case of group actions the previous theorem can be improved by settling that the condition that the $G$-density equal $1$ is equivalent to the previous two ones.

\begin{corollary}
If $G$ is a group, then the following conditions are equivalent:
\begin{enumerate}
\item The action is transitive.
\item The $G$-density character of $M$ is $1$.
\item The $G$-dimension of $M$ is $1$.
\end{enumerate}
\end{corollary}

\begin{proof}
In the first place, notice that conditions {\em 2} and {\em 3} above are exactly the same in virtue of Theorem \ref{dense}. Now, if the $G$-dimension of $M$ is $1$, then there exists an orbit which equals the whole set $M$. Since $G$ is a group we immediately deduce that the action is transitive.
\end{proof}

Now we will take care of the case of maximum $G$-density. We remind the reader that an action is trivial provided that $gm=m$ for all $g\in G$ and all $m\in M$.

\begin{theorem}
The following conditions are equivalent:
\begin{enumerate}
\item The action is trivial.
\item The only $G$-dense subset of $M$ is $M$.
\item The only $G$-generator subset of $M$ is $M$.
\end{enumerate}
In this situation the $G$-density character of $M$ is $\mathrm{card}\left(M\right)$.
\end{theorem}

\begin{proof}
\mbox{}
\begin{enumerate}
\item[{\em 1} $\Rightarrow$ {\em 2}] Let $N$ be a $G$-dense subset of $M$. Let $m\in M$. By hypothesis we have that $N \cap Gm\neq \varnothing$, so there are $g\in G$ and $n\in N$ such that $n=gm=m$. As a consequence, $N=M$.
\item[{\em 2} $\Rightarrow$ {\em 3}] Let $N$ be a $G$-generator subset of $M$. Fix arbitrary elements $g\in G$ and $n\in N$ and suppose that $gn\neq n$. Then $M\setminus \left\{n\right\}\subsetneq M$ and $M\setminus \left\{n\right\}$ is $G$-dense in $M$. This contradicts the hypothesis, therefore it must happen that $gn=n$. Since this is for every $g\in G$ and every $n\in N$, the fact that $N$ is a $G$-generator of $M$ automatically implies that $N=M$.
\item[{\em 3} $\Rightarrow$ {\em 1}] Let $g\in G$ and $m\in M$ and suppose to the contrary that $gm\neq m$. Then $M\setminus \left\{gm\right\}\subsetneq M$ and $G\left(M\setminus \left\{gm\right\}\right)=M$, which is a contradiction.
\end{enumerate}
\end{proof}

\section{Applications of $G$-density}

In this final section we will see two applications of $G$-density. Due to the fact that the $G$-density coincides with the $G$-generation whenever $G$ is a group, we will make a study of this topic in modules and vector spaces and compare it to the linear independence.

\subsection{$G$-density in modules}\label{mod}

This subsection is meant to compare the linear independence to the $G$-freeness. We will begin with the following crucial lemma.

\begin{lemma}\label{bixi}
Consider a submonoid $H$ of $G$. Then:
\begin{enumerate}
\item If $N\subset M$ is $G$-free in $M$, then $N$ is $H$-free in $M$.
\item If $N\subset M$ is an $H$-generator of $M$, then $N$ is a $G$-generator of $M$.
\item If both $\dim\left(_HM\right)$ and $\dim\left(_GM\right)$ exist, then $\dim\left(_GM\right)\leq \dim\left(_HM\right)$.
\item If $N\subset M$ is $H$-dense in $M$, then $N$ is $G$-dense in $M$.
\item If both $\mathrm{dchar}\left(_HM\right)$ and $\mathrm{dchar}\left(_GM\right)$ exist, then $\mathrm{dchar}\left(_GM\right)\leq \mathrm{dchar}\left(_HM\right)$.
\end{enumerate}
\end{lemma}

\begin{proof}
\mbox{}
\begin{enumerate}
\item Simply notice that if $n \in N$, then $\left\{n\right\}\subseteq Hn\cap N\subseteq Gn \cap N=\left\{n\right\}$.
\item Observe that $M=HN\subseteq GN\subseteq M$.
\item Let $N\subseteq M$ such that $\mathrm{card}\left(N\right)=\dim\left(_HM\right)$ and $N$ is an $H$-generator of $M$. By taking into account the previous paragraph we have that $N$ is a $G$-generator of $M$ and so $\dim\left(_GM\right)\leq \mathrm{card}\left(N\right)=\dim\left(_HM\right)$.
\item Let $m\in M$. Notice that $\varnothing \neq N \cap Hm \subseteq N\cap Gm$.
\item Let $N\subseteq M$ such that $\mathrm{card}\left(N\right)=\mathrm{dchar}\left(_HM\right)$ and $N$ is $H$-dense in $M$. By taking into account the previous paragraph we have that $N$ is $G$-dense in $M$ and so $\mathrm{dchar}\left(_GM\right)\leq \mathrm{card}\left(N\right)=\mathrm{dchar}\left(_HM\right)$.
\end{enumerate}
\end{proof}

Given an associative ring with unity $R$ we will denote by $\mathcal{U}\left(R\right)$ to the multiplicative group of all invertible elements of $R$. Recall that, since $\mathcal{U}\left(R\right)$ is a group, the concepts of $\mathcal{U}\left(R\right)$-dense and $\mathcal{U}\left(R\right)$-generator coincide in virtue of Theorem \ref{dense}. The previous lemma together with Theorem \ref{pollaenvinagre} gives us the following result.

\begin{proposition}\label{elcuerpo}
Let $R$ be a non-zero associative ring with unity and consider a left $R$-module $_RM$. Then we have the following:
\begin{enumerate}
\item If $F$ is an $R$-linearly independent subset of $M$, then $F$ is both $R$-free and $\mathcal{U}\left(R\right)$-free.
\item If $F$ is an $R$-generator of $M$ or a $\mathcal{U}\left(R\right)$-generator of $M$, then the $R$-linear span of $F$ is $M$.
\item A subset $F$ of $M$ is $R$-dense if and only if $0\in F$. In particular, $\left\{0\right\}$ is $R$-dense and thus $\mathrm{dchar}\left(_RM\right)=1$.
\item If $M$ is free as a left $R$-module, then $\dim\left(_RM\right) \leq \mathrm{dchar}\left(_{\mathcal{U}\left(R\right)}M\right)$, where $\dim\left(_RM\right)$ denotes now the dimension of $_RM$ as a free left $R$-module.
\end{enumerate}
\end{proposition}

\begin{proof}
\mbox{}
\begin{enumerate}
\item Let $m_1,m_2\in F$ and $r\in R$ such that $rm_1=m_2$. If $m_1\neq m_2$, then $rm_1-m_2=0$ and this contradicts the fact that $F$ is a $R$-linearly independet subset of $M$.
\item Obvious since $RF=M$ by hypothesis.
\item Assume that $F$ is $R$-dense. Notice that $R0\cap F\neq \varnothing$, therefore $0\in F$. Conversely, $\left\{0\right\}$ is clearly $R$-dense because every orbit $Rm$, with $m\in M$, contains $0$.
\item Let $B\subseteq M$ a basis of $M$ as a free left $R$-module. By {\em 1} of this proposition we have that $B$ is $\mathcal{U}\left(R\right)$-free in $M$ and so $\dim\left(_RM\right)=\mathrm{card}\left(B\right)\leq\mathrm{dchar}\left(_{\mathcal{U}\left(R\right)}M\right)$ in accordance to Theorem \ref{pollaenvinagre}.
\end{enumerate}
\end{proof}

The previous proposition finds also its version for vector spaces in the following result.

\begin{proposition}
Let $K$ be a non-zero division ring and consider a left $K$-vector space $_KX$. Consider a non-empty subset $F\subseteq X$. We have the following:
\begin{enumerate}
\item If $F$ is not a singleton, then $F$ is $K$-free if and only if $0\notin F$ and $F$ is $\left(K\setminus\left\{0\right\}\right)$-free.
\item If $F\subseteq X\setminus\left\{0\right\}$ is $\left(K\setminus\left\{0\right\}\right)$-free, then $F\cup \left\{0\right\}$ is also $\left(K\setminus\left\{0\right\}\right)$-free.
\item If $0\notin F$, then $F$ is $\left(K\setminus\left\{0\right\}\right)$-dense in $X\setminus\left\{0\right\}$ if and only if $F$ is $K$-dense in $X\setminus\left\{0\right\}$.
\item $F$ is a $\left(K\setminus\left\{0\right\}\right)$-generator of $X$ if and only if $0\in F$ and $F$ is a $K$-generator of $X$.
\end{enumerate}
\end{proposition}

\begin{proof}
\mbox{}
\begin{enumerate}
\item Assume first that $F$ is $K$-free. By {\em 1} of Lemma \ref{bixi} we have that $F$ is $\left(K\setminus\left\{0\right\}\right)$-free. Now suppose to the contrary that $0\in F$. Since $F$ is not a singleton there must exist $x\in F\setminus\left\{0\right\}$. In this situation $Kx\cap F=\left\{x,0\right\}$ which contradicts the fact that $F$ is $K$-free.

Conversely, assume that $0\notin F$ and $F$ is $\left(K\setminus\left\{0\right\}\right)$-free. Let $x,y\in F$ and $k\in K$ such that $kx=y$. Notice that $k\neq 0$ since $y\in F\subseteq X\setminus\left\{0\right\}$, therefore by hypothesis $y\in \left(K\setminus\left\{0\right\}\right)x \cap F=\left\{x\right\}$.

\item Let $x,y\in F\cup \left\{0\right\}$ and $k\in K\setminus\left\{0\right\}$ such that $kx=y$. If $y=0$, then $x=0$ since $k\neq 0$. If $y\neq 0$, then $x\neq 0$, and thus the $\left(K\setminus\left\{0\right\}\right)$-freeness of $F$ allows that $x=y$.

\item In accordance to {\em 3} of Lemma \ref{bixi}, if $F$ is $\left(K\setminus\left\{0\right\}\right)$-dense in $X\setminus\left\{0\right\}$, then $F$ is $K$-dense in $X\setminus\left\{0\right\}$.

Conversely, assume that $F$ is $K$-dense in $X\setminus\left\{0\right\}$. Let $x\in X\setminus\left\{0\right\}$ and consider the orbit $\left(K\setminus\left\{0\right\}\right)x$. By hypothesis there exists an element $y\in F$ such that $y\in Kx\cap F$, which means that there is $k\in K$ such that $kx=y$. Notice that $y\neq 0$ since $0\notin F$ and thus $k\neq 0$. As a consequence, $y\in \left(K\setminus\left\{0\right\}\right)x\cap F$.

\item By taking into account {\em 2} of Lemma \ref{bixi}, if $F$ is a $\left(K\setminus\left\{0\right\}\right)$-generator of $X$, then $F$ is a $K$-generator of $X$. We will show now that $0\in F$. By hypothesis, there are $k\in K\setminus \left\{0\right\}$ and $x\in F$ such that $kx=0$. This implies that $x=0$.

Conversely, assume that $0\in F$ and $F$ is a $K$-generator of $X$. If $x\in X$, then by hypothesis we can find $k\in K$ and $y\in F$ such that $ky=x$. If $x\neq 0$, then $k\neq 0$ and we are done. If $x=0$, then we are also done because $0\in F$.

\end{enumerate}
\end{proof}

\subsection{$G$-density in normed linear spaces}\label{vsp}

In this subsection we will make use of norms and absolute values to compute density characters and find $G$-bases of vector spaces over division rings.

\begin{remark}
Assume $X$ is a non-zero normed linear space over a non-zero absolute-valued division ring $K$. Notice that $$\left\|\cdot\right\|\left(X\right) = \bigcup_{x\in X}\left|\cdot\right|\left(K\right)\left\|x\right\|.$$ As a consequence, the unit sphere of $X$, $\mathsf{S}_X:=\left\{x\in X: \left\|x\right\|=1\right\}$, is not empty if and only if $\left\|\cdot\right\|\left(X\right) \supseteq \left|\cdot\right|\left(K\right).$
\end{remark}

\begin{theorem}\label{yav}
Let $X$ be a non-zero normed linear space over the non-zero absolute-valued division ring $K$. Assume that $\mathsf{S}_X\neq \varnothing$. The following conditions are equivalent:
\begin{enumerate}
\item $\mathsf{S}_X$ is a $\left(K\setminus\left\{0\right\}\right)$-dense subset of $X\setminus\left\{0\right\}$.
\item $\left\|\cdot\right\|\left(X\right) = \left|\cdot\right|\left(K\right).$
\end{enumerate}
\end{theorem}

\begin{proof}
\mbox{}
\begin{itemize}
\item[{\em 1} $\Rightarrow$ {\em 2}] Notice that it is sufficient to show that $\left\|\cdot\right\|\left(X\right) \subseteq \left|\cdot\right|\left(K\right)$ since $\mathsf{S}_X\neq \varnothing$. Let $x\in X\setminus\left\{0\right\}$. By hypothesis there exist $k\in K$ and $y\in \mathsf{S}_X$ such that $ky=x$. Then $\left|k\right|=\left|k\right|\left\|y\right\|=\left\|ky\right\|=\left\|x\right\|.$
\item[{\em 2} $\Rightarrow$ {\em 1}] Let $x\in X\setminus\left\{0\right\}$. By hypothesis there exist $k\in K\setminus\left\{0\right\}$ such that $\left|k\right|=\left\|x\right\|$. Notice that $\left|k^{-1}\right|=\left|k\right|^{-1}=\left\|x\right\|^{-1}$, therefore $k^{-1}x\in \mathsf{S}_X$ and $k\left(k^{-1}x\right)=x$.
\end{itemize}
\end{proof}

Our purpose in this subsection is to explicitly construct a $\left(K\setminus\left\{0\right\}\right)$-basis for any non-zero left $K$-vector space over an arbitrary non-zero division ring $K$. For this we will strongly rely on Theorem \ref{yav}.

\begin{theorem}\label{qr}
Let $X$ be a non-zero normed linear space over the non-zero absolute-valued division ring $K$. Assume that both $\mathsf{S}_X$ and $\mathsf{S}_K$ are not empty. Consider the equivalence relation on $\mathsf{S}_X$ given by $$\mathcal{R}:=\left\{\left(x,y\right)\in\mathsf{S}_X\times \mathsf{S}_X:\text{ exists }k\in\mathsf{S}_K\text{ such that }x=ky\right\}.$$ If $\mathsf{S}_X$ is a $\left(K\setminus\left\{0\right\}\right)$-dense subset of $X\setminus\left\{0\right\}$, then the $\left(K\setminus\left\{0\right\}\right)$-density character of $X\setminus\left\{0\right\}$ is $$\mathrm{card}\left(\frac{\mathsf{S}_X}{\mathcal{R}}\right).$$
\end{theorem}

\begin{proof}
It suffices to consider, in virtue of the Axiom of Choice, a set $B$ composed exactly of only one element from each equivalent class in $\mathsf{S}_X/\mathcal{R}$ and to apply Lemma \ref{minimal} to come to the conclusion that $B$ is a $\left(K\setminus\left\{0\right\}\right)$-basis of $X\setminus\left\{0\right\}$.
\end{proof}

\begin{remark}\label{asinosvaostia}
Let $X$ be a non-zero vector space over a non-zero division ring $K$. Let $\left|\cdot\right|$ be a non-zero absolute value on $K$. Consider $\left(x_j\right)_{j\in J}$ to be a basis of $X$ as a $K$-vector space and define the following norm on $X$: $$\left\|k_1x_{j_1}+\cdots+k_mx_{j_m}\right\|:=\max\left\{\left|k_1\right|,\dots,\left|k_m\right|\right\}$$ where $k_1,\dots,k_m\in K$. It is clear that $\mathsf{S}_X\neq \varnothing$ and $\left\|\cdot\right\|\left(X\right) = \left|\cdot\right|\left(K\right)$. In case $\left|\cdot\right|\left(K\right)$ is a submonoid of the additive monoid $\left[0,\infty \right)$, then the norm on $X$ given by $$\left\|k_1x_{j_1}+\cdots+k_mx_{j_m}\right\|:=\left|k_1\right|+\cdots+\left|k_m\right|,$$ where $k_1,\dots,k_m\in K$, also verifies that $\mathsf{S}_X\neq \varnothing$ and $\left\|\cdot\right\|\left(X\right) = \left|\cdot\right|\left(K\right)$.
\end{remark}

By bearing in mind the previous remark all we need to focus on is finding a non-zero absolute value over any non-zero division ring.

\begin{theorem}
Let $X$ be a non-zero vector space over a non-zero division ring $K$. There exist an absolute value $\left|\cdot\right|$ on $K$ and a norm $\left\|\cdot\right\|$ on $X$ such that $\mathsf{S}_X\neq \varnothing$ and $\left\|\cdot\right\|\left(X\right) = \left|\cdot\right|\left(K\right).$ In this situation, if we consider the equivalence relation on $\mathsf{S}_X$ given by $$\mathcal{R}:=\left\{\left(x,y\right)\in\mathsf{S}_X\times \mathsf{S}_X:\text{ exists }k\in\mathsf{S}_K\text{ such that }x=ky\right\},$$ then the $\left(K\setminus\left\{0\right\}\right)$-density character of $X\setminus\left\{0\right\}$ is $$\mathrm{card}\left(\frac{\mathsf{S}_X}{\mathcal{R}}\right)$$ and a $\left(K\setminus\left\{0\right\}\right)$-basis of $X\setminus\left\{0\right\}$ can be obtained by considering, in virtue of the Axiom of Choice, a set $B$ composed exactly of only one element from each equivalent class in $\mathsf{S}_X/\mathcal{R}$.
\end{theorem}

\begin{proof}
It suffices to apply Theorem \ref{qr} after considering the following absolute value on $K$:
\begin{equation*}
\begin{array}{rrcl}
\left|\cdot\right|:&K&\to&\left\{0\right\}\\
&k&\mapsto&\left|k\right|:=\left\{\begin{array}{ll}
0&\text{ if }k=0\\
1&\text{ if }k\neq0
\end{array}\right.
\end{array}
\end{equation*}
\end{proof}

\bibliographystyle{chicago} 
\bibliography{bibliography}

%
%
%
%

%
%
%
%


\printindex

\end{document}